\documentclass[letterpaper, 12pt, oneside]{book}

    \usepackage[left=108pt, right=74pt, top=108pt, bottom=81pt, dvips, pdftex]{geometry}
    \usepackage{verbatim}
    \usepackage{amssymb}
    \usepackage{graphicx}
    \usepackage{graphics}
    \usepackage{amsfonts}
    \usepackage{calc}
    \usepackage{amscd}
    \usepackage{color}
    \definecolor{darkblue}{rgb}{0, 0, .4}
\definecolor{grey}{rgb}{.7, .7, .7}
\usepackage[breaklinks]{hyperref}
\hypersetup{
	colorlinks=true,
	linkcolor=darkblue,
	anchorcolor=darkblue,
	citecolor=darkblue,
	pagecolor=darkblue,
	urlcolor=darkblue,
	pdftitle={},
	pdfauthor={}
}
    \usepackage{epsfig}
    \usepackage{fancyhdr} 
    \renewcommand{\bibname}{References}

\renewcommand{\a}{\alpha}
\newcommand{\e}{\varepsilon}
    
    \usepackage{amsmath,amsthm, amsfonts,amssymb} 
    \usepackage{setspace} 
    
    \newtheorem{theorem}{Theorem}[subsection]
    \newtheorem{proposition}[theorem]{Proposition}
    \newtheorem{lemma}[theorem]{Lemma}
    \newtheorem{corollary}[theorem]{Corollary}
    \theoremstyle{definition}
        \newtheorem{definition}[theorem]{Definition}
        \newtheorem{example}[theorem]{Example}
        \newtheorem{conjecture}[theorem]{Conjecture}

    \theoremstyle{remark}
        \newtheorem{remark}[theorem]{Remark}
    
    \numberwithin{equation}{section}
    \numberwithin{figure}{section}
    
\newlength\cellsize 
\savebox2{%
\begin{picture}(18,18)
\put(0,0){\line(1,0){18}}
\put(0,0){\line(0,1){18}}
\put(18,0){\line(0,1){18}}
\put(0,18){\line(1,0){18}}
\end{picture}}
\newcommand\cellify[1]{\def\thearg{#1}\def\nothing{}%
\ifx\thearg\nothing
\vrule width0pt height\cellsize depth0pt\else
\hbox to 0pt{\usebox2\hss}\fi%
\vbox to 18\unitlength{
\vss
\hbox to 18\unitlength{\hss$#1$\hss}
\vss}}
\newcommand\tableau[1]{\vtop{\let\\=\cr
\setlength\baselineskip{-16000pt}
\setlength\lineskiplimit{16000pt}
\setlength\lineskip{0pt}
\halign{&\cellify{##}\cr#1\crcr}}}
\savebox3{%
\begin{picture}(15,15)
\put(0,0){\line(1,0){15}}
\put(0,0){\line(0,1){15}}
\put(15,0){\line(0,1){15}}
\put(0,15){\line(1,0){15}}
\end{picture}}
\newcommand\expath[1]{%
\hbox to 0pt{\usebox3\hss}%
\vbox to 15\unitlength{
\vss
\hbox to 15\unitlength{\hss$#1$\hss}
\vss}}
    
    \newcommand{\n}{\vspace{12pt}} 
      
    \newcommand{\newchapter}[3] 
	{                           
        \chapter[#2]{#3}
        \chaptermark{#1}
        \thispagestyle{myheadings}
	}

\begin{document}
    

    \pagenumbering{roman}
    \pagestyle{plain}

    %
    %

    \singlespacing

    ~\vspace{-0.75in} 
    \begin{center}

        \begin{huge}
            Combinatorics of $(\ell,0)$-JM partitions, $\ell$-cores, the ladder crystal and the finite Hecke algebra
        \end{huge}\\\n
        By\\\n
        {\sc Christopher J. Berg}\\
        B.S. (UC Santa Barbara) 2003\\
        M.A. (UC Davis) 2005\\\n
        DISSERTATION\\\n
        Submitted in partial satisfaction of the requirements for the degree of\\\n
        DOCTOR OF PHILOSOPHY\\\n
        in\\\n
        MATHEMATICS\\\n
        in the\\\n
        OFFICE OF GRADUATE STUDIES\\\n
        of the\\\n
        UNIVERSITY OF CALIFORNIA\\\n
        DAVIS\\\n\n
        
        Approved:\\\n\n
        
        \rule{4in}{1pt}\\
        ~Monica Vazirani (Chair)\\\n\n
        
        \rule{4in}{1pt}\\
        ~Anne Schilling\\\n\n
        
        \rule{4in}{1pt}\\
        ~Dmitry Fuchs\\
        
        \vfill
        
        Committee in Charge\\
        ~2009

    \end{center}

    \newpage

    %
    %
        



    %
    %
    
    \begin{center}\huge
    Dedicated to Sonya and Kai (a.k.a. $\chi$).
    \end{center}
    \newpage

    \doublespacing

    %
    %
    
    \tableofcontents
    
    \newpage


    \centerline{\large $(\ell,0)$-JM Partitions and $\ell$-cores}
    
    \centerline{\textbf{\underline{Abstract}}}

The following thesis contains results on the combinatorial representation theory of the finite Hecke algebra $H_n(q)$. 

In Chapter 2 simple combinatorial descriptions are given which
 determine when a Specht module corresponding to a partition $\lambda$ is irreducible. This is done
 by extending the results of James and Mathas. These descriptions depend on the 
crystal of the basic representation of the affine Lie algebra $\widehat{\mathfrak{sl}_\ell}$. 

In Chapter 3 these results are extended to determine which irreducible modules
 have a realization as a Specht module. To do this, a new condition of irreducibility due to
 Fayers is combined with a new description of the crystal from Chapter 2.

In Chapter 4 a bijection of cores first described by myself and 
Monica Vazirani is studied in more depth. 
Various descriptions of it are given, relating to the quotient $\widetilde{S_\ell}/{S_\ell}$ and to
 the bijection given by Lapointe and Morse.
    
    \newpage

    %
    %

    \chapter*{\vspace{-1.5in}Acknowledgments and Thanks}
    
\indent I would like to thank Steve Pon for his help in finding a proof of our crystal isomorphism. To Brant Jones I am very grateful for his help with Chapter \ref{sec:LabelForChapter4}, as well as saving my thesis by helping me repair a broken lemma in Chapter \ref{sec:LabelForChapter3}.

I would like to thank Dmitry Fuchs and Anne Schilling for reading this thesis.

Lastly, I would like to thank my advisor Monica Vazirani. She introduced me to the representation theory of the symmetric group, and certainly this work would not exist without her.

    \newpage

    %
    %

    \pagestyle{fancy}
    \pagenumbering{arabic}
       
    %
    %

    \newchapter{Representation Theory of  symmetric groups}{Historical Background of the Representation Theory of the symmetric groups}
{Historical Background of the Representation Theory of the Symmetric Group and Corresponding Hecke Algebra}    \label{sec:Intro}

        
        %

\section{Motivation of combinatorics}
The majority of this thesis is concerned with the combinatorics behind the representation theory of the symmetric group and corresponding Hecke algebra. The representation theory of the symmetric group and Hecke algebra are quite combinatorial in nature, for instance, a basis for their irreducible representations are in bijection with standard Young tableaux. In this thesis, we study the representation theory from the combinatorial side more than the algebraic. This allows us to give elementary arguments to the majority of our theorems, and is quite common in the study of the representation theory of the symmetric group. Of particular difficulty is when the characteristic of the field of the representation is not of characteristic zero (equivalently, when $q$ is a root of unity for the Hecke algebra). This chapter is a very brief summary of the combinatorial representation theory of the symmetric group and Hecke algebra.  

\section{Representation theory}\label{rep_theory}

\subsection{Definitions}
A \textit{representation} of a group $G$ on a vector space $V$ is a group homomorphism from $G$ to the group of all invertible linear maps on $V$ (this group is denoted $GL(V)$). We often identify the group element of $G$ with the invertible linear map without mentioning the representation itself. 

The dimension of the representation is defined to be the dimension of $V$. Throughout this thesis, all of the representations studied are finite dimensional. 

If $W$ is a subspace of $V$ for which every group element $g \in G$ is an automorphism of $W$ (i.e. for every $g \in G$ and $w \in W, gw\in W$), then $W$ is called a \textit{subrepresentation} of $V$. 

If $V$ has no nontrivial (not $\{0\}$ or $V$) subrepresentations then $V$ is said to be \textit{irreducible}.

If $V$ and $W$ are two representations of $G$, then their direct sum $V \bigoplus W$ has a natural structure of a representation ($g \in G$ acts on $v \oplus w$ as $g(v\oplus w) = gv \oplus gw)$. $V \bigoplus W$ naturally has $V$ and $W$ as subrepresentations.

A representation $V$ which is the direct sum of smaller representations is called \textit{decomposable}.
A representation which is not the direct sum of smaller representations is called \textit{indecomposable}.

The group algebra $\mathbb{C} G$ is the algebra with basis $\{g : g \in G\}$ with multiplication taken from the group multiplication and extended linearly.

The representation theory of a finite group is equivalent to the study of modules of the group algebra $\mathbb{C} G$. Because of this, we will often interchange the words representation and module.

\subsection{Classical results in representation theory}
Suppose we wish to study the representation theory of a finite group $G$ where all vector spaces are defined over $\mathbb{C}$. Then the following results are well known (see \cite{S}).

\begin{proposition}
The number of irreducible representations of the group $G$ is the same as the number of conjugacy classes of $G$.
\end{proposition}

\begin{theorem}[Maschke's Theorem]
Let $V$ be a representation of $G$. Then $V$ is irreducible if and only if it is indecomposable. 
\end{theorem}

Another way to state Maschke's theorem is that every representation can be written as a direct sum of irreducible representations. Given a representation, one can go about decomposing the representation by using \textit{character theory}. Although we will not go into detail, character theory yields a way of determining the number of irreducible components in a representation.

\subsection{Modular representation theory}
Now suppose we change to a field of characteristic $p$. Then the number of irreducible representations of the group changes; it is now the number of $p$-regular conjugacy classes of $G$ (a conjugacy class is $p$-regular if $p$ does not divide the order of $g$ for some (equivalently any) representative $g$ of the conjugacy class). 

We also lose Maschke's theorem; there exists representations which are indecomposable but not irreducible.  In other words, an arbitrary representation cannot be written as a direct sum of irreducible representations. To give meaning then to the number of times a given irreducible occurs in an arbitrary representation, we recall the definition of a composition series:

\begin{definition}
A \textit{composition series} of a finite dimensional representation $V$ is a sequence of representations $\{0\} = V_{0} \subset V_1 \subset \dots \subset V_k = V$ such that each quotient $V_{i+1}/V_i$ is irreducible. Each $V_{i+1}/V_i$ is called a \textit{composition factor} of $V$. 
\end{definition}

\begin{proposition}
Let $V$ be a representation and let $W$ be an irreducible representation. Then the number of times that $W$ occurs as a composition factor of $V$ is independent of the choice of composition series. In other words, composition factors are well defined.
\end{proposition}

The number of times that the irreducible representation $W$ occurs in $V$ will be denoted $[V:W]$. 

\section{The symmetric group}
The construction of the irreducible representations below can be found in Sagan's book \cite{S}.

\subsection{Definitions}
The symmetric group of degree $n$ (denoted $S_n$) is defined to be the group of all bijective functions on the set $[n] := \{ 1, 2, \dots, n\}$ of $n$ elements, with group multiplication being composition of functions. Equivalently, it can be thought of as the set of permutations on a set with $n$ elements. Functions which swap two entries and fix the rest are called \textit{transpositions}. The transpositions
which take $i$ to $i+1$ for some $i$ are called the \textit{simple transpositions}. 

One can equivalently use the definition described by generators and relations. We take as a set of generators the simple transpositions $s_i$ which take $i$ to $i+1$. The symmetric group is generated by the simple transpositions, with the relations:
\begin{itemize}
\item $s_i^2 = 1$
\item $s_i s_{i+1} s_i = s_{i+1} s_i s_{i+1}$
\item $ s_i s_j = s_j s_i  \;\;\;\;\;\;\;\;\;\;\; \textit{for } |i-j|>1 $ 
\end{itemize}

A cycle is a permutation which has no fixed points. We write $\sigma = (x_1, x_2, \dots x_k)$ to mean that $\sigma$ sends $x_1\to x_2, x_2 \to x_3, \dots , x_k \to x_1$. 

A permutation $\sigma \in S_n$ can be written as a group of disjoint cycles on subsets of $[n]$.

\begin{definition}
The \textit{cycle type} of a permutation is the non-increasing sequence of integers of the cycle lengths of the permutation.
\end{definition}

\begin{example}
The cycle type of a transposition is $(2,1,1,\dots, 1)$. 
\end{example}

\begin{example}
In $S_5$ the permutation $\sigma = (1,4,3)(2,5)$ has cycle type $(3,2)$. 
\end{example}

A permutation $\sigma$ can always be written as a product of simple transpositions. If $\sigma = s_{i_1} \dots s_{i_k}$ is a minimal expression for $\sigma$, then we say that the length of $\sigma$ is $k$, written $len(\sigma) = k$. 

\subsection{Constructing irreducible representations} 
\begin{definition}
A partition $\lambda$ of $n$ is a non-increasing sequence of non-negative integers ($\lambda = \lambda_1 \geq \lambda_2 \geq \dots \geq \lambda_k \geq \dots$) which sums to $n$ $(\sum_i \lambda_i = n)$. The set of all partitions of $n$ will be denoted $\mathcal{P}_n$. The set of all partitions (of any integer) will be denoted $\mathcal{P}$.
\end{definition}
As mentioned in Section \ref{rep_theory}, the number of irreducible representations of a finite group is the same as the number of conjugacy classes of the group. For the symmetric group, it is well known that 
the number of conjugacy classes is the same as the number of partitions (the bijection sending a conjugacy class to its cycle type).

 In what follows, we will recall the bijection between partitions and irreducible representations of the symmetric group. To any partition $\lambda$, we construct the Specht module $S^{\lambda}$ below.
 
\begin{definition}
A Young diagram for $\lambda = (\lambda_1, \dots, \lambda_k)$ is a diagram which has $\lambda_j$ boxes in row $j$.
\end{definition}

\begin{example}
$\lambda = (4,2,2,1)$ is a partition of 9. Its Young diagram is below.

\begin{center}
$\tableau{\mbox{}&\mbox{}&\mbox{}&\mbox{}\\
\mbox{}&\mbox{}\\
\mbox{}&\mbox{}\\
\mbox{}}$
\end{center}
\end{example}

\begin{definition}
A tableau of shape $\lambda$ will be a bijective assignment of the set $[n]$ to each box in the Young diagram of $\lambda$.
\end{definition}

\begin{example}
Continuing with the previous example, we give a tableau of $(4,2,2,1)$ below.

\begin{center}
$\tableau{4&3&1&9\\
5&2\\
6&8\\
7}$
\end{center}
\end{example}

Two tableau are called \textit{row equivalent} if the content in each of their rows are the same. The set of all row equivalent tableau is called a tabloid. We will always represent a tabloid with its entries increasing across rows.

The group $S_n$ has a natural action on the set of all tabloids, by acting on the entries in each box. We make the set into a vector space by making each tabloid a basis vector. This is called the permutation module, denoted $M^{\lambda}$. $M^\lambda$ is usually reducible. 

For a partition $J = (J_1, \dots , J_k)$ of the set $[n]$, we let $S_J$ be the subgroup of $S_n$ of all of the permutations $\sigma$ for which $\sigma (j_i) \in J_i$ for all $j_i \in J_i$. Then the antisymmetrizer of $J$,
 written $s_J$, is the group algebra element $\sum_{\sigma \in S_J} (-1)^{len(\sigma)} \sigma$. 

For any tableau $t$, the column partition $C_t$ will be the set partition of $[n]$ described by grouping all of the letters which share a column together.

A \textit{polytabloid} of $\lambda$ is an element in $M^\lambda$ of the form $s_{C_t} \{t\}$, where $\{t\}$ is the equivalence class of $t$ in $M^\lambda$. 

\begin{example}
If $t = \tableau{1&2\\3}$ then the columns partition $\{1,2,3\}$ into $\{1,3\}$ and $\{2\}$. Hence the corresponding polytabloid is $\tableau{1&2\\3} - \tableau{2&3\\1}\;.$
\end{example}

\begin{definition}
The Specht module $S^\lambda$ corresponding to a partition $\lambda$ is defined to be the submodule spanned by all of the polytabloids of $\lambda$.
\end{definition}

\begin{theorem}
The Specht modules $S^\lambda$ are irreducible and distinct $(S^\lambda \neq S^\mu $  if  $\lambda \neq \mu)$. Hence $\{S^\lambda: \lambda \in \mathcal{P}_n \}$ labels a complete set of irreducible representations of $S_n$.
\end{theorem}

\begin{definition}
A tableau $t$ of $\lambda$ will be called \textit{standard} if all of its entries are increasing from left to right along a row and from top to bottom down a column.
\end{definition}

\begin{theorem}
A basis for the Specht module $S^\lambda$ is the set of polytabloids generated by standard tableaux.
\end{theorem}

\begin{example}
Let $\lambda = (2,1)$. Then $S^\lambda$ is 2 dimensional. The polytabloid corresponding to the tableau $\tableau{1&2\\3}$ is $v_1 := \tableau{1&2\\3} - \tableau{2&3\\1}$. The polytabloid corresponding to the tableau $\tableau{1&3\\2}$ is $v_2 := \tableau{1&3\\2} - \tableau{2&3\\1}$. The corresponding matrices for the generators $s_1, s_2$ are: 

\begin{center}
$s_1 = \left( \begin{array}{cc} 1&0\\-1&-1\end{array} \right) $ \;\;\;\;\;\;\;\;\;\;\;\;\;\;\;\;\;\;\;\;\;\;
$s_2 = \left( \begin{array}{cc} 0&1\\1&0\end{array} \right) $
\end{center}

A simple check shows that these two operators have no common eigenvector and hence $S^{(2,1)}$ is irreducible.

\end{example}

\begin{proposition}
The matrices representing the group $S_n$ in the Specht module $S^\lambda$ are always integer matrices.
\end{proposition}

There is a functor from the category of $S_{n+1}$ modules to the category of $S_n$ modules, called \textit{restriction}, which simply views $S_n$ as a subgroup of $S_{n+1}$ and forgets the action of the generator $s_n$. It has an adjoint functor which is called \textit{induction}. Induction and restriction are described on Specht modules in the following theorem.

\begin{definition}
For a partition $\lambda$ of $n$ and $\mu$ of $n+1$, we say that $\mu \succ \lambda$ if the Young diagram of $\lambda$ fits inside the Young diagram of $\mu$. This is usually read as $\mu$ covers $\lambda$. The poset it generates is called Young's lattice.
\end{definition}

\begin{theorem}[Branching Rule]  The operations of induction and restriction have simple combinatorial descriptions in the setting of the symmetric group:
\begin{itemize}

\item$Ind \; S^\lambda = \bigoplus_{\mu \supset \lambda} S^\mu$.

\item$Res \;S^\mu = \bigoplus_{\lambda \subset \mu} S^\lambda$.
\end{itemize}

\end{theorem}

\subsection{Modular representation theory}
If $p \mid n!$ then the representation theory of the symmetric group over a field of characteristic $p$ is not well understood. Since the Specht modules constructed above have integer entries, by taking these entries mod $p$ we obtain a Specht module for the symmetric group in characteristic $p$. We will denote it by $\overline{S^\lambda}$. Similarly, we could have constructed $\overline{S^\lambda}$ by following the same construction in the previous section (these two descriptions are isomorphic). 

$\overline{S^\lambda}$ may be reducible. The  following example illustrates this point.

\begin{example}
Let $\lambda = (2,1)$ and $p = 3$. Then the matrices representing the generators are now:

\begin{center}
$s_1 = \left( \begin{array}{cc} 1&0\\2&2\end{array} \right)$ \;\;\;\;\;\;\;\;\;\;\;\;\;\;\;\;\;\;\;\;\;\;
$s_2 = \left( \begin{array}{cc} 0&1\\1&0\end{array} \right) $
\end{center}

The one dimensional subspace $W$ spanned by the vector $v_1 + v_2$ is a subrepresentation of $\overline{S^{(2,1)}}$.

 This subspace is the \textit{trivial} representation of $S_3$ (all elements of $S_3$ act as the identity).  $\overline{S^{(2,1)}} / W$ is the \textit{sign} representation (all of the simple transpositions in $S_3$ act as $-1$). 
Note though that the sign representation is not a subrepresentation of $\overline{S^{(2,1)}}$. Hence $\overline{S^{(2,1)}}$ is indecomposable but not irreducible.
\end{example}

The module $M^\lambda$ has a simple bilinear form on it; defined by making the tabloids an orthonormal basis. Restricting this form to $\overline{S^\lambda} \subset M^\lambda$, we see that this form can be degenerate. In fact, the kernel of this form is a submodule, called the radical of $\overline{S^\lambda}$, or $rad(\overline{S^\lambda})$. 

\begin{definition}
A partition $\lambda$ is called $p$-regular if it does not have $p$ or more nonzero repeating parts.
\end{definition}

We can now constuct all of the irreducible representations for the symmetric group in this case. This construction can be found in the book of James and Kerber \cite{JK}.

\begin{theorem}
The module $D^\lambda := \overline{S^\lambda} / rad(\overline{S^\lambda})$ is non-zero if and only if $\lambda$ is a $p$-regular partition. If $\lambda$ is $p$-regular then $D^\lambda$ is irreducible. The set 

\hspace{1.5in}$\{ D^\lambda \textrm{ : }\lambda \textrm{ a } p \textrm{-regular partition of } n \}$

\noindent is a complete set of nonequivalent irreducible representations of $S_n$.
\end{theorem}

One of the major questions in this field is to count the number of times a $D^{\mu}$ occurs as a composition factor of the module $\overline{S^\lambda}$. This number is denoted $d_{\lambda, \mu}$ and the corresponding matrix is called the \textit{decomposition matrix}.

\begin{definition}
The socle of a module $M$ is the direct sum of all of the irreducible submodules of $M$.
\end{definition}

A branching rule for the modular representation theory of the symmetric group is not fully understood. However, Kleshchev \cite{Kl2} gave a branching rule for the socle. 

\begin{theorem}
The socle of $Ind \,D^\lambda$ is $\bigoplus_{\lambda \rightarrow \mu} D^\mu$, where $\lambda \to \mu$ means that there is an arrow connecting $\lambda$ to $\mu$ in the crystal graph $B(\Lambda_0)$ of $\widehat{\mathfrak{sl}_p}$.

Similarly, the socle of $Res \,D^\mu$ is $\bigoplus_{\lambda \rightarrow \mu} D^\lambda$.
\end{theorem}
For a description of the crystal $B(\Lambda_0)$ see Section \ref{crystal_description}.

\subsection{Carter's criterion for the symmetric group}
Finally, we would like to mention a theorem originally due to James and Murphy (see \cite{JMu}) which is the basis of this thesis.

\begin{definition} For any box $(a,b)$ in the Young diagram of $\lambda$, the hook length corresponding to the box $(a,b)$, written $h_{(a,b)}^\lambda$, is the number of boxes to the right and below the box $(a,b)$, including the box $(a,b)$ itself.
\end{definition}

The following theorem classifies the $p$-regular partitions $\lambda$ for which $\overline{S^\lambda}$ remains irreducible in characteristic $p$. 

\begin{theorem} Let $\lambda$ be a $p$-regular partition.
The Specht module $\overline{S^\lambda}$ in characteristic $p$ is irreducible if and only if for every $a,b,$ and $c$ such that $(a,c)$ and $(b,c)$ are boxes in the Young diagram of $\lambda$, the exponent of $p$ dividing the hook length $h_{(a,c)}^\lambda$ is equal to the exponent of $p$ dividing the hook length $h_{(b,c)}^\lambda$. 
\end{theorem}

\begin{example}\label{32}
Let $\lambda = (3,2)$ and $p = 2$. Then $S^\lambda$ is reducible because $2^2 \mid h_{(1,1)}^\lambda$ but $2^2 \nmid h_{(2,1)}^\lambda$. $\lambda$, with hook lengths inserted in boxes, is drawn below.

\begin{center}
$\tableau{4&3&1\\2&1}$
\end{center}

\end{example}

\section{Hecke algebras}
We now highlight the representation theory of the finite Hecke algebra. The Hecke algebra is a deformation of the group algebra of the symmetric group. For the whole of this thesis, the algebra is defined over a field $\mathbb{F}$ of characteristic zero.
\subsection{Definitions}

\begin{definition}
Let $0 \neq q \in \mathbb{F}$. The finite Hecke Algebra $H_n(q)$ is defined to be the algebra over $\mathbb{F}$ generated by $T_1, ... ,T_{n-1}$ with relations

$$\begin{array}{ll} 	T_i  T_j = T_j  T_i    				& \textrm{for $|i-j|>1$}\\
			T_i  T_{i+1}  T_i = T_{i+1}  T_i  T_{i+1} 	& \textrm{for $i < n-1$}\\
			T_i^2 = (q-1)T_i + q				& \textrm{for $i \leq n-1$}.\\

\end{array}$$
\end{definition}
If $w$ is an element of the symmetric group, and $w = s_{i_1} \dots s_{i_k}$ is a minimal length representation of $w$, then we associate $w$ with the  corresponding Hecke algebra element $T_w := T_{i_1} \dots T_{i_k}$. The set $\{ T_w : w \in S_n \}$ is a basis of $H_n(q)$. 

We note that when $q=1$ we obtain the group algebra of the symmetric group. 

\subsection{Constructing irreducible representations}
When $q$ is not a root of unity, the number of representations of $H_n(q)$ is again the number of partitions of $n$. The representation theory of $H_n(q)$ mimics that of $\mathbb{C} S_n$. 

Similar to the representation theory of the symmetric group, Dipper and James \cite{DJ} constructed Specht modules $S^\lambda$ for $H_n(q)$. We outline the construction of $S^\lambda$ below.

First, we start by constructing the permutation module $M^\lambda$. As a vector space, $M^\lambda$ is again spanned by all tabloids. The action of $T_i$ on  a tabloid $\{t\}$ is defined by: 
\begin{itemize}
\item $T_i \{t\} = q\cdot \{t\}$ if $i$ and $i+1$ are in the same row.
\item $T_i \{t\} = \{s_i(t)\}$ if $i$ is in a row above $i+1$.
\item $T_i \{t\} =  (q-1)\cdot \{t\} + q\cdot  \{s_i(t)\}$ if $i$ is in a row below $i+1$.
\end{itemize}
Here $\{s_i(t)\}$ is the tabloid $\{t\}$ with the positions $i$ and $i+1$ switched.

We modify the definition of an antisymmetrizer of a set partition $J$ of $[n]$ by defining $s_J = \sum_{\sigma \in S_J} (-q)^{-len(\sigma)} T_\sigma$. 

For any tableau $t$, we then define the polytabloid corresponding to $t$ to be $s_{C_t} \{t\}$. The Specht module $S^\lambda$ for the finite Hecke algebra is the defined to be the submodule of $M^\lambda$ spanned by all polytabloids.

\begin{theorem}[Dipper and James]
The Specht modules $S^\lambda$ are irreducible and distinct when $q$ is not a root of unity. Hence $\{S^\lambda: \lambda \in \mathcal{P}_n \}$ labels a complete set of irreducible representations of $H_n(q)$. A basis for $S^\lambda$ are those polytabloids coming from standard tableau. Furthermore, when $q$ is specialized to be 1, then this construction yields the Specht module of $S_n$. 
\end{theorem}

\begin{example}
Let $\lambda = (2,1)$. Then the two basis vectors $v_1$ and $v_2$ are the polytabloids:

$\displaystyle v_1 = (1-q^{-3}T_1T_2T_1)\,\tableau{1&2\\3} = \tableau{1&2\\3} - q^{-2}\; \tableau{2&3\\1}$ 
\;\;\; and \\\\

$\displaystyle v_2 = (1-q^{-1}T_1)\, \tableau{1&3\\2} = \tableau{1&3\\2} - q^{-1}\tableau{2&3\\1}$\;\;. 
\\

The corresponding matrices representing $T_1$ and $T_2$ are:

\begin{center}
$T_1 = \left( \begin{array}{cc} q&0 \\ -q^{-1}&-1 \end{array} \right)$\;\;\;\;\;\;\;\;\;\;\;\;\;\;\;\;\;\;\;\;\;\;
$T_2 = \left( \begin{array}{cc} 0&q \\ 1&q-1 \end{array} \right)$
\end{center}
\end{example}
\subsection{Representation theory at a root of unity}
When $q$ is a primitive $\ell^{th}$ root of unity, the representation theory of $H_n(q)$ resembles the modular representation theory of $S_n$. Again, the Specht modules $S^{\lambda}$ can be constructed, but now they can be reducible. We show this on the example from above.

\begin{example}\label{Specht21}
Let $\lambda = (2,1)$ and $\ell = 3$ (so $1+q+q^2 = 0$). Then the matrices for $T_1$ and $T_2$ are:

\begin{center}
$T_1 = \left( \begin{array}{cc} q&0 \\ -q^2&-1 \end{array} \right)$\;\;\;\;\;\;\;\;\;\;\;\;\;\;\;\;\;\;\;\;\;\;
$T_2 = \left( \begin{array}{cc} 0&q \\ 1&q-1 \end{array} \right)$.
\end{center}

Now $S^\lambda$ has a one dimensional subrepresentation $W$ spanned by the vector $v_1+v_2$. This is the trivial subrepresentation (where $T_\sigma$ acts by $q^{len(\sigma)}$). The quotient is the sign representation (where $T_\sigma$ acts by $(-1)^{len(\sigma)}$).
\end{example}

Again, $M^\lambda$ has a bilinear form. This form is uniquely determined by the following three conditions:
\begin{itemize}
\item $\langle \{t_\lambda\}, \{t_\lambda\} \rangle = 1$. 
\item $\langle \{t\}, \{t'\} \rangle = 0 \textrm{ unless } \{t\} = \{t'\}$.
\item $\langle T_i v, w \rangle = \langle v, T_i w \rangle$ for all $v, w \in M^\lambda$.
\end{itemize}
Here $t_\lambda$ is the tableau which has the entries $1, \dots ,n$ inserted in order, going down columns and then across rows. (So for instance, $t_{(2,1)}$ has 1 and 2 in its first column).

This form can be restricted to the Specht module. It has a kernel, called the radical of $S^\lambda$. The following theorem, analogous to the modular representation theory of $S_n$, is due to Dipper and James \cite{DJ}.

\begin{theorem}
When $q$ is a primitive $\ell^{th}$ root of unity, the module $D^\lambda := S^\lambda / rad(S^\lambda)$ is non-zero if and only if $\lambda$ is an $\ell$-regular partition. If $\lambda$ is $\ell$-regular then $D^\lambda$ is irreducible. The set $\{ D^\lambda \textrm{ : }\lambda \textrm{ an } \ell \textrm{-regular partition of } n \}$ is a complete set of nonequivalent irreducible representations of $H_n(q)$.
\end{theorem}

\begin{example}
When $n=3$ and $\ell=3$ there are two $\ell$-regular partitions of $n$. They are $(3)$ and $(2,1)$. The Specht module $S^{(3)}$ is one dimensional, hence irreducible. It is $D^{(3)}$, the trivial representation. The Specht module $S^{(2,1)}$ from Example \ref{Specht21} is reducible. Its radical is the one dimensional trivial representation. The quotient $D^{(2,1)}$ is the sign representation. 
\end{example}

\subsection{Decomposition matrices}
Just like the representation theory of $S_n$, the computation of the decomposition numbers $d_{\lambda, \mu} = [S^\lambda: D^\mu]$ was difficult to describe until an algorithm (conjectured by Lascoux, Leclerc and Thibon \cite{LLT}) was given (and proven by Ariki \cite{A}) which involved the representation theory of the Fock space of the affine Lie algebra $\widehat{\mathfrak{sl}_\ell}$. This algorithm will not be described in detail. 

One particularly nice result, due to James \cite{J}, ties the representation theory of $H_n(q)$ at a $p^{th}$ root of unity to that of $S_n$ over a field of characteristic $p$. 
\begin{theorem}
There exists a unitriangular matrix $A$ with positive integer coefficients such that the decomposition matrices $D_{S_n}$ of $S_n$ over a field of characterisitic $p$ and $D_{H_n(q)}$ when $q$ is a $p^{th}$ root of unity are related by:
$$D_{S_n} = D_{H_n(q)} \cdot A.$$
\end{theorem}

\subsection{Carter's criterion for the finite Hecke Algebra}
Similar to the representation theory of $S_n$, James and Mathas \cite{JM} gave a simple criterion for $S^{\lambda}$ to be an irreducible module for $H_n(q)$ when $q$ is an $\ell^{th}$ root of unity. Their rule is defined over any field, but we will restrict it to a field of characteristic zero. 

For $k \in \mathbb{Z}$, let
$$\nu_{\ell}(k) = \left\{ 	\begin{array}{ll}
			  1 &  \textrm{ $\ell \mid k$}\\
			  0 &  \textrm{ $\ell \nmid k$}. 
			\end{array} \right.$$ 

The generalization of Carter's criterion in this setting is then:

\begin{theorem}
The Specht module $S^{\lambda}$ indexed by an $\ell$-regular partition $\lambda$  is irreducible  if and only if 
$$\begin{array}{lccrr} (\star) &  \nu_{\ell}(h_{(a,c)}^\lambda) = \nu_{\ell}(h_{(b,c)}^\lambda) & \textrm{for all pairs $(a,c)$, $(b,c) \in \lambda$} \end{array}$$
\end{theorem}

Partitions which satisfy $(\star)$ are called $(\ell,0)$-Carter partitions.

\begin{example}
Let $\ell = 2$ and $\lambda = (3,2)$ as in Example \ref{32}. Then $S^\lambda$ is irreducible when $q$ is an $\ell^{th}$ root of unity, because each of the hook lengths in the first column are divisible by $\ell$, and none of the other hook lengths are.
\end{example}

The results of Chapter \ref{sec:LabelForChapter2} are mostly about $(\ell,0)$-Carter partitions.

\section{The basic crystal}\label{crystal_description}
We will not outline the theory of Lie algebras, but will note that the representation theory of the symmetric group over a field of characteristic $\ell$ (or of the representation theory of $H_n(q)$ at an $\ell^{th}$ root of unity) is connected to the crystal theory of the basic representation of the affine Lie algebra $\widehat{\mathfrak{sl}_{\ell}}$. This crystal is denoted $B(\Lambda_0)$ and will be described combinatorially in this section.

\subsection{Combinatorial description of the basic crystal} 

The set of nodes of $B(\Lambda_0)$ is denoted $B := \{ \lambda \in \mathcal{P} : \, \lambda \textrm{ is } \ell \textrm{-regular} \} $. We will describe the arrows of $B(\Lambda_0)$ below. This description is originally due to Misra and Miwa (see \cite{MM}).

We now view the Young diagram for $\lambda$ as a set of boxes, with their corresponding residues $b-a \mod \ell$ written into the box in row $a$ and column $b$. A box $x$ in $\lambda$ is said to be a removable $i$-box if it has residue $i$ and after removing $x$ from $\lambda$ the remaining diagram is still a partition. A position $y$ not in $\lambda$ is an addable $i$-box if it has residue $i$ and adding $y$ to $\lambda$ yields a partition.

\begin{example}
Let $\lambda = (8,5,4,1)$ and $\ell = 3$. Then the residues are filled into the corresponding Young diagram as follows:

$$\begin{array}{cc}
\lambda = & \tableau{0&1&2&0&1&2&0&1 \\
2&0&1&2&0 \\
1&2&0&1\\0}
\end{array}
$$

Here $\lambda$ has two removable 0-boxes (the boxes (2,5) and (4,1)), two removable 1-boxes (the boxes (1,8) and (3,4)), no removable 2-boxes, no addable 0-boxes, two addable 1-boxes (at (2,6) and (4,2)), and three addable 2-boxes (at (1,9), (3,5) and (5,1)).
\end{example}

For a fixed $i$, ($0 \leq i < \ell$), we place $-$ in each removable $i$-box and $+$ in each addable $i$-box. The $i$-signature of $\lambda$ is the word of $+$ and $-$'s in the diagram for $\lambda$, read from bottom left to top right. The reduced $i$-signature is the word obtained after repeatedly removing from the $i$-signature all pairs $- +$. The reduced $i$-signature is of the form $+ \dots +++--- \dots -$. The boxes corresponding to $-$'s in the reduced $i$-signature are called \textit{normal $i$-boxes}, and the boxes corresponding to $+$'s are called \textit{conormal $i$-boxes}. $\varepsilon_i(\lambda)$ is defined to be the number of normal $i$-boxes of $\lambda$, and $\varphi_i(\lambda)$ is defined to be the number of conormal $i$-boxes. If a leftmost $-$ exists in the reduced $i$-signature, the box corresponding to said  $-$ is called the \textit{good $i$-box} of $\lambda$. If a rightmost $+$ exists in the reduced $i$-signature, the box corresponding to said $+$ is called the \textit{cogood $i$-box}. All of these definitions can be found in Kleshchev's book \cite{Kl}.

\begin{example}
Let $\lambda = (8,5,4,1)$ and $\ell =3$ be as above. Fix $i=1$. The diagram for $\lambda$ with removable and addable 1-boxes marked looks like:
$$\tableau{ \mbox{} & \mbox{} & \mbox{}& \mbox{} & \mbox{}& \mbox{}& \mbox{}& -\\  \mbox{}& \mbox{}& \mbox{}&\mbox{} & \mbox{} \,\,\,\,\,\,\,\,\,\,\,\,\,\,\,\,\,\,\,\,\,+\\ \mbox{}& \mbox{}& \mbox{}& -\\ \mbox{}\,\,\,\,\,\,\,\,\,\,\,\,\,\,\,\,\,\,\,\,\,+ } $$

The 1-signature of $\lambda$ is $+-+-$, so the reduced 1-signature is $+ \,\,\,\,\,\,\,\,\,\,-$ and the diagram has a good 1-box in the first row, and a cogood 1-box in the fourth row. Here $\varepsilon_1(\lambda)=1$ and $\varphi_1(\lambda)=1$. 
\end{example}

We recall the action of the crystal operators on $B.$ The crystal operator $\widetilde{e}_{i}: B \xrightarrow{i} B \cup \{0\}$ assigns to a partition  $\lambda$ the partition $\widetilde{e}_{i}(\lambda) = \lambda \setminus x$, where $x$ is the good $i$-box of $\lambda$. If no such box exists, then $\widetilde{e}_{i}(\lambda)=0$. It is clear that $\varepsilon_i(\lambda) = \max\{k : \widetilde{e}_{i}^k \lambda \neq 0\}$.

Similarly, $\widetilde{f}_{i}: B \xrightarrow{i} B \cup \{0\}$ is the operator which assigns to a partition  $\lambda$ the partition $\widetilde{f}_{i}(\lambda) = \lambda \cup x$, where $x$ is the cogood $i$-box of $\lambda$. If no such box exists, then $\widetilde{f}_{i}(\lambda)=0$. It is clear that $\varphi_i(\lambda) = \max\{k : \widetilde{f}_{i}^k \lambda \neq 0\}$.

For $i$ in $\mathbb{Z}/ \ell \mathbb{Z}$, we write $\lambda \xrightarrow{i} \mu$ to stand for $\widetilde{f}_{i} \lambda = \mu$. We say that there is an $i$-arrow from $\lambda$ to $\mu$. Note that $\lambda \xrightarrow{i} \mu$ if and only if $\widetilde{e}_{i} \mu = \lambda$. A maximal chain of consecutive $i$-arrows is called an $i$-string. We note that the empty partition $\emptyset$ is the unique highest weight node of the crystal. For a picture of the first few levels of this crystal graph, see \cite{LLT} for the cases $\ell = 2$ and $3$.

\begin{example}
Continuing with the above example, $\widetilde{e}_{1} (8,5,4,1) = (7,5,4,1)$ and $\widetilde{f}_{1}(8,5,4,1) = (8,5,4,2)$. Also, $\widetilde{e}_{1} ^2(8,5,4,1) = 0$ and $\widetilde{f}_{1}^2(8,5,4,1) = 0$. The sequence $(7,5,4,1) \xrightarrow{1} (8,5,4,1) \xrightarrow{1} (8,5,4,2)$ is a $1$-string of length 3.
\end{example}
  For the rest of this thesis, $\varphi = \varphi_i(\lambda)$ and $\varepsilon = \varepsilon_i(\lambda)$.

    %
    %

    \newchapter{l-partitions}{A Combinatorial Study of l-partitions}{Carter partitions, their crystal-theoretic behavior and generating function}
    \label{sec:LabelForChapter2}

        
    
The results of this chapter are joint work with Monica Vazirani.
\section{Introduction}
\subsection{Preliminaries}
Let $\lambda$ be a partition of $n$ and $\ell \geq 2$ be an integer. We will use the convention $(x,y)$ to denote the box which sits in the $x^{\textrm{th}}$ row and the $y^{\textrm{th}}$ column of the Young diagram of $\lambda$. The length of a partition $\lambda$ is defined to be the number of nonzero parts of $\lambda$ and is denoted $len(\lambda)$.

A \textit{removable $\ell$-rim hook} in $\lambda$ is a connected sequence of $\ell$ boxes in the Young diagram of $\lambda$, containing no $2 \times 2$ square, such that when removed from $\lambda$, the remaining diagram is the Young diagram of some other partition. We will often abbreviate and call a removable $\ell$-rim hook an $\ell$-rim hook.

Any partition which has no removable $\ell$-rim hooks is called an \textit{$\ell$-core}. Every partition $\lambda$ has a well defined $\ell$-core, which is obtained by removing $\ell$-rim hooks  from the outer edge while at each step the removal of a hook is still a (non-skew) partition. The $\ell$-core is uniquely determined from the partition, independently of choice of  the order in which one successively removes $\ell$-rim hooks. The number of $\ell$-rim hooks which must be removed from a partition $\lambda$ to obtain its core is called the \textit{weight} of $\lambda$. The set of all $\ell$-cores is denoted $\mathcal{C}_{\ell}$. See \cite{JK} for more details.

\begin{remark}\label{divisibility}
A necessary and sufficient condition that $\lambda$ be an $\ell$-core is that $\ell \nmid h_{(a,c)}^\lambda$ for all $(a,c) \in \lambda$ (see \cite{JK}).
\end{remark}

Removable $\ell$-rim hooks which are flat (i.e. those whose boxes all sit in one row) will be called \textit{horizontal $\ell$-rim hooks}. Removable $\ell$-rim hooks which are not flat will be called \textit{non-horizontal $\ell$-rim hooks}.

\begin{definition}\label{lpartition}
An \textit{$\ell$-partition} is an $\ell$-regular partition which
\begin{itemize}
\item Has no non-horizontal $\ell$-rim hooks.
\item After removing any number of horizontal $\ell$-rim hooks, in any order, the remaining diagram has no non-horizontal $\ell$-rim hooks. 
\end{itemize}
\end{definition}
\begin{example}  Any $\ell$-core is also an $\ell$-partition. \end{example}

\begin{example}
$(5,4,1)$ is a 6-core, hence a 6-partition. It is a 2-partition, but not a 2-core. It is not a 3-, 4-, 5-, or 7-partition. It is an $\ell$-core for $\ell > 7$.
$$ 
	\tableau{ \mbox{} &  \mbox{} &  \mbox{} &  \mbox{} &  \mbox{}\\
	 \mbox{} &  \mbox{} &  \mbox{} &  \mbox{}\\
	 \mbox{}\\}
$$
\end{example}

We turn our attention to the representation theory of $H_n(q)$ at an $\ell^{th}$ root of unity. As we mentioned in Chapter \ref{sec:Intro},
it is known that the Specht module $S^{\lambda}$ indexed by an $\ell$-regular partition $\lambda$  is irreducible  if and only if 
$$\begin{array}{lccrr} (\star) &  \nu_{\ell}(h_{(a,c)}^\lambda) = \nu_{\ell}(h_{(b,c)}^\lambda) & \textrm{for all pairs $(a,c)$, $(b,c) \in \lambda$} \end{array}$$
(see \cite{JM}). Partitions which satisfy $(\star)$ have been called in the literature $(\ell,0)$-Carter partitions.    So, an equivalent condition for the irreducibility of the Specht 
module indexed by an $\ell$-regular partition is that the hook lengths in any column of the partition $\lambda$ are either all divisible by $\ell$ or none of them are.

\subsection{Main results of Chapter \ref{sec:LabelForChapter2}}
In Section \ref{new_def_ell_part} we show the equivalence of $\ell$-partitions and $(\ell,0)$-Carter partitions (see Theorem \ref{maintheorem}). Section \ref{?} gives a different classification of $\ell$-partitions  which allows us to give an explicit formula for a generating function for the number of $\ell$-partitions with respect to the statistic of a partitions first part. In Section \ref{??} we give a crystal theoretic interpretation of $\ell$-partitions. There we explain where in the crystal graph $B(\Lambda_0)$ one can expect to find $\ell$-partitions (see Theorems \ref{top_and_bottom},  \ref{other_cases} and \ref{second_from_bottom}). In Section \ref{new_proof}, we give a representation theoretic proof of Theorem \ref{top_and_bottom}.
\section{l-partitions}\label{new_def_ell_part}
We now claim that a partition is an $\ell$-partition if and only if it satisfies $(\star)$ (We drop the $(\ell,0)$-Carter partition notation and abbreviate it by $(\star)$). To prove this, we will first need two lemmas which tell us when we can add/remove horizontal $\ell$-rim hooks to/from a diagram.
\subsection{Equivalence of the combinatorics}
\begin{lemma}\label{lemma1}
For $\lambda$ a partition which does not satisfy ($\star$), if we add a horizontal $\ell$-rim hook to $\lambda$ to form a new partition $\mu$, $\mu$ will also not satisfy ($\star$).
\end{lemma}
\begin{proof} If $\lambda$ does not satisfy ($\star$), it means that somewhere in the partition there are two boxes $(a,c)$ and $(b,c)$ with $\ell$ dividing exactly one of $h_{(a,c)}^{\lambda}$ and $h_{(b,c)}^{\lambda}$. We will assume $a < b$. Here we prove the case that $\ell \mid h_{(a,c)}^{\lambda}$ and $\ell \nmid h_{(b,c)}^{\lambda}$, the other case being similar. 

\subsubsection*{Case 1}
It is easy to see that adding a horizontal $\ell$-rim hook in row $i$ for $i < a$ or $a<i<b$ will not change the hook lengths in the boxes $(a,c)$ and $(b,c)$. In other words, $h_{(a,c)}^{\lambda}$ = $h_{(a,c)}^{\mu}$ and $h_{(b,c)}^{\lambda} = h_{(b,c)}^{\mu}$. 

\subsubsection*{Case 2}
 If the horizontal $\ell$-rim hook is added to row $a$, then $h_{(a,c)}^{\lambda} + \ell = h_{(a,c)}^{\mu}$ and $h_{(b,c)}^{\lambda} = h_{(b,c)}^{\mu}$.  Similarly if the new horizontal $\ell$-rim hook is added in row b, $h_{(a,c)}^{\lambda} = h_{(a,c)}^{\mu}$ and $h_{(b,c)}^{\lambda} + \ell = h_{(a,c)}^{\mu}$. Still, $\ell \mid h_{(a,c)}^{\mu}$ and $\ell \nmid h_{(b,c)}^{\mu}$.

\subsubsection*{Case 3}
 Suppose the horizontal $\ell$-rim hook is added in row $i$ with $i>b$. If the box $(i,c)$ is not in the added $\ell$-rim hook then $h_{(a,c)}^{\lambda}$ = $h_{(a,c)}^{\mu}$ and $h_{(b,c)}^{\lambda} = h_{(b,c)}^{\mu}$. If the box $(i,c)$ is in the added $\ell$-rim hook, then there are two sub-cases to consider. 
 If $(i,c)$ is the rightmost box of the added $\ell$-rim hook then $\ell \mid h_{(a,c- \ell + 1)}^{\mu}$ and $\ell \nmid h_{(b,c- \ell + 1)}^{\mu}$. 
 Otherwise $(i,c)$ is not at the end of the added $\ell$-rim hook, in which case $\ell \mid h_{(a,c + 1)}^{\mu}$ and $\ell \nmid h_{(b,c + 1)}^{\mu}$. In all cases, $\mu$ does not satisfy $(\star)$.
\end{proof}

\begin{example}

 Let $\lambda$ = $(14,9,5,2,1)$ and $\ell = 3$. This partition does not satisfy ($\star$). For instance, looking at boxes $(2,3)$ and $(3,3)$, we see that $3 \mid h_{(3,3)}^{\lambda} =3$ but $3 \nmid h_{(2,3)}^{\lambda} = 8$. Let $\lambda[i]$ denote the partition obtained when adding a horizontal $\ell$-rim hook to the $i$th row of $\lambda$ (when it is still a partition). Adding a horizontal $3$-rim hook in row 1 will not change  $h_{(2,3)}^{\lambda}$ or $h_{(3,3)}^{\lambda}$ (Case 1 of Lemma \ref{lemma1}).  Adding a horizontal 3-rim hook to row 2 will make $h_{(2,3)}^{\lambda[2]} = 11$, which is equal to $h_{(2,3)}^{\lambda}$ mod 3 (Case 2 of Lemma \ref{lemma1}). Adding in row 3 is also Case 2. Adding a horizontal 3-rim hook to row 4 will make $h_{(2,3)}^{\lambda[4]} = 9$ and $h_{(3,3)}^{\lambda[4]}=4$, but one column to the right, we see that now $h_{(2,4)}^{\lambda[4]} = 8$ and $h_{(3,4)}^{\lambda[4]} = 3$ (Case 3 of Lemma \ref{lemma1}).
$$\begin{array}{cc} \tableau{18 & 16 & 14 & 13 & 12 & 10 & 9 & 8 & 7 & 5 & 4 & 3 & 2 & 1\\
12&10&8&7&6&4&3&2&1\\
7&5&3&2&1\\
3&1\\
1} & 
\linethickness{2.5pt}
\put (-226,-37){\line(0,1){37}}
\put (-208,-37){\line(0,1){37}}
\put (-226,-36){\line(1,0){19}}
\put (-226,-18){\line(1,0){19}}

\put (-226,-1){\line(1,0){19}}

\end{array}$$
\end{example}

\begin{lemma}\label{lemma2}
Suppose ${\lambda}$ does not satisfy ($\star$). Let $a,b,c$ be such that $\ell$ divides exactly one of $h_{(a,c)}^{\lambda}$ and $h_{(b,c)}^{\lambda}$ with $a<b$. Suppose $\nu$ is a partition obtained from $\lambda$ by removing a horizontal $\ell$-rim hook, and that $(b,c) \in \nu$. Then $\nu$ does not satisfy ($\star$).
\begin{proof} All cases are done similar to the proof of Lemma \ref{lemma1}. The only case which does not apply is when you remove a horizontal strip from row $b$ and remove the box $(b,c)$. \end{proof} 
\end{lemma}
\begin{remark}\label{remark1} In the proof of Lemma \ref{lemma1} we have also shown that when adding a horizontal $\ell$-rim hook to a partition which does not satisfy $(\star)$, the violation to ($\star$) occurs in the same rows as the original partition. It can also be shown in Lemma \ref{lemma2} that when removing a horizontal $\ell$-rim hook (in the cases above), the violation will stay in the same rows as the original partition. This will be useful in the proof of Theorem \ref{maintheorem}.
\end{remark}
\begin{example}
$(5,4,1)$ does not satisfy ($\star$) for $\ell=3$. The boxes $(1,2)$ and $(2,2)$ are a violation of $(\star)$. Removing a horizontal 3-rim hook will give the partition $(5,1,1)$ which does satisfy ($\star$). Note that this does not violate Lemma \ref{lemma2}, since the removed horizontal $3$-rim hook contains the box $(2,2)$.
$$\begin{array}{cccc}
	\tableau{7&5&4&3&1\\
		5&3&2&1\\
		1}&$\:\:\:\:\:\:\:\:\,\,\,\,$
		&
	\tableau{7&4&3&2&1\\
         		2\\
		1}
		&
		\linethickness{2.5pt}
		\put (-220,18){\line(1,0){18}}		
		\put (-220,0){\line(1,0){18}}
		\put (-220,-18){\line(1,0){18}}
		\put (-220,-18){\line(0,1){36}}
		\put (-202,-18){\line(0,1){36}}

		\end{array}
$$
\end{example}

\indent

\begin{theorem}\label{maintheorem}
A partition is an $\ell$-partition if and only if it satisfies ($\star$).
\end{theorem}
\begin{proof} 
{\sloppy Suppose $\lambda$ is not an $\ell$-partition. We may remove horizontal $\ell$-rim hooks from $\lambda$ until we obtain a partition $\mu$ which has a non-horizontal $\ell$-rim hook.

}

{\sloppy
We label the upper rightmost box of the non-horizontal $\ell$-rim hook $(a,c)$ and lower leftmost box $(b,d)$ with $a < b$. 
Then $h_{(a,d)}^{\mu} = \ell$ and $h_{(b,d)}^{\mu} < \ell$, so $\mu$ does not satisfy ($\star$). 
From Lemma \ref{lemma1}, since $\lambda$ is obtained from $\mu$ by adding horizontal $\ell$-rim hooks, $\lambda$ also does not satisfy ($\star$).
}

Conversely, suppose  $\lambda$ does not satisfy ($\star$). Let $(a,c), (b,c) \in \lambda$ be such that $\ell$ divides exactly one of $h_{(a,c)}^{\lambda}$ and $h_{(b,c)}^{\lambda}$.  Let us assume that $\lambda$ is an $\ell$-partition and we will derive a contradiction.

\subsubsection*{Case 1} Suppose that $a<b$ and that $\ell \mid h_{(a,c)}^{\lambda}$. Then without loss of generality we may assume that $b=a+1$.  By the equivalent characterization of $\ell$-cores mentioned in Remark \ref{divisibility}, there exists at least one removable $\ell$-rim hook in $\lambda$ . By assumption it must be horizontal.
 If an $\ell$-rim hook exists which does not contain the box $(b,c)$ then let $\lambda^{(1)}$ be $\lambda$ with this $\ell$-rim hook removed. By Lemma \ref{lemma2}, since we did not remove the $(b,c)$ box, $\lambda^{(1)}$ will still not satisfy ($\star$).  Then there are boxes $(a, c_1)$ and $(b, c_1)$ for which $\ell \mid h_{(a,c_1)}^{\lambda^{(1)}}$ but $\ell \nmid h_{(b,c_1)}^{\lambda^{(1)}}$. 
By Remark \ref{remark1} above, we can assume that the violation to $(\star)$ is in the same rows $a$ and $b$ of $\lambda^{(1)}$. We apply the same process as above repeatedly until we must remove a horizontal $\ell$-rim hook from the partition $\lambda^{(k)}$ which contains the $(b,c_k)$ box, and in particular we cannot remove a horizontal $\ell$-rim hook from row $a$. Let $d$ be so that $h_{(b,d)} = 1$. Such a $d$ must exist since we can remove a horizontal $\ell$-rim hook from this row. Since $(b,c_k)$ is removed from $\lambda^{(k)}$ when we remove the horizontal $\ell$-rim hook, $h_{(b,c_k)}^{\lambda^{(k)}}  < \ell$ ($\ell$ does not divide $h_{(b,c_k)}^{\lambda^{(k)}}$ by assumption, so in particular $h_{(b,c_k)}^{\lambda^{(k)}} \neq \ell$). Note that $h_{(a,c_k)}^{\lambda^{(k)}} = h_{(b,c_k)}^{\lambda^{(k)}} + h_{(a,d)}^{\lambda^{(k)}} -1$, $\ell \mid h_{(a,c_k)}^{\lambda^{(k)}}$ and $\ell \nmid h_{(b,c_k)}^{\lambda^{(k)}}$, so $\ell \nmid (h_{(a,d)}^{\lambda^{(k)}} -1)$. If $h_{(a,d)}^{\lambda^{(k)}} -1 > \ell$ then we could remove a horizontal $\ell$-rim hook from row $a$, which we cannot do by assumption. Otherwise $h_{(a,d)}^{\lambda^{(k)}} < \ell$. Then a non-horizontal $\ell$-rim hook exists starting at the rightmost box of the $a^{th}$ row, going left to $(a,d)$, down to $(b,d)$ and then left. This is a contradiction as we have assumed that $\lambda$ was an $\ell$-partition.
\subsubsection*{Case 2} Suppose that $a<b$ and that $\ell \mid h_{(b,c)}^{\lambda}$. We will reduce this to Case 1. Without loss of generality we may assume that $b = a+1$ and that $\ell \mid h_{(n,c)}^{\lambda}$ for all $n>a$, since otherwise we are in Case 1. Let $m$ be so that $(m,c) \in \lambda$ but $(m+1,c) \not\in \lambda$. Then because $h_{(m,c)}^{\lambda} \geq \ell$, the list $h_{(a,c)}^{\lambda}, h_{(a, c+1)}^{\lambda} = h_{(a,c)}^{\lambda} -1 , \ldots ,h_{(a, c+\ell-1)}^{\lambda} = h_{(a,c)}^{\lambda} -\ell+1$ consists of $\ell$ consecutive integers. Hence one of them must be divisible by $\ell$. Suppose it is $h_{(a, c+i)}^{\lambda}$. Note $\ell \nmid h_{(m, c+i)}^{\lambda}$, since $h_{(m, c+i)}^{\lambda} = h_{(m,c)}^{\lambda} - i$ and $\ell \mid h_{(m,c)}^{\lambda}$. Then we may apply Case 1 to the boxes $(a,c+i)$ and $(m,c+i)$.

\end{proof}

\begin{remark} This result can actually be obtained using a more general result of James and Mathas (\cite{JM}, Theorem 4.20), where they classified which $S^{\lambda}$ remain irreducible for $\lambda$ $\ell$-regular. However, we have included this proof to emphasize the simplicity of the theorem and its simple combinatorial proof in this context. \end{remark}

\begin{remark}  When $q$ is a primitive $\ell$th root of unity, and $\lambda$ is an $\ell$-regular partition, the Specht module $S^{\lambda}$ of $H_n(q)$ is irreducible if and only if  $\lambda$ is an $\ell$-partition.
This follows from what was said above concerning the James and Mathas result on  the equivalence of ($\star$) and irreducibility of Specht modules, and Theorem \ref{maintheorem}.
\end{remark}

\section{Generating functions}\label{?}
Let $\mathcal{L}_{\ell} = \{\lambda \in \mathcal{P} : \lambda\textrm{ is an }\ell\textrm{-partition} \}$. In this section, we count $\ell$-partitions via generating functions. We study the generating function with the statistic being the first part of the partition. Let $B_{\ell}(x) = \sum_{k=0}^{\infty} b_k^{\ell} x^k$ where $b_k^{\ell} = \# \{ \lambda: \lambda_1 = k, \lambda \in \mathcal{L}_{\ell}  \}$ is the number of $\ell$-partitions with the first part of the partition equal to $k$, i.e. $B_{\ell}(x) = \sum_{\lambda \in \mathcal{L}_{\ell}} x^{\lambda_1}$.

\subsection{Counting cores}\label{counting_ell_cores}
We will count $\ell$-cores first, with respect to the statistic of the first part of the partition. Let $$C_{\ell}(x) = \sum_{k=0}^{\infty} c_k^{\ell} x^k$$ where $c_k^{\ell} =  \#\{ \lambda \in \mathcal{C}_{\ell} : \lambda_1 = k\}$. Note that this does not depend on the size of the partition, only its first part. Also, the empty partition is considered as the unique partition with first part $0$, and is always a core, so that $c_0^{\ell} =1$ for every $\ell$.

\begin{example}\label{2_cores}

For $\ell = 2$, all 2-cores are staircases, i.e. are of the form $\lambda = (k,k-1,\dots,2,1)$. Hence $C_2 (x) = \sum_{k=0}^{\infty} x^k = \frac{1}{1-x}$.
\end{example}
\begin{example}
For $\ell = 3$, the first few cores are $$\emptyset, (1), (1,1), (2), (2,1,1), (2,2,1,1), (3,1),(3,1,1),(3,2,2,1,1),(3,3,2,2,1,1), \dots $$ so $C_3(x) = 1+ 2x +  3x^2 + 4x^3 + \hdots$
\end{example}

For a partition $\lambda = (\lambda_1,...,\lambda_s)$ with $\lambda_s>0$, the \textit{$\beta$-numbers }$(\beta_1,...,\beta_s)$ of $\lambda$ are defined to be the hook lengths of the first column (i.e. $\beta_i = h_{(i,1)}^{\lambda}$). Note that this is a modified version of the $\beta$-numbers defined by James and Kerber in \cite{JK}, where all definitions in this section can be found (for the standard definition of $\beta$-numbers, see Chapter \ref{sec:LabelForChapter4}). We draw a diagram $\ell$ columns wide with the numbers $\{0,1,2,\dots,\ell-1\}$ inserted in the first row in order, $\{ \ell, \ell +1, \dots, 2\ell-1\}$ inserted in the second row in order, etc. Then we circle all of the $\beta$-numbers for $\lambda$. The columns of this diagram are called \textit{runners}, the circled numbers are called \textit{beads}, the uncircled numbers are called \textit{gaps}, and the diagram is called an \textit{abacus} . It is well known that a partition $\lambda$ is an $\ell$-core if and only if all of the beads lie in the last $\ell -1$ runners and there is no gap above any bead.

\begin{example}

 $\lambda = (4,2,2,1,1)$ has $\beta$-numbers $8,5,4,2,1$, and for $\ell =3$ our abacus has the first runner empty, the second runner has beads at 1 and 4, and the third runner has beads at 2, 5 and 8 (as pictured below). Hence this is a $3$-core.
\begin{center}
$
\begin{array}{lr}
\tableau{8 & 5 & 2 & 1\\
	5 & 2\\
	4 & 1\\
	2\\
	1}

&
\begin{picture}(80,40)
\put (10,0) {0}
\put (40,0){1}
\put (70,0){2}
\put (10,-15){3}
\put (40, -15){4}
\put (70,-15){5}
\put (10,-30){6}
\put (40,-30){7}
\put (70,-30){8}
\put (10,-45){9}
\put (38,-45){10}
\put (68,-45){11}
\put (72.5,-12){\circle{13}}
\put (72.5,3){\circle{13}}
\put (72.5,-27){\circle{13}}
\put (42.5,3){\circle{13}}
\put (42.5,-12){\circle{13}}
\put (42.5,-53){.}
\put (42.5,-57){.}
\put (42.5,-61){.}
\put (72.5,-53){.}
\put (72.5,-57){.}
\put (72.5,-61){.}
\put (11.5,-53){.}
\put (11.5,-57){.}
\put (11.5,-61){.}
\end{picture}
\end{array}
$
\end{center}

$$\begin{array}{c} \textrm{Young diagram and abacus of $\lambda = (4,2,2,1,1)$ } \end{array}$$
\end{example}

\begin{proposition}\label{ell_core_bijection}
There is a bijection between the set of $\ell$-cores with first part $k$ and the set of $(\ell-1)$-cores with first part $\leq k$.
\end{proposition}

\begin{proof}
Using the abacus description of cores, we describe our bijection as follows: 

Given an $\ell$-core with largest part $k$, remove the whole runner which contains the largest bead (the bead with the largest $\beta$-number). In the case that there are no beads, remove the rightmost runner. The remaining runners can be placed into an $\ell-1$ abacus in order. The remaining abacus will clearly have its first runner empty. This will correspond to an $(\ell - 1)$-core with largest part at most $k$. This map gives a bijection between the set of all $\ell$-cores with largest part $k$ and the set of all $(\ell-1)$-cores with largest part at most $k$.

To see that it is a bijection, we will give its inverse. Given the abacus for an $(\ell-1)$-core $\lambda$ and a $k \geq \lambda_1$, insert the new runner directly after the $k^{th}$ gap, placing a bead on it directly after the $k^{th}$ gap and at all places above that bead on the new runner.
\end{proof}

\begin{corollary}\label{ell_core_binom}
$c_k^{\ell} = \binom{k+\ell -2}{k}.$
\end{corollary}

\begin{proof}  This proof is by induction on $\ell$. For $\ell = 2$, as the only $1$-core is the empty partition, by Proposition \ref{ell_core_bijection}\, $c_k^2 = 1 = \binom{k}{k}$. Note this was also observed in Example \ref{2_cores}. For the rest of the proof, we assume that $\ell >2$. 

It follows directly from Proposition \ref{ell_core_bijection} that $$(\sharp) \,\,\,\,\,\,\,c_k^{\ell} = \sum_{j=0}^k c_j^{\ell -1}.$$ 
Recall the fact that $\binom{\ell+k-2}{k} = \binom{\ell-3}{0} + \binom{\ell-2}{1}+ \dots + \binom{\ell+k-3}{k}$ for $\ell>2$. 
Applying our inductive hypothesis to all of the terms in the right hand side of $(\sharp)$ we get that  $c_k^\ell=  \sum_{j=0}^k c_j^{\ell-1}  = \sum_{j=0}^k \binom{\ell+j-3}{j}=  \binom{\ell+k-2}{k}$. 
Therefore, the set of all $\ell$-cores with largest part $k$ has cardinality $\binom{k+\ell-2}{k}$.
\end{proof}

\begin{remark}
The bijection above between $\ell$-cores with first part $k$ and $(\ell -1)$-cores with first part $\leq k$ has several other beautiful descriptions, using different interpretations of $\ell$-cores. In Chapter \ref{sec:LabelForChapter4} we explore these descriptions in detail.
\end{remark}

\begin{example}

Let $\ell = 3$ and $\lambda = (4,2,2,1,1)$. The abacus for $\lambda$ is:
 
\begin{center}
\begin{picture}(80,80)
\put (10,65) {0}
\put (40,65){1}
\put (70,65){2}
\put (10,50){3}
\put (40, 50){4}
\put (70,50){5}
\put (10,35){6}
\put (40,35){7}
\put (70,35){8}
\put (10,20){9}
\put (38,20){10}
\put (68,20){11}
\put (72.5,53){\circle{13}}
\put (72.5,68){\circle{13}}
\put (72.5,38){\circle{13}}
\put (42.5,68){\circle{13}}
\put (42.5,53){\circle{13}}
\put (42.5,12){.}
\put (42.5,8){.}
\put (42.5,4){.}
\put (72.5,12){.}
\put (72.5,8){.}
\put (72.5,4){.}
\put (11.5,12){.}
\put (11.5,8){.}
\put (11.5,4){.}
\end{picture}
\end{center}
The largest $\beta$-number is 8. Removing the whole runner in the same column as the 8, we get the remaining diagram with runners relabeled for $\ell = 2$
\begin{center}
\begin{picture}(80,80)
\put (10,65) {0}
\put (40,65){1}
\put (70,65){}
\put (10,50){2}
\put (40, 50){3}
\put (70,50){}
\put (10,35){4}
\put (40,35){5}
\put (70,35){}
\put (10,20){6}
\put (40,20){7}
\put (68,20){}
\put (72.5,53){\circle{13}}
\put (72.5,68){\circle{13}}
\put (72.5,38){\circle{13}}
\put (42.5,68){\circle{13}}
\put (42.5,53){\circle{13}}

\put (68,64){$\times$}
\put (68,49){$\times$}
\put (68,34){$\times$}
\put (68,19){$\times$}
\put (68,4){$\times$}

\put (42,12){.}
\put (42,8){.}
\put (42,4){.}
\put (10.5,12){.}
\put (10.5,8){.}
\put (10.5,4){.}
\end{picture}.
\end{center}

These are the $\beta$-numbers for the partition $(2,1)$, which is a 2-core with largest part $\leq 4$. 
\end{example}
 From the above reasoning, we obtain $C_{\ell}(x) = \sum_{k \geq 0} \binom{k+\ell-2}{k} x^k$ and so conclude the following.
 
 \begin{proposition}
 $$C_{\ell}(x) = \frac{1}{(1-x)^{\ell -1}}.$$
\end{proposition}

\subsection{Decomposing l-partitions}\label{construct}
We now describe a decomposition of $\ell$-partitions. We will use this to build $\ell$-partitions from $\ell$-cores and extend our generating function to $\ell$-partitions.

\begin{lemma}\label{newcores}
Let $\lambda$ be an $\ell$-core and $r>0$ an integer. Then 
\begin{enumerate}
\item $\nu = (\lambda_1+r(\ell-1), \lambda_1 + (r-1)(\ell-1), \dots, \lambda_1 + (\ell-1), \lambda_1, \lambda_2, \dots)$ is an $\ell$-core;

\item $\mu = (\lambda_2, \lambda_3, \dots)$ is an $\ell$-core.

\end{enumerate}
\end{lemma}

\begin{proof}
For $1 \leq i \leq r$, $\nu_i - \nu_{i+1} = \ell-1$, so the $i^{th}$ row of $\nu$ can never contain part of an $\ell$-rim hook. Because $\lambda$ is an $\ell$-core, $\nu$ cannot have an $\ell$-rim hook that is supported entirely on the rows below the $r^{th}$ row. Hence $\nu$ is an $\ell$-core.

For the second statement of the lemma, note the partition $\mu$ is simply $\lambda$ with its first row deleted. In particular, $h_{(a,b)}^\mu = h_{(a+1,b)}^\lambda$ for all $(a,b) \in \mu$, so that by Remark \ref{divisibility} it is an $\ell$-core.

\end{proof}


 We now construct a partition $\lambda$ from a triple of data $(\mu,r,\kappa)$ as follows.
 Let $\mu = (\mu_1, \dots , \mu_s)$ to be any $\ell$-core where $\mu_1 - \mu_2 \neq \ell - 1$. For an integer $r \geq 0$ we form a new $\ell$-core $\nu = (\nu_1, \dots \nu_r, \nu_{r+1}, \dots, \nu_{r+s})$ by 
attaching $r$ rows above $\mu$ so that:
$$\nu_r = \mu_1+ \ell-1,\, 
\nu_{r-1} = \mu_1+2(\ell-1),\,\dots \, ,\nu_1 = \mu_1 + r(\ell-1),$$ 
$$\nu_{r+i} = \mu_i \textrm{ for } i = 1, 2, \dots , s.$$
 By Lemma \ref{newcores}, $\nu$ is an $\ell$-core.

 Fix a partition $\kappa = (\kappa_1, \dots , \kappa_{r+1})$ with at most $(r+1)$ parts. Then the new partition $\lambda$ is obtained from $\nu$ by adding $\kappa_i$ horizontal $\ell$-rim hooks to row $i$ for every $i \in \{1, \dots, r+1\}$. In other words $\lambda_i = \nu_i + \ell \kappa_i$ for $i \in \{ 1, 2, \dots, r+1\}$ and $\lambda_i  = \nu_i$ for $i>r+1$. 
 
 From now on, when we associate $\lambda$ with the triple $(\mu,r, \kappa)$, we will think of $\mu \subset \lambda$ as embedded in the rows below the $r^{th}$ row in $\lambda$. We introduce the notation $\lambda \approx (\mu, r, \kappa)$ for this decomposition.

 \begin{theorem}\label{construction_theorem}
 Let $\mu$, $r$ and $\kappa$ be as above. Then $\lambda \approx (\mu,r,\kappa)$ is an $\ell$-partition. Conversely, every $\ell$-partition corresponds uniquely to a triple $(\mu,r,\kappa)$.
 \end{theorem}
 \begin{proof}
 
Suppose $\lambda \approx (\mu,r, \kappa)$ were not an $\ell$-partition. Then after removal of some number of horizontal $\ell$-rim hooks we obtain a partition $\rho$ which has a removable non-horizontal $\ell$-rim hook. Note that for $1 \leq i \leq r$, $\lambda_i-\lambda_{i+1} \equiv -1 \mod \ell$, and likewise  $\rho_i-\rho_{i+1} \equiv -1 \mod \ell$. Suppose the non-horizontal $\ell$-rim hook had its rightmost topmost box in the $j^{th}$ row of $\rho$. Necessarily it is the rightmost box in that row. Clearly we must have $j\leq r$ since $\mu$ is an $\ell$-core. If $\rho_j -\rho_{j+1} > \ell-1$ then this $\ell$-rim hook must lie entirely in the $j^{th}$ row, i.e. be horizontal. If $\rho_j - \rho_{j+1} = \ell-1$ then the $\ell$-rim hook is clearly not removable.
 
 Conversely, if $\lambda$ is an $\ell$-partition, then let $\kappa_i$ denote the number of removable horizontal $\ell$-rim hooks which must be removed from row $i$ to obtain the $\ell$-core $\nu$ of $\lambda$. Let $r$ denote the index of the first row for which $\nu_r - \nu_{r+1} \neq \ell-1$. Let $\mu = (\nu_{r+1}, \dots)$. Then $\lambda \approx (\mu,r,\kappa)$. 
 \end{proof}
 
\begin{example}

 For $\ell = 3$, $\mu = (2,1,1)$ is a 3-core with $\mu_1-\mu_2 \neq 2$. We may add three rows ($r=3$) to it to obtain $\nu = (8,6,4,2,1,1)$, which is still a 3-core. Now we may add three horizontal $\ell$-rim hooks to the first row, three to the second, one to the third and one to the fourth ($\kappa = (3,3,1,1)$) to obtain the partition $\lambda = (17,15,7,5,1,1)$, which is a 3-partition.
$$\begin{array}{cc}
	\mu =\,\,\,\,\,\,\,\, 
	\tableau{
	         4&1\\
		2\\
		1}\,\,\,\,\,\,\,\,\,&
\linethickness{2.5pt}
\put (40,-37){\line(1,0){37}}
\put (40,-92){\line(0,1){56}}
\put (39, -91){\line(1,0){20}}
\put (58,-92){\line(0,1){38}}
\put (58, -55){\line(1,0){19}}
\put (76,-55){\line(0,1){19}}
	\nu = \,\,\,\,\,\,\,\,\,\tableau{13&10&8&7&5&4&2&1\\
         		10&7&5&4&2&1\\
		7&4&2&1\\
		4&1\\
		2\\
		1} \end{array}$$
$$\begin{array}{c} \mu \subset \nu \textrm{ is highlighted. } \end{array}$$
		
		$$\begin{array}{lr} \lambda = &
\linethickness{2.5pt} 
\put (306,0){\line(0,1){18}}
\put (252,0){\line(0,1){18}}
\put (198,0){\line(0,1){18}}
\put (144,0){\line(0,1){18}}
\put (270,-18){\line(0,1){18}}
\put (216,-18){\line(0,1){18}}
\put (162,-18){\line(0,1){18}}
\put (108,-18){\line(0,1){18}}
\put (126,-36){\line(0,1){18}}
\put (72,-36){\line(0,1){18}}
\put (90,-54){\line(0,1){18}}
\put (36,-54){\line(0,1){18}}
\put (36,-54){\line(1,0){54}}
\put (36,-36){\line(1,0){90}}
\put (72,-18){\line(1,0){198}}
\put (108,0){\line(1,0){198}}
\put (144,18){\line(1,0){162}}
	\tableau{22&19&18&17&16&14&13&11&10&9&8&7&6&5&4&2&1\\
			19&16&15&14&13&11&10&8&7&6&5&4&3&2&1\\
			10&7&6&5&4&2&1\\
			7&4&3&2&1\\
			2\\
			1}\end{array}
$$
$$\begin{array}{c} \textrm{The 3-partition constructed above, with cells contributed by }\kappa \textrm{ highlighted.}\end{array}$$
\end{example}

\begin{remark}
In the proofs of Theorems \ref{top_and_bottom} and \ref{second_from_bottom} we will prove that a partition is an $\ell$-partition by giving its decomposition into $(\mu,r,\kappa)$.
\end{remark}

\subsection{Counting l-partitions}
We derive a closed formula for our generating function $B_{\ell}$ by using our $\ell$-partition decomposition described above. First we note that if $$\begin{array}{lcr} \displaystyle C_{\ell}(x) = \sum_{\mu \in \mathcal{C}_{\ell}} x^{\mu_1} & \textrm{then}  & \displaystyle x^{\ell-1} C_{\ell}(x) = \sum_{\mu \in \mathcal{C}_{\ell} \,:\, \mu_1-\mu_2 \,= \,\ell -1} x^{\mu_1}. \end{array}$$ Therefore, $\displaystyle \sum_{\mu \in \mathcal{C}_{\ell} : \, \mu_1-\mu_2 \,\neq \,\ell-1} x^{\mu_1} = (1-x^{\ell-1}) C_{\ell}(x)$.  Hence the generating function for all cores $\mu$ with $\mu_1-\mu_2 \neq \ell-1$ is $\displaystyle \frac{1-x^{\ell-1}}{(1-x)^{\ell-1}}$.

We are now ready to prove our closed formula for $B_{\ell}$. First, recall that $B_{\ell}(x) = \sum b_n^{\ell} x^n$ where $b_n^{\ell}$ = $\{ \lambda \in \mathcal{L}_{\ell} : \lambda_1 = n\}$.
\begin{theorem} $$B_{\ell}(x) = \frac{1-x^{\ell-1}}{(1-x)^{\ell -1}(1-x^{\ell-1} - x^{\ell})}.$$
\end{theorem}
\begin{proof}
We will follow our construction of $\ell$-partitions from Section \ref{construct}. Note that if $\lambda \approx (\mu, r, \kappa)$, then the first part of $\lambda$ is ${\mu_1+\ell \kappa_1 + r (\ell-1)}$. 
Hence $\lambda$ contributes $x^{\mu_1+\ell \kappa_1 + r (\ell-1)}$ to $B_{\ell}$.

Fix a core $\mu$ with $\mu_1-\mu_2 \neq \ell-1$. Let $r$ and $\kappa_1$ be fixed non-negative integers. Let $\gamma_{r, \kappa_1}$ be the number of partitions with first part $\kappa_1$ and length less than or equal to $r+1$. $\gamma_{r, \kappa_1}$ counts the number of $\ell$-partitions with $r$ and $\kappa_1$ fixed that can be constructed from $\mu$. Note that $\gamma_{r, \kappa_1}$ is independent of what $\mu$ is. 

$\gamma_{r,\kappa_1}$ is the same as the number of partitions which fit inside a box of size height $r$ and width $\kappa_1$. Hence $\gamma_{r, \kappa_1} = \binom{r+\kappa_1}{r}$. Fixing $\mu$ and $r$ as above, the generating function for the number of $\ell$-partitions with core $(\mu,r,\emptyset)$ with respect to the number of boxes added to the first row is $\displaystyle \sum_{\kappa_1=0}^{\infty} \gamma_{r,\kappa_1} x^{\kappa_1 \ell} = \frac{1}{(1-x^{\ell})^{r+1}}$.

Now for a fixed $\mu$ as above, the generating function for the number of $\ell$-partitions which can be constructed from $\mu$ with respect to the number of boxes added to the first row is $\displaystyle \sum_{r=0}^{\infty} x^{r (\ell-1)} \frac{1}{(1-x^{\ell})^{r+1}}$. Multiplying through by $1-x^{\ell-1} - x^{\ell}$, one can check that

 $$\sum_{r=0}^{\infty} x^{r (\ell-1)} \frac{1}{(1-x^{\ell})^{r+1}} = \frac{1}{1-x^{\ell-1} - x^{\ell}}.$$

 Therefore $B_{\ell}(x)$ is just the product of the two generating functions $(1-x^{\ell-1})C_{\ell}(x)$ and $\frac{1}{1-x^{\ell-1} - x^{\ell}}$. Hence $$B_{\ell}(x) = \frac{1-x^{\ell-1}}{(1-x)^{\ell -1}(1-x^{\ell-1} - x^{\ell})}.$$

\end{proof}

\subsection{Counting l-partitions of a fixed weight and fixed core} Independently of the authors, Cossey, Ondrus and Vinroot have a similar construction of partitions associated with irreducible representations. In \cite{COV}, they gave a construction analogous to our construction on $\ell$-partitions from Section \ref{construct} for the case of the symmetric group over a field of characteristic $p$. After reading their work and noticing the similarity to our own, we decided to include the following theorem, which is an analogue of their theorem for symmetric groups. The statement is a direct consequence of our construction, so no proof will be included.
\begin{theorem}
For a fixed core $\nu$ satisfying $\nu_i-\nu_{i+1} = \ell-1$ for $i = 1,2, \dots r$ and $\nu_{r+1}-\nu_{r+2} \neq \ell-1$, the number of $\ell$-partitions of a fixed weight $w$ is the number of  partitions of $w$ with at most $r+1$ parts. 
The generating function for the number of $\ell$-partitions of a fixed $\ell$-core $\nu$ with respect to the statistic of the weight of the partition is thus $\displaystyle \prod_{i=1}^{r+1} \frac{1}{1-x^i}$. Hence the generating function for all $(\ell,0)$-Carter partitions with fixed core $\nu$ with respect to the statistic of the size of the partition is 
$\displaystyle \sum_{\lambda\in \mathcal{L}_{\ell} \textrm{ with core } \nu} x^{|\lambda|} = x^{|\nu|} \prod_{i=1}^{r+1} \frac{1}{1-x^{\ell i}}$.

\end{theorem}
\begin{example}
Let $\ell = 3$ and let $\nu = (6,4,2,1,1) \approx ((2,1,1),2,\emptyset)$ be a 3-core. Then the number of 3-partitions of weight $5$ with core $\nu$ is exactly the number of partitions of $5$ into at most $3$ parts. There are 5 such partitions ($(5)$, $(4,1)$, $(3,2)$, $(3,1,1)$, $(2,2,1)$). Therefore, there are $5$ such $\ell$-partitions. They are $(21,4,2,1,1),$  $(18,7,2,1,1),$  $(15,10,2,1,1),$  $(15,7,5,1,1)$ and $(12,10,5,1,1)$. 

$$\begin{array}{cc} \nu = & \tableau{\mbox{}&\mbox{}&\mbox{}&\mbox{}&\mbox{}&\mbox{}\\
\mbox{}&\mbox{}&\mbox{}&\mbox{}\\
\mbox{}&\mbox{}\\
\mbox{}\\
\mbox{}&
}
\end{array}$$
\end{example}

$$\begin{array}{c} \textrm{For } \nu \textrm{ above, } r = 2, \textrm{ so horizontal 3-rim hooks can be added} \\ \textrm{ to the first three rows.} \end{array}$$

\section{The crystal of the basic representation}\label{??}

There is a crystal graph structure on the set of $\ell$-regular partitions (see Section \ref{crystal_description} for a description of $B(\Lambda_0)$). In this section we examine the properties of $\ell$-partitions in this crystal graph. In particular, just as one can obtain $\ell$-cores by following an $i$-string from another $\ell$-core, one can obtain $\ell$-partitions by following an $i$-string from  another $\ell$-partition. However, more $\ell$-partitions exist. Sometimes $\ell$-partitions live one position before the end of an $i$-string (but nowhere else). In particular, there is not a nice interpretation in terms of the braid group orbits of nodes. In Theorem \ref{second_from_bottom} we characterize (combinatorially) where this happens.

\subsection{Crystal operators and l-partitions}\label{carterpartitioncrystal}
We first recall some well-known facts about the behavior of $\ell$-cores in this crystal graph $B(\Lambda_0)$. There is an action of the affine Weyl group $\widetilde{S_{\ell}}$ on the crystal such that the simple reflection $s_i$ reflects each $i$-string. In other words, $s_i$ sends a node $\lambda$ to 
$$\left\{ \begin{array}{lr}
\displaystyle
			  \widetilde{f}_{i}^{\varphi_i(\lambda) - \varepsilon_i(\lambda)} \lambda \;\;\;\;&  \varphi_i(\lambda) - \varepsilon_i(\lambda) >0\\
			  \widetilde{e}_{i}^{\varepsilon_i(\lambda) - \varphi_i(\lambda)} \lambda \;\;\;\;&  \varphi_i(\lambda) - \varepsilon_i(\lambda) <0\\
			  \lambda \;\;\;\; & \varphi_i(\lambda) - \varepsilon_i(\lambda) =0
			\end{array} \right.$$ 
			
The $\ell$-cores are exactly the $\widetilde{S_{\ell}}$-orbit of $\emptyset$, the highest weight node. This implies the following Proposition.
\begin{proposition} If $\mu$ is an $\ell$-core and $\varepsilon_i(\mu) \neq 0$ then $\varphi_i(\mu) = 0$ and $\widetilde{e}_{i}^{\varepsilon_i(\mu)}\mu$ is again an $\ell$-core. Furthermore, all $\widetilde{e}_{i}^{k}\mu$ for $0<k<\varepsilon_i(\mu)$ are not $\ell$-cores. Similarly, if $\varphi_i(\mu) \neq 0$ then $\varepsilon_i(\mu) = 0$ and $\widetilde{f}_{i}^{\varphi_i(\mu)}\mu$ is an $\ell$-core but $\widetilde{f}_i^{k} \mu$ is not for $0 < k < \varphi_i(\mu)$.
\end{proposition}

In this chapter, given an $\ell$-partition $\lambda$, we will determine when $\widetilde{f}_{i}^{k}\lambda$ and $\widetilde{e}_i^k \lambda$ are also $\ell$-partitions.

The following remark will help us in the proofs of the upcoming Theorems \ref{top_and_bottom}, \ref{other_cases} and \ref{second_from_bottom}. 
\begin{remark}\label{residue_remark} Suppose $\lambda$ is a partition. Consider its Young diagram. If any $\ell$-rim hook has an upper rightmost box of residue $i$, then the lower leftmost box has residue $i+1\mod \ell$. Conversely, a hook length $h_{(a,b)}^{\lambda}$ is divisible by $\ell$ if and only if there is an $i$ so that the rightmost box of row $a$ has residue $i$, and the lowest box of column $b$ has residue $i+1\mod \ell$.
\end{remark}
In Lemma \ref{adding_to_l_partition} we will generalize Proposition \ref{adding_to_core} to $\ell$-partitions.
\begin{proposition}\label{adding_to_core}
Let $\lambda$ be an $\ell$-core, and suppose $0\leq i <\ell$. Then the $i$-signature for $\lambda$ is the same as the reduced $i$-signature.
\end{proposition}
\begin{lemma}\label{adding_to_l_partition} Let $\lambda$ be an $\ell$-partition, and suppose $0 \leq i < \ell$. Then the $i$-signature for $\lambda$ is the same as the reduced $i$-signature.
\end{lemma}
\begin{proof}
We need to show that there does not exist positions $(a,b)$ and $(c,d)$ such that $(a,b)$ is an addable $i$-box, $(c,d)$ is a removable $i$-box, and $c>a$. But if this were the case, then the hook length $h_{(a,d)}^\lambda$ would be divisible by $\ell$ (by Remark \ref{residue_remark}), but $\ell$ does not divide $h_{(c,d)}^\lambda =1$. Then $\lambda$ would violate $(\star)$, so it would not be an $\ell$-partition.
\end{proof}

\begin{remark}\label{adding_to_l_partition_remark}
As a consequence of Lemma \ref{adding_to_l_partition}, the action of the operators $\widetilde{e}_i$ and $\widetilde{f}_i$ are simplified in the case of $\ell$-partitions. Now applying successive $\widetilde{f}_{i}$'s to $\lambda$ corresponds to adding all addable boxes of residue $i$ from right to left. Similarly, applying successive $\widetilde{e}_{i}$'s to $\lambda$ corresponds to removing all removable boxes of residue $i$ from left to right.
\end{remark}
In the following Theorems \ref{top_and_bottom}, \ref{other_cases} and \ref{second_from_bottom}, we implicitly use Remark \ref{residue_remark} to determine when a hook length is divisible by $\ell$, and  Remark \ref{adding_to_l_partition_remark} when applying $\widetilde{e}_{i}$ and $\widetilde{f}_{i}$ to $\lambda$. 

\begin{remark} Suppose $\lambda \approx (\mu,r,\kappa)$. When viewing $\mu$ embedded in $\lambda$, we note that if a box $(a,b) \in \mu \subset \lambda$ has residue $i \mod \ell$ in $\lambda$, then it has residue $i-r \mod \ell$ in $\mu$. 
\end{remark}
Let $\lambda = (\lambda_1, \lambda_2, \dots )$ be a partition, and $r$ be any integer. We define $\overline{\lambda} = (\lambda_2, \lambda_3, \dots  )$, $\hat{\lambda} = (\lambda_1, \lambda_1, \lambda_2, \lambda_3, \dots )$ and $\lambda+1^r = (\lambda_1+1, \lambda_2+1, \dots, \lambda_r+1, \lambda_{r+1}, \dots )$, extending $\lambda$ by $r-len(\lambda)$ parts of size $0$ if $r > len(\lambda)$. We note that Lemma \ref{newcores} implies that $\overline{\lambda}$ is an $\ell$-core.

\subsection{l-partitions in the crystal}

\begin{theorem}\label{top_and_bottom}
Suppose that $\lambda$ is an $\ell$-partition and $0\leq i < \ell$. Then
\begin{enumerate}
\item\label{f} $\widetilde{f}_{i}^{\varphi} \lambda$ is an $\ell$-partition,
\item\label{e} $\widetilde{e}_{i}^{\varepsilon} \lambda$ is an $\ell$-partition.
\end{enumerate}
\end{theorem}

\begin{proof} We will prove only \eqref{f}, as \eqref{e} is similar. Recall all addable $i$-boxes of $\lambda$ are conormal by Lemma \ref{adding_to_l_partition}. The proof of \eqref{f} relies on the decomposition of the $\ell$-partition as in Section \ref{construct}.  Let $\lambda \approx (\mu, r, \kappa)$. We break the proof of \eqref{f} into three cases:
\begin{enumerate}
    \renewcommand{\labelenumi}{(\alph{enumi})}

\item If the first row of $\mu$ embedded in $\lambda$ does not have an addable $i$-box then we cannot add an $i$-box to the first $r+1$ rows of $\lambda$. Hence $\varphi = \varphi_{i-r}(\mu)$. Therefore the decomposition of $\widetilde{f}_{i}^{\varphi} \lambda$ has the same $r$ and $\kappa$ as $\lambda$, and $\mu$ will be replaced by $\widetilde{f}_{{i-r}}^{\varphi} \mu$, 
which is still a core by Proposition \ref{adding_to_core}. Hence $\widetilde{f}_{i}^{\varphi} \lambda \approx ( \widetilde{f}_{{i-r}}^{\varphi} \mu, r, \kappa)$.

 \item If the first row of $\mu$ embedded in $\lambda$ does have an addable $i$-box and $\mu_1-\mu_2 < \ell-2$, then the first $r+1$ rows of $\lambda$ have addable $i$-boxes. Also some rows of $\mu$ will have addable $i$-boxes. $\widetilde{f}_i^{\varphi}$ adds an $i$-box to the first $r$ rows of $\lambda$, plus adds to any addable $i$-boxes in the core $\mu$. Note that $\varphi_{i-r}(\mu) = \varphi -r$.  Since $\mu_1 - \mu_2 < \ell-2$, the first and second rows of $\widetilde{f}_{i-r}^{\varphi-r} \mu$ differ by at most $\ell-2$. Therefore $\widetilde{f}_{i}^{\varphi} \lambda \approx (\widetilde{f}_{{i-r}}^{\varphi-r} \mu, r, \kappa) $.

\item If the first row of $\mu$ embedded in $\lambda$ does have an addable $i$-box and $\mu_1-\mu_2 = \ell-2$, then $\widetilde{f}_i^{\varphi}$ will add the addable $i$-box in the $r+1^{st}$ row (i.e. the first row of $\mu$). Since the $(r+2)^{nd}$ row does not have an addable $i$-box, we know that the $(r+1)^{st}$ and $(r+2)^{nd}$ rows of $\widetilde{f}_i^{\varphi}(\lambda)$ differ by $\ell-1$. Therefore $\widetilde{f}_{i}^{\varphi} \lambda \approx (\overline{\widetilde{f}_{{i-r}}^{\varphi-r} \mu}, r+1, \kappa)$ is an $\ell$-partition.
 \end{enumerate}
\end{proof}

\begin{lemma}
Let $\lambda$ be an $\ell$-partition. Then $\lambda$ cannot have one normal box and two conormal boxes of the same residue.
\end{lemma}

\begin{proof}\label{lemmaforproofofothercases}
Label any two of the conormal boxes $n_1$ and $n_2$, with $n_1$ to the left of $n_2$. Pick any normal box and label it $n_3$.  By Lemma \ref{adding_to_l_partition}, $n_3$ must lie to the right of $n_1$. Then the hook length in the column of $n_1$ and row of $n_3$ is a multiple of $\ell$, but the hook length in the column of $n_1$ and row of $n_2$ is not a multiple of $\ell$ by Remark \ref{residue_remark}.
\end{proof}

\begin{theorem}\label{other_cases}
Suppose that $\lambda$ is an $\ell$-partition. Then
\begin{enumerate}
\item\label{f_theorem} $\widetilde{f}_{i}^{k} \lambda$ is not an $\ell$-partition for $0<k < \varphi -1,$
\item\label{e_theorem} $\widetilde{e}_{i}^k \lambda$ is not an $\ell$-partition for $1<k< \varepsilon$.
\end{enumerate}
\end{theorem}
\begin{proof}

If $0< k < \varphi-1$ then there are at least two conormal $i$-boxes in $\widetilde{f}_{i}^{k} \lambda$ and at least one normal $i$-box.  By Lemma \ref{lemmaforproofofothercases}, $\widetilde{f}_{i}^{k} \lambda$ is not an $\ell$-partition. The proof of \eqref{e_theorem} is similar to that of \eqref{f_theorem}.
\end{proof}

The above theorems told us the position of an $\ell$-partition relative to the $i$-string which it sits on in the crystal $B(\Lambda_0)$. If an $\ell$-partition occurs on an $i$-string, then both ends of the $i$-string are also $\ell$-partitions. Furthermore, the only places $\ell$-partitions can occur are at the ends of $i$-strings or possibly one position before the final node. The next theorem describes when this latter case occurs.

\begin{theorem}\label{second_from_bottom}
Suppose that $\lambda \approx (\mu,r,\kappa)$ is an $\ell$-partition. Then

\begin{enumerate}
\item\label{first}
If $\varphi > 1$ then $\widetilde{f}_{i}^{\varphi-1} \lambda$ is an $\ell$-partition if and only if:

 $(\dagger) \,\,\,\,\,\,\,\,\,\,\,\kappa_{r+1} = 0$, the first row of $\lambda$ has a conormal  $i$-box, and  $\varphi =r+1.$

\item\label{second} If $\varepsilon > 1$ then $\widetilde{e}_{i} \lambda$ is an $\ell$-partition if and only if:

\begin{center}
$ (\ddagger)\,\,\,\,\,\,\,\,\,\,\,$ the first row of  $\lambda$ has a conormal $(i+1)$-box and either 

$ \varepsilon = r  $ and $ \kappa_r = 0$,  or $ \varepsilon = r+1$  and  $\kappa_{r+1} = 0 .$

\end{center}
\end{enumerate}
\end{theorem}

\begin{proof}
We first prove \eqref{first} and then show how \eqref{second} can be derived from \eqref{first}. 

If $\lambda$ satisfies condition $(\dagger)$ then $\lambda$ differs from $\widetilde{f}_{i}^{\varphi-1} \lambda$ by one box in each of the first $r$ rows. Hence $(\widetilde{f}_{i}^{\varphi-1} \lambda)_r - (\widetilde{f}_{i}^{\varphi-1} \lambda)_{r+1}$ is a multiple of $\ell$, so that the first $r$ rows each have one more horizontal $\ell$-rim hook than they had in $\lambda$. After removing these horizontal $\ell$-rim hooks, we get the partition $(\widehat{\mu}, r-1, \emptyset)$. This decomposition is valid, as we will now show $\widehat{\mu}$ is an $\ell$-core. 
Since $\varphi_i(\mu)=1$, $\widetilde {f_i} \mu$ 
is also an $\ell$-core and so in particular $\ell \nmid
h_{(1,b)}^{\widetilde {f_i} \mu}$
for $1 \le b \le \mu_1+1$.
Note that $h_{(1,b)}^{\widehat  \mu} =
h_{(1,b)}^{\widetilde {f_i} \mu}=
h_{(1,b)}^{\mu} +1 $
for  $1 \le b \le \mu_1$, and for $a > 1$,
$h_{(a,b)}^{\widehat  \mu} =
h_{(a,b)}^{\mu}$, yielding $\ell \nmid h_{(a,b)}^{\widehat  \mu}$ for
all
boxes $(a,b) \in \widehat  \mu$.  By Remark \ref{divisibility},
$\widehat  \mu$ is an $\ell$-core. It is then easy to see that  $\widetilde{f}_{i}^{\varphi-1} \lambda \approx (\widehat{\mu}, r-1, \kappa+1^r)$, so therefore $\widetilde{f}_{i}^{\varphi-1} \lambda$ is an $\ell$-partition.

Conversely:
\begin{enumerate} 
    \renewcommand{\labelenumi}{(\alph{enumi})}
\item If the first part of $\lambda$ has a conormal $j$-box, with $j \neq i$, call this box $n_1$. If $j = i+1$ then the box $(r+1, \lambda_{r+1})$ has residue $i$. If an addable $i$-box exists, say at $(a,b)$, it must be below the first $r+1$ rows. But then the hook length $h_{(r+1,b)}^{\mu}$ is divisible by $\ell$. This implies that $\mu$ is not a core. So we assume $j \neq i+1$. Then $\widetilde{f}_{i}^{\varphi-1} \lambda$ has at least one normal $i$-box $n_2$ and exactly one conormal $i$-box $n_3$ with $n_3$ left of $n_2$ left of $n_1$. The hook length of the box in the column of $n_3$ and row $n_2$ is divisible by $\ell$, but the hook length in the box of column $n_3$ and row $n_1$ is not (by Remark \ref{residue_remark}). By Theorem \ref{maintheorem}, $\widetilde{f}_{i}^{\varphi-1} \lambda$ is not an $\ell$-partition. 

\item By (a), we can assume that the first row has a conormal $i$-box. If $\varphi \neq r+1$ then row $r+2$ of  $\widetilde{f}_{i}^{\varphi-1} \lambda$ will end in a $j$-box, for some $j \neq i$. Call this box $n_1$. Also let $n_2$ be any normal $i$-box in $\widetilde{f}_{i}^{\varphi-1} \lambda$ and $n_3$ be the unique conormal $i$-box. Then the box in the row of $n_1$ and column of $n_3$ has a hook length which is not divisible by $\ell$, but the box in the row of $n_2$ and column of $n_3$ has a hook length which is (by Remark \ref{residue_remark}). By Theorem \ref{maintheorem}, $\widetilde{f}_{i}^{\varphi-1} \lambda$ is not an $\ell$-partition.

\item Suppose $\kappa_{r+1} \neq 0$. By $(a)$ and $(b)$, we can assume that $\varphi = r+1$ and that the first row of $\lambda$ has a conormal $i$-box. Then the difference between $\lambda$ and $\widetilde{f}_{i}^{\varphi-1} \lambda = \widetilde{f}_i^r \lambda$ is an added box in each of the first $r$ rows. Remove $\kappa_r - \kappa_{r+1} +1$ horizontal $\ell$-rim hooks from row $r$ of $\widetilde{f}_i^{\varphi-1} \lambda$. Call the remaining partition $\nu$. Then $\nu_r = \nu_{r+1} = \mu_1 + \ell \kappa_{r+1}$. Hence a removable non-horizontal $\ell$-rim hook exists in $\nu$ taking the rightmost box from row $r$ with the rightmost $\ell-1$ boxes from row $r+1$. Thus $\widetilde{f}_{i}^{\varphi-1}(\lambda)$ is not an $\ell$-partition.
\end{enumerate}

To prove \eqref{second}, we note that by Theorem \eqref{other_cases} that if $\varphi \neq 0$ and $\varepsilon>1$ then $\widetilde{e}_{i} \lambda = \widetilde{e}_{i}^2 \widetilde{f}_{i} \lambda$ cannot be an $\ell$-partition. Hence we only consider $\lambda$ so that $\varphi_i(\lambda) = 0$. But then $\widetilde{f}_{i}^{\varphi_i(\widetilde{e}_{i}^{\varepsilon}(\lambda))-1} \widetilde{e}_{i}^{\varepsilon} \lambda = \widetilde{f}_{i}^{\varepsilon-1} \widetilde{e}_{i}^{\varepsilon} \lambda = \widetilde{e}_{i} \lambda$. From this observation, it is enough to show that $\lambda$ satisfies $(\ddagger)$ if and only if $\widetilde{e}_{i}^{\varepsilon} \lambda$ satisfies ($\dagger$). The proof of this follows a similar line as the above proofs, so it will be left to the reader.
\end{proof}

\begin{example} Fix $\ell = 3$. Let $\lambda = (9,4,2,1,1) \approx ((2,1,1), 2, (1))$.
$$
\begin{array}{cc} \lambda = &
\tableau{0&1&2&0&1&2&0&1&2 \\ 2&0&1&2\\1&2\\0\\2}
\end{array}$$ 
Here $\varphi_0 (\lambda) = 3$. $\widetilde{f}_{0} \lambda = (10,4,2,1,1)$ is not a $3$-partition, but $\widetilde{f}_{0}^{2} \lambda = (10,5,2,1,1) \approx ((2,2,1,1), 1, (2,1))$ and $\widetilde{f}_{0}^{3} \lambda = (10,5,3,1,1) \approx ((1,1), 3, (1))$ are $3$-partitions.
\end{example}

\section{A representation-theoretic proof of Theorem \ref{top_and_bottom}}\label{new_proof}
This proof relies heavily on the work of Grojnowski, Kleshchev et al. We recall some notation from \cite{G} but repeat very few definitions below. 

\subsection{Definitions and preliminaries}
In the category $Rep_n$ of finite-dimensional representations of the finite Hecke algebra $H_n(q)$, we define the Grothendieck group $K(Rep_n)$ to be the group generated by isomorphism classes of finite-dimensional representations, with relations $[\mathcal{M}_1] +[\mathcal{M}_3] = [\mathcal{M}_2]$ if there exists an exact sequence $0 \to \mathcal{M}_1 \to \mathcal{M}_2 \to \mathcal{M}_3 \to 0$. This is a finitely generated abelian group with generators corresponding to the irreducible representations of $H_n(q)$. The equivalence class corresponding to the module $\mathcal{M}$ is denoted $[\mathcal{M}]$.

Just as $S_n$ can be viewed as the subgroup of $S_{n+1}$ consisting of permutations which fix $n+1$, $H_n(q)$ can be viewed as a subalgebra of $H_{n+1}(q)$ (the generators $T_1, T_2, \dots ,T_{n-1}$ generate a subalgebra isomorphic to $H_n(q)$). Let $\mathcal{M}$ be a finite-dimensional representation of $H_{n+1}(q)$. Then it makes sense to view $\mathcal{M}$ as a representation of $H_n(q)$. This module is called the \textit{restriction of }$\mathcal{M}$\textit{ to } $H_n(q)$, and is denoted $Res^{H_{n+1}(q)}_{H_n(q)} \mathcal{M}$.
Similarly, we can define the induced representation of $\displaystyle \mathcal{M}$ by $ Ind^{H_{n+1}(q)}_{H_n(q)} \mathcal{M} = H_{n+1}(q) \otimes_{H_n(q)} \mathcal{M}$. Just as $S_b \subset S_a$, we can also consider $H_b(q) \subset H_a(q)$ and define corresponding restriction and induction functors.
To shorten notation, $Res^{H_{a}(q)}_{H_b(q)}$ will be written $Res_b^a$, and $Ind^{H_{a}(q)}_{H_b(q)}$ will be written as $Ind_b^{a}$.

If $\lambda$ and $\mu$ are partitions, it is said that $\mu$ covers $\lambda$, (written $\mu \succ \lambda$) if the Young diagram of $\lambda$ is contained in the Young diagram of $\mu$ and $|\mu| = |\lambda| +1$. 

The following proposition  is well known and can be found in \cite{M}. 
\begin{proposition}\label{branching_rule} Let $\lambda$ be a partition of $n$ and $S^{\lambda}$ be the Specht module corresponding to $\lambda$. Then $$\displaystyle [Ind^{n+1}_{n} S^{\lambda}] =  \sum_{\mu \succ \lambda} [S^{\mu}].$$
\end{proposition}
We consider functors $\widetilde{e}_{i}: Rep_n \to Rep_{n-1}$ and $\widetilde{f}_{i}: Rep_n \to Rep_{n+1}$ which commute with the crystal action on partitions in the following sense (see \cite{G} for definitions and details).
\begin{theorem}\label{tilda} Let $\lambda$ be an $\ell$-regular partition. Then:
\begin{enumerate}

\item $\widetilde{e}_{i} D^{\lambda} = D^{\widetilde{e}_{i} \lambda}$;
\item\label{f_tilda} $\widetilde{f}_{i} D^{\lambda} = D^{\widetilde{f}_{i} \lambda}$.
\end{enumerate}
\end{theorem}

We now consider the functors $f_i: Rep_n \to Rep_{n+1} $ and $e_i: Rep_n  \to Rep_{n-1} $ which refine induction and restriction (for a definition of these functors, especially in the more general setting of cyclotomic Hecke algebras, see \cite{G}). For a representation $\mathcal{M} \in Rep_n $ let $\varepsilon_i (\mathcal{M}) = \max\{k : {e}_{i} ^k \mathcal{M} \neq 0\}$ and 
$\varphi_i (\mathcal{M}) = \max\{k : {f}_{i}^k \mathcal{M} \neq 0\}$. Grojnowski concludes the following theorem.
\begin{theorem}\label{groj} Let $\mathcal{M}$ be a finite-dimensional representation of $H_n(q)$. Let $\varphi = \varphi_i (\mathcal{M})$ and $\varepsilon = \varepsilon_i (\mathcal{M})$.
\begin{enumerate}
\item\label{ind} $Ind_{n}^{n+1} \mathcal{M} =
 \bigoplus_i f_i  \mathcal{M}; \,\,\,\,\,\:\:\:\;\;\;\;\;\;  Res_{n}^{n+1} \mathcal{M} = 
 \bigoplus_i e_i \mathcal{M};$
\item\label{f^phi} $[f_i^{\varphi} \mathcal{M}] = \varphi ! [\widetilde{f}_i^{\varphi} \mathcal{M}]; \,\,\,\,\,\,\,\,\,\,\,\,\:\:\: \,\,\,\,\,\,\,\:\:\: \,\,\,\,\,\,\,\:\:\: [e_i^{\varepsilon} \mathcal{M}] = \varepsilon ! [\widetilde{e}_i^{\varepsilon} \mathcal{M}]. $ 
\end{enumerate}
\end{theorem}

For a module $D^{\mu}$ the \textit{central character} $\chi(D^{\mu})$ can be identified with the multiset of residues of the partition $\mu$. The following theorem allows us to define $\chi(S^{\mu})$ as well.
\begin{theorem}\label{central}
All composition factors of the Specht module $S^{\lambda}$ have the same central character.
\end{theorem}
\begin{theorem}\label{character}
$\chi(f_i(D^{\lambda})) = \chi(D^{\lambda}) \cup \{ i \}; \;\;\;\;\;\;\;\; \chi(e_i(D^{\lambda})) = \chi(D^{\lambda}) \setminus  \{ i \}$. 
\end{theorem}

We are now ready to present a representation-theoretic proof of Theorem \ref{top_and_bottom}, which states that if $\lambda$ is an $\ell$-partition lying anywhere on an $i$-string in the crystal $B(\Lambda_0)$, then the extreme ends of the $i$-string through $\lambda$ are also $\ell$-partitions.
\subsection{A representation-theoretic proof of Theorem \ref{top_and_bottom}}

\begin{proof}[Alternate Proof of Theorem \ref{top_and_bottom}] Suppose $\lambda$ is an $\ell$-partition and $| \lambda| = n$. Recall that by the result of James and Mathas \cite{JM} combined with Theorem \ref{maintheorem}, $S^{\lambda} = D^{\lambda}$ if and only if $\lambda$ is an $\ell$-partition. Let $F$ denote the number of addable $i$-boxes of $\lambda$ and let $\nu$ denote the partition corresponding to $\lambda$ plus all addable $i$-boxes.

First, we induce $S^{\lambda}$ from $H_n(q)$ to $H_{n+F}(q)$. Applying  Proposition \ref{branching_rule} $F$ times yields $$\displaystyle [Ind_n^{n+F} S^{\lambda}] = \sum_{\mu_{F} \succ \mu_{F-1} \succ \dots \succ \mu_1 \succ \lambda} [S^{\mu_{F}}].$$
Note $[S^{\nu}]$ occurs in this sum with coefficient $F !$ (add the $i$-boxes in any order), and everything else in this sum has different central character than $S^{\nu}$. Hence the direct summand which has the same central character as $S^{\nu}$ (i.e. the central character of $\lambda$ with $F$ more $i$'s) is $F! [S^{\nu}]$ in $K(Rep_{n+F})$. 

We next  
apply \eqref{ind} from Theorem \ref{groj} $F$ times to obtain 
$$[Ind_n^{n+F} D^{\lambda}] = \bigoplus_{i_1, \dots, i_{F}} [f_{i_1} \dots f_{i_{F}} {D^{\lambda}}].$$

\sloppy{The direct summand with central character $\chi(S^{\nu})$ is $[f_i^F D^{\lambda}]$ in $K(Rep_{n+F})$. Since $\lambda$ is an $\ell$-partition, $S^{\lambda} = D^{\lambda}$, so $Ind_n^{n+F} S^{\lambda} = Ind_n^{n+f} D^{\lambda}$ and we have shown that $F![S^{\nu}] = [f_i^F D^{\lambda}]$.

}

Since $S^{\lambda} = D^{\lambda}$ and $f_i^{F} D^{\lambda} \neq 0$, we know that $F \leq \varphi$. Similarly, since $Ind_n^{n+F+1} S^{\lambda}$ has no composition factors with central character $\chi(S^{\lambda}) \cup \{ \underbrace{i, i, \dots , i}_{F+1}\},$  we know that $F \geq \varphi$. Hence $F = \varphi$. 

By part \ref{f^phi} of Theorem \ref{groj},  $[(f_i)^{\varphi} D^{\lambda}] = \varphi ! [\widetilde{f}_i^{\varphi} D^{\lambda}]$. Then by Theorem \ref{tilda}, $[S^{\nu}] = [D^{\widetilde{f}_i^{\varphi}\lambda}]$. 

Since $F=\varphi$, $\nu = \widetilde{f}_i^{\varphi} \lambda$, so in particular $\nu$ is $\ell$-regular and $S^{\nu} = D^\nu$. 
Hence $\widetilde{f}_{i}^{\varphi} \lambda = \nu$ is an $\ell$-partition.

  The proof that $\widetilde{e}_{i}^{\varepsilon_i(\lambda)} \lambda$ is an $\ell$-partition follows similarly, with the roles of induction and restriction changed in Proposition \ref{branching_rule}, and the roles of $e_i$ and $f_i$ changed in Theorem \ref{groj}.

\end{proof}

We do not yet have representation-theoretic proofs of our other Theorems \ref{other_cases} and \ref{second_from_bottom}. We expect an analogue of Theorem \ref{other_cases}  to be true for the Hecke algebra over a field of arbitrary characteristic. In Theorem \ref{second_from_bottom} the conditions $(\dagger)$ and $(\ddagger)$ will change for different fields, so any representation-theoretic proof of this theorem should distinguish between these different cases.


    %
    %

    \newchapter{JM Partitions}{The Ladder Crystal and JM Partitions}{The Ladder Crystal and JM Partitions}
    \label{sec:LabelForChapter3}

        

\section{Introduction}

The main goal of this chapter is to generalize results of Chapter \ref{sec:LabelForChapter2} 
to a larger class of partitions. An $\ell$-partition is an $\ell$-regular 
partitions for which the Specht module $S^{\lambda}$ is irreducible for the Hecke algebra 
$H_n(q)$ when $q$ is a primitive $\ell^{th}$ root of unity.  A generalized $\ell$-partition $\lambda$ 
is an $\ell$-regular partition for which there is a (not necessarily $\ell$-regular) partition $\mu$ 
so that $S^{\mu} = D^{\lambda}$. We noticed that within the crystal  $B(\Lambda_0)$ the generalized 
$\ell$-partitions satisfied rules similar to the rules given in Chapter \ref{sec:LabelForChapter2} for 
$\ell$-partitions. In order to prove this, we built a modified version of the crystal $B(\Lambda_0)$,
 which we labeled $B(\Lambda_0)^L$. The crystal $B(\Lambda_0)^L$ satisfies the following properties:
\begin{itemize}
\item The nodes of $B(\Lambda_0)^L$ are partitions, and there is an $i$-arrow from $\lambda$ to $\mu$ only when the difference $\mu \setminus \lambda$ is a box of residue $i$. 
\item If $\lambda$ is a weak $\ell$-partition with $D^{\lambda} = S^{\nu}$ for some partition $\nu$ then $\nu$ is a node of $B(\Lambda_0)^L$.
\item $B(\Lambda_0) \cong B(\Lambda_0)^L$ and this crystal  isomorphism yields an interesting 
bijection on the nodes.  The map being used for the isomorphism has been well studied \cite{JK}, but 
never before in the context of a crystal isomorphism.

\item There exists elementary combinatorial arguments which generalize the crystal theoretic results of $B(\Lambda_0)$ to $B(\Lambda_0)^L$.
\item The partitions which are nodes of $B(\Lambda_0)^L$ can be identified by a simple combinatorial condition.
\end{itemize}

The new description of the crystal $B(\Lambda_0)$ is in many ways more important than the theorems which were proven by the existence of it. Besides the fact that it is a useful tool in proving theorems about $B(\Lambda_0)$, our new description also gives a set of partitions (in bijection to $\ell$-regular partitions), which can be interpreted in terms of the representation theory of $H_n(q)$.

\subsection{Background}
We now require that $\ell \geq 3$. We denote the transpose of $\lambda$ by $\lambda'$. Sometimes the shorthand $(a^k)$ will be used to represent the rectangular partition which has $k$-parts, all of size $a$.

A removable $\ell$-rim hook contained entirely in a single column of the Young diagram of a partition will be called a \textit{vertical $\ell$-rim hook}. Removable $\ell$-rim hooks not contained in a single column will be called \textit{non-vertical $\ell$-rim hooks}.
Two $\ell$-rim hooks will be called \textit{adjacent} if there exists boxes in each which share an edge.

\begin{example} Let $\lambda = (3,2,1)$ and let $\ell=3$. Then the boxes $(1,2), (1,3)$ and $(2,2)$ comprise a (non-vertical, non-horizontal) $3$-rim hook. After removal of this $3$-rim hook, the remaining partition is $(1,1,1)$, which is a vertical $3$-rim hook. These two $3$-rim hooks are adjacent.
\end{example}
$$\tableau{ \mbox{} &\mbox{}&\mbox{}\\\mbox{}&\mbox{}\\\mbox{}}$$
\begin{example} Let $\lambda = (4,1,1,1)$ and $\ell=3$. Then $\lambda$ has two $3$-rim hooks (one horizontal and one vertical). They are not adjacent.
\end{example}
$$\tableau{\mbox{}&\mbox{}&\mbox{}&\mbox{}\\\mbox{}\\\mbox{}\\\mbox{}}$$

We study the representation theory of $H_n(q)$ and the condition on when a Specht module $S^\lambda$ is irreducible (we will always assume that $q \neq 1$ and that $q \in \mathbb{F}$ is a primitive $\ell^{th}$ root of unity in a field $\mathbb{F}$ of characteristic zero).

Fayers recently settled the following conjecture of James and Mathas, which is an extension of Carter's criterion to all partitions (not necessarily $\ell$-regular).
\begin{theorem}[Fayers, \cite{F}]\label{JM_irred} Suppose $\ell > 2$. Let $\lambda$ be a partition. Then $S^{\lambda}$ is reducible if and only if there exists boxes (a,b) (a,y) and (x,b) in the Young diagram of $\lambda$ for which:
\begin{itemize}
\item $\nu_{\ell}(h_{(a,b)}) =1$
\item $\nu_{\ell}(h_{(a,y)}) = \nu_{\ell}(h_{(x,b)})= 0$ 
\end{itemize}
\end{theorem}

A partition which has no such boxes is called an \textit{$(\ell,0)$-JM partition}. 
Equivalently, $\lambda$ is an $(\ell,0)$-JM partition if and only if the Specht module 
$S^\lambda$ is irreducible.

For $\lambda$ not necessarily $\ell$-regular, $S^{\lambda}$ is irreducible if and only if there exists an $\ell$-regular partition $\mu$ so that $S^{\lambda} \cong D^{\mu}$.

\subsection{Results of this chapter}
In Section \ref{removing} we give a new way of characterizing 
$(\ell,0)$-JM partitions by their removable $\ell$-rim hooks. In Section 
\ref{construct_JM} we use this result to count the number of $(\ell,0)$-JM partitions in a block. 
 In Section \ref{new_crystal} we give our new description of the crystal $B(\Lambda_0)$. Section
 \ref{regularization} is a brief introduction to the map on partitions known as regularization, 
which was first described by James. Section \ref{locked} gives a new procedure for finding an inverse
 for the map of regularization. Section \ref{extend} extends our crystal theorems from Chapter 
\ref{sec:LabelForChapter2} to the crystal $B(\Lambda_0)^L$. Section \ref{crystal_lemmas} 
reinterprets the classical crystal rules in the combinatorial framework of the new crystal rules.
 Section \ref{reg_and_crystal} contains the proof that the two descriptions of the crystal 
$B(\Lambda_0)$ are the same; the isomorphism being regularization. Section \ref{reg_and_crystal} 
transfers the crystal theorems on the new crystal to theorems on the old crystal by using this 
isomorphism.  In Section \ref{new_proof_JM} we give a purely representation theoretic proof of 
one of our main theorems (Theorem \ref{top_and_bottom_weak}). Lastly, in Section \ref{MullMap} 
we show some connections between the ladder crystal and the Mullineux map, which leads us to a 
nice classification of the partitions in $B(\Lambda_0)^L$. 

\section{Classifying JM partitions by their removable rim hooks}\label{removing}
\subsection{Motivation}
In this section we give a new description of $(\ell,0)$-JM partitions. This condition is related to how $\ell$-rim hooks are removed from a partition and is a generalization of Theorem 2.1.6 in Chapter \ref{sec:LabelForChapter2} about $\ell$-partitions. The condition we give (equivalent to the condition conjectured by James and Mathas and proved by Fayers) will be used in several proofs throughout this Chapter. 
\subsection{Removing rim hooks and JM partitions}
\begin{definition}\label{generalize_l_partition} Let $\lambda$ be a partition. Let $\ell > 2$. Then $\lambda$ is a \textit{generalized $\ell$-partition} if:
\begin{enumerate} 
\item It has no removable $\ell$-rim hooks other than horizontal and vertical $\ell$-rim hooks.
\item\label{gen_condition} If $R$ is a vertical (resp. horizontal) $\ell$-rim hook of $\lambda$ and $S$ is a horizontal (resp. vertical) $\ell$-rim hook of $\lambda \setminus R$, then $R$ and $S$ are not adjacent. 
\item After removing any sequence of horizontal and vertical $\ell$-rim hooks from the Young diagram of $\lambda$, the remaining partition satisfies (1) and (2).
\end{enumerate}
\end{definition}

\begin{example} \label{6111}
Let $\lambda = (3,1,1,1)$. Then $\lambda$ is not a generalized $3$-partition. $\lambda$ has a vertical $3$-rim hook $R$ containing the boxes $(2,1), (3,1), (4,1)$.  Removing R leaves a horizontal $3$-rim hook $S$ containing the boxes $(1,1), (1,2), (1,3)$. $S$ is adjacent to $R$, so $\lambda$ is not a generalized $3$-partition.

\begin{center}
$\tableau{ S & S& S \\ R \\ R \\ R}$
\end{center}
\end{example}

\begin{remark}
We will sometimes abuse notation and say that $R$ and $S$ in Example \ref{6111} are adjacent vertical and horizontal $\ell$-rim hooks. The meaning here is not that they are both $\ell$-rim hooks of $\lambda$ ($S$ is not an $\ell$-rim hook of $\lambda$), 
but rather that they are an example of a violation of condition \ref{gen_condition} from 
Definition \ref{generalize_l_partition}.
\end{remark}

We will need a few lemmas before we come to our main theorem of this section, which states that the notions of $(\ell,0)$-JM partitions and generalized $\ell$-partitions are equivalent.

\begin{lemma}\label{adding}
Suppose $\lambda$ is not an $(\ell,0)$-JM partition. Then a partition obtained from $\lambda$ by adding a horizontal or vertical $\ell$-rim hook is not an $(\ell,0)$-JM partition.
\end{lemma}

The proof is nearly identical to our lemma in Chapter \ref{sec:LabelForChapter2} about adding horizontal $\ell$-rim hooks to an $
\ell$-regular $(\ell,0)$-JM partition, so we refer the reader there. 

Then next lemma gives an equivalent definition of an $\ell$-core described in terms of hook lengths. 

\begin{lemma}\label{coredoesntdividel} A partition $\lambda$ is an $\ell$-core if and only if the hook length $h_{(a,b)}^{\lambda}$ is not divisible by $\ell$ for every box $(a,b)$ in the Young diagram of $\lambda$.
\end{lemma}

The upcoming lemma simplifies the condition for being an $(\ell,0)$-JM partition and is used in the proof of Theorem \ref{main_theorem_JM}

\begin{lemma}\label{rearrange}
Suppose $\lambda$ is not an $(\ell,0)$-JM partition. Then there exist boxes $(c,d)$, $(c,w)$ and $(z, d)$ with $c < z$, $d < w$, and
 $\ell \mid h_{(c,d)}^{\lambda}$, $\ell \nmid h_{(c,w)}^{\lambda}, h_{(z,d)}^{\lambda}$.
 \end{lemma}

\begin{proof}
The condition proven by Fayers \cite{F} says that there exist boxes $(a,b)$, $(a,y)$ and $(x, b)$ where $\ell \mid h_{(a,b)}^{\lambda}$ and $\ell \nmid h_{(a,y)}^{\lambda}, h_{(x,b)}^{\lambda}$. If $a<x$ and $b<y$ then we are done. The other cases follow below:

Case 1: That $x<a$ and $y<b$. Assume no triple exists satisfying the statement of the lemma. Then either all boxes to the right of the $(a,b)$ box will have hook lengths divisible by $\ell$, or all boxes below will. Without loss of generality, suppose that all boxes below the $(a,b)$ box have hook lengths divisible by $\ell$. Let $c<a$ be the largest integer so that $\ell \nmid h_{(c,b)}$. Let $z = c+1$. Then one of the boxes $(c,b+1), (c,b+2), \dots (c,b+\ell-1)$ has a hook length divisible by $\ell$. This is because the box $(h,b)$ at the bottom of column $b$ has a hook length divisible by $\ell$, so the hook lengths $h_{(c,b)}^{\lambda} = h_{(c,b+1)}^{\lambda} +1 = \dots = h_{(c, b+\ell-1)}^{\lambda} + \ell -1$. Suppose it is $(c,d)$. Then $\ell \nmid h_{(z,d)}^{\lambda}$ since $h_{(z,b)} = h_{(z,d)} + d-b$ and $d-b < \ell$.

 If $d \neq b+ \ell-1$ or $h_{(h,b)}^{\lambda} > \ell$ then letting $w = d+1$ gives $(c, w)$ to the right of $(c,d)$ so that $\ell \nmid h_{(c,w)}^{\lambda}$ (in fact $h_{(c,w)}^{\lambda} = h_{(c,d)}^{\lambda} - 1$).

If $d = b+\ell-1$ and $h_{(h,b)}^{\lambda} = \ell$ then there is a box in spot $(c, d+1)$ with hook length $h_{(c,d+1)}^{\lambda} = h_{(c,d)}^{\lambda} - 2$ since there must be a box in the position $(h-1, d+1)$, due to the fact that $\ell \mid h_{(h-1, b)}^\lambda$ and $h_{(h-1, b)}^\lambda > \ell$. Letting $w = d+1$ again yields $\ell \nmid h_{(c,w)}^\lambda$. Note that this requires that $\ell > 2$. In fact if $\ell =2$ we cannot even be sure that there is a box in position $(c, d+1)$. 

Case 2: That $x<a$ and $y>b$.
If there was a box $(n,b)$ ($n > a$) with a hook length not divisible by $\ell$ then we would be done. So we can assume that all hook lengths in column $b$ below row $a$ are divisible by $\ell$. Let $c<a$ be the largest integer so that $\ell \nmid h_{(c,b)}^{\lambda}$. Let $z = c+1$. Similar to Case 1 above, we find a $d$ so that $\ell \mid h_{(c,d)}^\lambda$. Then $\ell \nmid h_{(z,d)}^\lambda$ and by the same arguement as in Case 1, if we let $w = d+1$ then $\ell \nmid h_{(c,w)}^\lambda$. 

Case 3: That $x>a$ and $y<b$. Then transpose $\lambda$, apply Case 2, and transpose back.
\end{proof}
\begin{theorem}\label{main_theorem_JM} A partition is an $(\ell,0)$-JM partition if and only if it is a generalized $\ell$-partition.
\end{theorem}
\begin{proof}
Suppose that $\lambda$ is not a generalized $\ell$-partition. Then remove non-adjacent horizontal 
and vertical $\ell$-rim hooks until you obtain a partition $\mu$ which has either a non-vertical
 non-horizontal $\ell$-rim hook, or adjacent horizontal and vertical $\ell$-rim hooks. If there is 
a non-horizontal, non-vertical $\ell$-rim hook in $\mu$, lets say the $\ell$-rim hook has
 southwest most box $(a,b)$ and northeast most box $(c,d)$. Then $\ell \mid h_{(c,b)}^{\mu}$ but 
 $\ell \nmid h_{(a,b)}^{\mu} , h_{(c,d)}^{\mu}$ since $h_{(a,b)}^\mu = h_{(c,d)}^\mu = 1$.
 We know $a \neq c$ and $b \neq d$ since our $\ell$-rim hook is non-horizontal and non-vertical. 
Therefore, $\mu$ is not an $(\ell,0)$-JM partition. By Lemma \ref{adding}, $\lambda$ is not an 
$(\ell,0)$-JM partition. Similarly, if $\mu$ has adjacent vertical and horizontal $\ell$-rim hooks,
 then let $(a,b)$ be the southwest most box in the vertical $\ell$-rim hook and let $(c,d)$ be the
 position of the northeast most box in the horizontal $\ell$-rim hook. Again, $\ell \mid h_{(c,b)}^{\mu}$
 but  $\ell \nmid h_{(a,b)}^{\mu} , h_{(c,d)}^{\mu}$. Therefore $\mu$ cannot be an $(\ell,0)$-JM 
partition, so $\lambda$ is not an $(\ell,0)$-JM partition.

Conversely, suppose $\lambda$ is not an $(\ell,0)$-JM partition, but that it is a generalized
 $\ell$-partition. Then by Lemma \ref{rearrange} there are boxes $(a,b)$, $(a,y)$ and $(x,b)$ 
with $a < x$ and $b < y$, which satisfy $\ell \mid h_{(a,b)}^{\lambda}$, and 
$\ell \nmid h_{(a,y)}^\lambda, h_{(x,b)}^\lambda$ . Form a new partition $\mu$ by taking all of the 
boxes $(m,n)$ in $\lambda$ such that $m \geq a$ and $n \geq b$. Since $\lambda$ was a generalized $\ell$-partition, 
$\mu$ must be also. By Lemma \ref{coredoesntdividel}, we know that there must exist a removable 
$\ell$-rim hook from $\mu$, since $\ell \mid h_{(1,1)}^{\mu}$. Since $\mu$ is a generalized 
$\ell$-partition, the $\ell$-rim hook must be either horizontal or vertical. 
If we remove a horizontal $\ell$-rim 
hook we can only change $h_{(1,1)}^\mu$ by either $0$, $1$ or $\ell$. 

If it was changed by 1, then we let $(x',1)$ denote the box in $\mu$ which corresponds to $(x,b)$ 
in $\lambda$. Letting $\nu$ denote the partition $\mu$ with the horizontal $\ell$-rim hook removed,
we then note that $\ell \mid h_{(1,\ell)}^\nu$, $\ell \nmid h_{(1,1)}^nu$ and 
$\ell \nmid h_{(x', \ell)}^nu$. This is because $h_{(1,\ell)}^\mu = 1 \mod \ell$, 
$h_{(1,1)}^nu = 0 \mod \ell$, $h_{(x',\ell)}^\mu = h_{(x',1)}^\mu + 1 mod \ell$ and the fact that 
removing the $\ell$-rim hook question will drop each of these hook lengths by exactly $1$. Hence we deduce
that $\nu$ is not an $(\ell,0)$-JM partition, but it is still a generalized $\ell$-partition. In this
 case we start this process over, with $\nu$ in place of $\lambda$.
We do the same for vertical $\ell$-rim hooks which change the hook length $h_{(1,1)}^\mu$ by $1$.

Now we may assume that removing horizontal or vertical 
$\ell$-rim hooks from $\mu$ will not change that $\ell$ divides $h_{(1,1)}^\mu$ (because removing 
each $\ell$-rim hook will change the hook length $h_{(1,1)}^mu$ by either $0$ or $\ell$). Therefore we 
can keep removing $\ell$-rim hooks until we have have removed box (1,1) entirely, in which case 
the remaining partition had a horizontal $\ell$-rim hook adjacent to a vertical $\ell$-rim hook 
(since both $(x,b)$ and $(a,y)$ must have been removed, the $\ell$-rim hooks could not have been 
exclusively horizontal or vertical). This contradicts $\mu$ being a generalized $\ell$-partition.
\end{proof}

\begin{example} Let $\lambda = (10,8,3,2^2, 1^5)$. Then $\lambda$ is a generalized 3-partition and a $(3,0)$-JM partition. $\lambda$ is drawn below with each hook length $h_{(a,b)}^\lambda$ written in the box $(a,b)$ and the possible removable $\ell$-rim hooks outlined. Also, hook lengths which are divisible by $\ell$ are underlined.

\begin{center}
$\tableau{ 19&13&10&8&7&\underline{6}&5&4&2&1\\
16&10&7&5&4&\underline{3}&2&1\\
10&4&1\\
8&2\\
7&1\\
5\\
4\\
\underline{3}\\
2\\
1
}
\linethickness{2.5pt}
\put (-180,-162){\line(0,1){54}}
\put (-162,-162){\line(0,1){54}}
\put (-180,-162){\line(1,0){18}}
\put (-180,-108){\line(1,0){18}}
\put (-54,0){\line(1,0){54}}
\put (-54,18){\line(1,0){54}}
\put (-54,0){\line(0,1){18}}
\put (0,0){\line(0,1){18}}
\put (-90,-18){\line(1,0){54}}
\put (-90,0){\line(1,0){54}}
\put (-90,-18){\line(0,1){18}}
\put (-36,-18){\line(0,1){18}}
$
\end{center}
\end{example}

\begin{example} $\lambda = (3,1,1,1)$ from Example \ref{6111} is not a $(3,0)$-JM partition and it is not a generalized $3$-partition. 
\end{example}

\section{Decomposition of JM partitions}\label{construct_JM} 
\subsection{Motivation} In Chapter \ref{sec:LabelForChapter2} we gave a decomposition of $\ell$-partitions. In this section we give a similar decomposition for all $(\ell,0)$-JM partitions. This decomposition is important for the proofs of the theorems in later sections. It also implies Theorem \ref{gen_func}, which counts all $(\ell,0)$-JM partitions of a given core and weight.

\subsection{Decomposing JM partitions}
Let $\mu$ be an $\ell$-core with $\mu_1-\mu_2 < \ell-1$ and $\mu_1'-\mu_2' < \ell-1$. Let $r,s \geq 0$. Let $\rho$ and $\sigma$ be partitions with $len(\rho) \leq r+1$ and $len(\sigma) \leq s+1$. If $\mu = \emptyset$ then we require at least one of $\rho_{r+1}, \sigma_{s+1}$ to be zero. Following the construction of Chapter \ref{sec:LabelForChapter2}, we construct a partition corresponding to $(\mu,r,s,\rho,\sigma)$ as follows. Starting with $\mu$, attach $r$ rows above $\mu$, with each row $\ell-1$ boxes longer than the previous. Then attach $s$ columns to the left of $\mu$, with each column $\ell-1$ boxes longer than the previous. This partition will be denoted $(\mu,r,s,\emptyset,\emptyset)$. Formally, if $\mu = (\mu_1,\mu_2,\dots,\mu_m)$ then $(\mu,r,s,\emptyset,\emptyset)$ represents the partition (which is an $\ell$-core):  
$$(s+\mu_1+r(\ell-1), s+\mu_1+ (r-1)(\ell-1), \dots, s+\mu_1+\ell-1, s+\mu_1, $$ $$s+\mu_2, \dots, 
s+\mu_m, s^{\ell-1}, (s-1)^{\ell-1}, \dots, 1^{\ell-1})$$
where $s^{\ell-1}$ stands for $\ell-1$ copies of $s$.
Now to the first $r+1$ rows attach $\rho_i$ horizontal $\ell$-rim hooks to row $i$. Similarly, to the first $s+1$ columns, attach $\sigma_j$ vertical $\ell$-rim hooks to column $j$. The resulting partition $\lambda$ corresponding to $(\mu,r,s,\rho,\sigma)$ will be
$$\lambda = (s+\mu_1+r(\ell-1)+ \rho_1 \ell, s+\mu_1+(r-1)(\ell-1)+ \rho_2 \ell, \dots, $$ $$s+\mu_1+ (\ell-1) + \rho_r 
\ell,  s+\mu_1+\rho_{r+1}\ell, s+\mu_2, s+\mu_3, \dots, $$ $$s+\mu_m,
(s+1)^{\sigma_{s+1}\ell}, s^{\ell-1+(\sigma_s - \sigma_{s+1})\ell}, (s-1)^{\ell-1 + (\sigma_{s-1} - \sigma_s)\ell}, \dots , 1^{\ell-1 + (\sigma_1-\sigma_2)\ell}).$$

We denote this decomposition as $\lambda \approx (\mu,r,s,\rho,\sigma)$.
\begin{example}
Let $\ell = 3$ and $(\mu,r,s,\rho,\sigma) = ((1),3,2,(2,1,1,1),(2,1,0))$. Then $((1),3,2, \emptyset, \emptyset)$ is drawn below, with $\mu$ framed.
\\

$\begin{array}{cc}

\linethickness{2.5pt}
\put (45,-36){\line(1,0){18}}
\put (45,-54){\line(1,0){18}}
\put (45,-54){\line(0,1){18}}
\put (63,-54){\line(0,1){18}}
\linethickness{1pt}
\put (10,-140){\line(1,0){36}}
\put (10,-140){\line(0,1){6}}
\put (46,-140){\line(0,1){6}}
\put (15,-136){$s = 2$}
\put (190,-36){\line(0,1){54}}
\put (190,-36){\line(-1,0){6}}
\put (190,18){\line(-1,0){6}}
\put (200,-10){$r = 3$}
&
\tableau{\mbox{}&\mbox{}&\mbox{}&\mbox{}&\mbox{}&\mbox{}&\mbox{}&\mbox{}&\mbox{}\\
\mbox{}&\mbox{}&\mbox{}&\mbox{}&\mbox{}&\mbox{}&\mbox{}\\
\mbox{}&\mbox{}&\mbox{}&\mbox{}&\mbox{}\\
\mbox{}&\mbox{}&\mbox{}\\
\mbox{}&\mbox{}\\
\mbox{}&\mbox{}\\
\mbox{}\\
\mbox{}
}
\end{array}$
\vspace{12 pt}

$((1),3,2,(2,1,1,1),(2,1,0))$ is drawn below, now with $((1),3,2, \emptyset, \emptyset)$ framed.

$\begin{array}{cc}
\linethickness{2.5pt}
\put (63,-36){\line(1,0){36}}
\put (45,-54){\line(1,0){18}}
\put (45,-90){\line(0,1){36}}
\put (63,-54){\line(0,1){18}}
\put (27,-90){\line(1,0){18}}
\put (27,-126){\line(0,1){36}}
\put (9,-126){\line(1,0){18}}
\put (9,-126){\line(0,1){144}}
\put (9,18){\line(1,0){162}}
\put (171,0){\line(0,1){18}}
\put (135,0){\line(1,0){36}}
\put (135,-18){\line(0,1){18}}
\put (99,-18){\line(1,0){36}}
\put (99,-36){\line(0,1){18}}

&
\tableau{\mbox{}&\mbox{}&\mbox{}&\mbox{}&\mbox{}&\mbox{}&\mbox{}&\mbox{}&\mbox{}&\mbox{}&\mbox{}&\mbox{}&\mbox{}&\mbox{}&\mbox{}\\
\mbox{}&\mbox{}&\mbox{}&\mbox{}&\mbox{}&\mbox{}&\mbox{}&\mbox{}&\mbox{}&\mbox{}\\
\mbox{}&\mbox{}&\mbox{}&\mbox{}&\mbox{}&\mbox{}&\mbox{}&\mbox{}\\
\mbox{}&\mbox{}&\mbox{}&\mbox{}&\mbox{}&\mbox{}\\
\mbox{}&\mbox{}\\
\mbox{}&\mbox{}\\
\mbox{}&\mbox{}\\
\mbox{}&\mbox{}\\
\mbox{}&\mbox{}\\
\mbox{}\\
\mbox{}\\
\mbox{}\\
\mbox{}\\
\mbox{}
}
\end{array}
$

\end{example}

\begin{theorem}\label{construct_JMs}
If $\lambda \approx (\mu,r,s,\rho,\sigma)$ (with at least one of $\rho_{r+1}, \sigma_{s+1} = 0$ if $\mu=\emptyset$), then $\lambda$ is an $(\ell,0)$-JM partition. Conversely, all such $(\ell,0)$-JM partitions are of this form.
\end{theorem}
\begin{proof}
If $\lambda \approx (\mu,r,s,\rho,\sigma)$ then it is clear by construction that $\lambda$ satisfies the criterion for a generalized $\ell$-partition (see Definition \ref{generalize_l_partition}). By Theorem \ref{main_theorem_JM}, $\lambda$ is an $(\ell,0)$-JM partition.

Conversely, if $\lambda$ is an $(\ell,0)$-JM partition then by Theorem \ref{main_theorem_JM} its only removable $\ell$-rim hooks are horizontal or vertical. Let $\rho_i$ be the number of removable horizontal $\ell$-rim hooks in row $i$ which are removed in going to the $\ell$-core of $\lambda$, and let $\sigma_j$ be the number of removable vertical $\ell$-rim hooks in column $j$ (since $\lambda$ has no adjacent $\ell$-rim hooks, these numbers are well defined). Once all $\ell$-rim hooks are removed, let $r$ (resp. $s$) be the number of rows (resp. columns) whose successive differences are $\ell-1$. Removing these topmost $r$ rows and leftmost $s$ columns leaves an $\ell$-core $\mu$. Then $\lambda \approx (\mu, r, s, \rho, \sigma)$.
\end{proof}

Further in the text, we will make use of Theorem \ref{construct_JMs}. Many times we will show that a partition $\lambda$ is an $(\ell,0)$-JM partition by giving an explicit decomposition of $\lambda$ into $(\mu,r,s,\rho,\sigma)$.

The following theorem follows directly from Theorem \ref{construct_JMs}. We have decided to include it after reading \cite{COV} and recognizing the similarity between their work and ours. Let $\Gamma_{\ell} (\nu) = \{ \lambda : \lambda \textrm{ is an } (\ell,0) \textrm{-JM partition with core } \nu \}$.

\begin{theorem}\label{gen_func} Suppose $\nu$ is an $\ell$-core, with 
\begin{itemize}
\item $\nu_1 - \nu_2 = \nu_2- \nu_3 = \dots =\nu_r-\nu_{r+1} = \ell-1$ and $\nu_{r+1}-\nu_{r+2} \neq \ell-1$,
\item $\nu'_1 - \nu'_2 = \nu'_2- \nu'_3 = \dots =\nu'_s-\nu'_{s+1} = \ell-1$ and $\nu'_{s+1}-\nu'_{s+2} \neq \ell-1$.
\end{itemize}
If $\nu$ also satisfies $\nu_{r+1} > s$
then the number of $(\ell,0)$-JM partitions of a fixed $\ell$-weight $w$ with core $\nu$, denoted $\Gamma_\ell(\nu)$, is the number of pairs of partitions $(\alpha,\beta)$, $len(\alpha) \leq r+1$, $len(\beta) \leq s+1$ whose size sums to $w$ ($|\alpha|+|\beta| = w$). The generating function for all $(\ell,0)$-JM partitions with core $\nu$ with respect to the statistic of the $\ell$-weight of the partition is thus
 $\displaystyle \prod_{i=0}^{r+1} \frac{1}{1-x^i} \prod_{j=0}^{s+1} \frac{1}{1-x^j}.$ Hence the generating function for all $(\ell,0)$-JM partitions for a fixed core $\nu$ as above with respect to the statistic of the size of the partition is 
$$\displaystyle \sum_{\lambda \in \Gamma_{\ell}(\nu)} x^{|\lambda|} = x^{|\nu|} \prod_{i=0}^{r+1} \frac{1}{1-x^{\ell i}} \prod_{j=0}^{s+1} \frac{1}{1-x^{\ell j}}.$$

If $\nu$ satisfies $\nu_{r+1}= s$ then the number of $(\ell,0)$-JM partitions of a fixed $\ell$-weight $w$ with core $\nu$ is the number of pairs of partitions $(\alpha,\beta)$, $len(\alpha) \leq r+1$, $len(\beta) \leq s$ 
whose sizes sum to $w$ plus the number of pairs of partitions $(\gamma,\delta)$, $len(\gamma) \leq r$, $len(\delta) \leq s+1$ whose sizes sum to $w$. 
The generating function for all $(\ell,0)$-JM partitions with core $\nu$ with respect to the statistic of the $\ell$-weight of the partition is thus 
$\displaystyle \left(\frac{1}{1-x^{r+1}} + \frac{1}{1-x^{s+1}}\right) \prod_{i=0}^{r} \frac{1}{1-x^i} \prod_{j=0}^{s} \frac{1}{1-x^j} 
.$ Hence the generating function for all $(\ell,0)$-JM partitions for a fixed core $\nu$ with respect to the statistic of the size of the partition is 
$$\displaystyle \sum_{\lambda \in \Gamma_{\ell}(\nu)} x^{|\lambda|} = x^{|\nu|} \left(\frac{1}{1-x^{\ell (r+1)}} + \frac{1}{1-x^{\ell (s+1)}}\right)\prod_{i=0}^{r} \frac{1}{1-x^{\ell i}} \prod_{j=0}^{s} \frac{1}{1-x^{\ell j}} .$$

\end{theorem}

\begin{example}
Let $\ell = 3$ and $\lambda = (3,1)$. Then $\lambda$ is a 3-core. $\lambda$ has $r = 1$ and $s = 0$. Suppose we are looking for the number of $(3,0)$-JM partitions with core $\lambda$ and size 13 (i.e. weight 3). There are as many such (3,0)-JM partitions as there are pairs of partitions $(\alpha, \beta)$ with $len(\alpha) \leq 2$, $len(\beta) \leq 1$ and $|\alpha| + |\beta| = 3$. There are 6 such pairs of partitions. They are $((3), \emptyset), ((2,1), \emptyset), ((2),(1)), ((1,1), (1)), ((1),(2))$ and $(\emptyset, (3))$. They correspond to the 6 $(3,0)$-JM partitions (12,1), (9,4), $(9, 1^4),$ $(6,4,1^3),$ $(6, 1^7)$ and $(3, 1^{10})$ respectively.
\end{example}

\section{The ladder crystal}\label{new_crystal}
\subsection{Ladders}
We first recall what a ladder is in regards to a partition. Let $\lambda$ be a partition and let $\ell>2$ be a fixed integer. For any box $(a,b)$ in the Young diagram of $\lambda$, the \textit{ladder} of $(a,b)$ is the set of all positions $(c,d)$ which satisfy $\frac{c-a}{d-b} = \ell-1$ and $c,d >0$.

\begin{remark}
The definition implies that two boxes in the same ladder will share the same residue. An $i$-ladder will be a ladder which has residue $i$.
\end{remark}
 
\begin{example}
Let $\lambda = (3,3,1)$, $\ell = 3$. Then there is a 1-ladder which contains the boxes $(1,2)$ and $(3,1)$, and a different 1-ladder which has the box $(2,3)$ in $\lambda$ and the boxes $(4,2)$ and $(6,1)$ not in $\lambda$. In the picture below, lines are drawn through the different 1-ladders.

\begin{center}
$ \begin{array}{cc}
\put (13,-38){\line(1,2){27}}
\put (68,17){\line(-1,-2){52}}

&
\tableau{0&1&2\\
2&0&1\\
1}
\end{array}$
\end{center}
\end{example}
\subsection{The ladder crystal}
We will construct a new crystal $B(\Lambda_0)^L$ recursively as follows. First, the empty partition $\emptyset$ is the unique highest weight node of our crystal. From $\emptyset$, we will build the crystal by applying the operators $\widehat{f}_{i}$ for $0\leq i < \ell$. We define $\widehat{f}_{i}$ to act on partitions, taking a partition of $n$ to a partition of $n+1$ (or 0) in the following manner. Given $\lambda \vdash n$, first draw all of the $i$-ladders of $\lambda$ onto its Young diagram. Label any addable $i$-box with a $+$, and any removable $i$-box with a $-$. Now, write down the word of $+$'s and $-$'s by reading from leftmost $i$-ladder to rightmost $i$-ladder and reading from top to bottom on each ladder. This is called the \textit{ladder $i$-signature} of $\lambda$. From here, cancel any adjacent $-+$ pairs in the word, until you obtain a word of the form $+\dots+-\dots-$. This is called the \textit{reduced ladder $i$-signature of $\lambda$}. All boxes associated to a $-$ in the reduced ladder $i$-signature are called \textit{ladder normal $i$-boxes} and all boxes associated to a $+$ in the reduced ladder $i$-signature are called \textit{ladder conormal $i$-boxes}. The box associated to the leftmost $-$ is called the \textit{ladder good $i$-box} and the box associated to the rightmost $+$ is called the \textit{ladder cogood $i$-box}. Then we define $\widehat{f}_{i} \lambda$ to be the partition $\lambda$ union the ladder cogood $i$-box. If no such box exists, then $\widehat{f}_{i} \lambda = 0$. Similarly, $\widehat{e_i} \lambda$ is the partition $\lambda$ with the ladder good $i$-box removed. If no such box exists, then $\widehat{e_i} \lambda = 0$. We then define $\widehat{\varphi}_i(\lambda)$ to be the number of ladder conormal $i$-boxes of $\lambda$ and $\widehat{\varepsilon}_i(\lambda)$ to be the number of ladder normal $i$-boxes. It is then obvious that $\widehat{\varphi_i}(\lambda) = \max \{ k : \widehat{f}_i^k \lambda \neq 0 \}$ and that $\widehat{\varepsilon_i}(\lambda) = \max \{ k : \widehat{e}_i^k \lambda \neq 0 \}$. For the rest of the Chapter, $\widehat{\varphi} = \widehat{\varphi}_i(\lambda)$ and $\widehat{\varepsilon} = \widehat{\varepsilon}_i(\lambda)$.

\begin{remark}
It remains to be shown that this directed graph is a crystal. To see this, we will show that it is isomorphic to $B(\Lambda_0)$. 
\end{remark}

\begin{example}\label{5311111}
Let $\lambda = (5,3,1,1,1,1,1)$ and $\ell = 3$. Then there are four addable 2-boxes for $\lambda$. In the leftmost 2-ladder (containing box (2,1)) there are no addable (or removable) 2-boxes. In the next 2-ladder (containing box (1,3)) there is an addable 2-box in box (3,2). In the next 2-ladder (containing box (2,4)), there are two addable 2-boxes, in boxes (2,4) and (8,1). In the last drawn 2-ladder (containing box (1,6)) there is one addable 2-box, in box (1,6). There are no removable 2-boxes in $\lambda$. Therefore the ladder 2-signature (and hence reduced ladder 2-signature) of $\lambda$ is $+_{(3,2)}+_{(2,4)}+_{(8,1)}+_{(1,6)}$ (Here, we have included subscripts on the $+$ signs so that the reader can see the correct order of the $+$'s). Hence ${\widehat{f}_2} \lambda = (6,3,1,1,1,1,1)$, $({\widehat{f}_2})^2 \lambda = (6,3,1,1,1,1,1,1)$, $({\widehat{f}_2})^3 \lambda = (6,4,1,1,1,1,1,1)$ and $({\widehat{f}_2})^4 \lambda = (6,4,2,1,1,1,1,1)$. $({\widehat{f}_2})^5 \lambda = 0$. See Figure \ref{new_crystal_figure} for the 2-string of $\lambda$. 

\begin{center}
$\begin{array}{cc}
\put (10,-29){\line(1,2){23.5}}
\put (10,-82){\line(1,2){50}}
\put (15,-126){\line(1,2){72}}
\put (41,-126){\line(1,2){72}}
\put (106,6){2}
\put (70,-12){2}
\put (34,-30){2}
\put (16, -120){2}

& \tableau{0&1&2&0&1\\
2&0&1\\
1\\
0\\
2\\
1\\
0
}
\end{array}$
\end{center}
\end{example}


\begin{figure}\label{new_crystal_figure}
\includegraphics{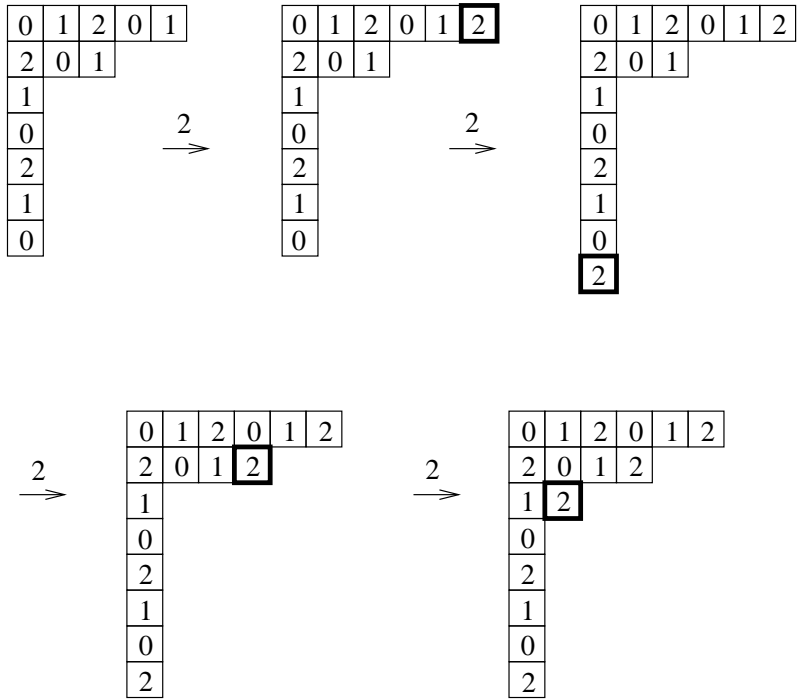}
\caption{The 2-string of the partition $(5,3,1,1,1,1,1)$ in $B(\Lambda_0)$ for $\ell = 3$ from Example \ref{5311111}}
\end{figure}

\begin{figure}
\includegraphics{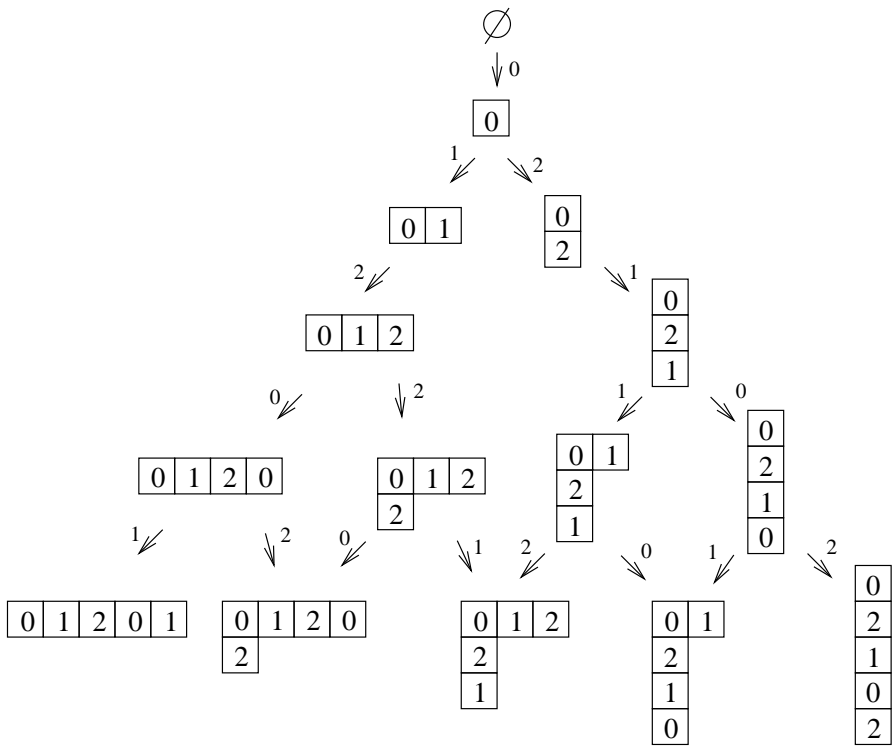}
\caption{The first 6 levels of $B(\Lambda_0)^L$ for $\ell =3 \;\;\;\;\;\;\;\;\;\;\;\;\;\;\;\;\;\;\;\;\;\;\;\;$ }
\end{figure}

From this description, it is not obvious that this is a crystal. However, we will soon show that it is isomorphic to $B(\Lambda_0)$. 
\begin{remark} We note that this crystal rule is well defined for all partitions. We will mainly be studying the connected component of $\emptyset$, which is what we will show is isomorphic to $B(\Lambda_0)$. 
\end{remark}

\begin{remark}
For the crystal expert: The weight function of this crystal is exactly the same as the weight function for $B(\Lambda_0)^L$. Explicitly, the weight of $\lambda$ is $\Lambda_0 - \sum c_i \alpha_i$ where $c_i$ is the number of boxes of $\lambda$ with residue $i$. Throughout this Chapter we will suppress the weight function as it is irrelevant to the combinatorics involved.
\end{remark}

We end this section by proving a simple property of $\widehat{e}_i$ and $\widehat{f}_i$.
\begin{lemma}\label{e_inverse_f} $\widehat{e}_i \lambda = \mu$ if and only if $\widehat{f}_i \mu = \lambda$.
\end{lemma}
\begin{proof}
In the reduced $i$-signature of $\mu$, the rightmost $+$ will be changed by $\widehat{f}_i$ to a $-$. The resulting word will have no additional cancelation. Therefore, the reduced $i$-signature of $\lambda = \widehat{f}_i \mu$ will have its leftmost $-$ corresponding to the same box. Thus $\widehat{e}_i \lambda = \mu$. The other direction is similar.
\end{proof}

\section{Regularization}\label{regularization}

\subsection{The operation of regularization} In this section we describe a map from the set of partitions to the set of $\ell$-regular partitions. The map is called regularization and was first defined by James (see \cite{J}). For a given $\lambda$, move all of the boxes up to the top of their respective ladders. The result is a partition, and that partition is called the \textit{regularization} of $\lambda$, and is denoted $\mathcal{R}_\ell \lambda$. Although $\mathcal{R}_\ell$ depends on $\ell$, we will usually just write $\mathcal{R}$. The following theorem contains facts about regularization originally due to James \cite{J} (see also \cite{JM}).
\begin{theorem}\label{reg_prop} Let $\lambda$ be a partition. Then
\begin{itemize}
\item $\mathcal{R} \lambda$ is $\ell$-regular
\item $\mathcal{R} \lambda = \lambda$ if and only if $\lambda$ is $\ell$-regular.
\item If $\lambda$ is $\ell$-regular and $D^{\lambda} \cong S^{\nu}$ for some partition $\nu$, then $\mathcal{R} \nu = \lambda$.
\end{itemize}
\end{theorem}

Regularization provides us with an equivalence relation on the set of partitions. Specifically, we say $\lambda \sim \mu$ if $\mathcal{R} \lambda = \mathcal{R} \mu$. The equivalence classes are called \textit{regularization classes}, and the class of a partition $\lambda$ is denoted $\mathcal{RC}(\lambda) := \{ \mu \in \mathcal{P} : \mathcal{R}\mu = \mathcal{R}\lambda \}$.

\begin{example}\label{reg_example} Let $\lambda = (2,2,2,1,1,1)$ and let $\ell = 3$. Then $\mathcal{R}\lambda = (3,3,2,1)$. Also, 
$$\mathcal{RC}(\lambda) = \{ (2, 2, 2, 1, 1, 1),
(2, 2, 2, 2, 1),
(3, 2, 1, 1, 1, 1),$$
$$
(3, 2, 2, 2),
(3, 3, 1, 1, 1),
(3, 3, 2, 1)    \}$$.

\begin{center}
$\begin{array}{lcr}
\tableau{0&1\\
2&0\\
1&2\\
0\\
2\\
1}
&
\displaystyle
\xrightarrow{\mathcal{R}}
&
\tableau{0&1&2\\
2&0&1\\
1&2\\
0}

\end{array}$
\end{center}

\end{example}

\section{Deregularization}\label{locked}

The goal of this section is to provide an algorithm for finding the smallest partition in dominance order in a given regularization class. It is nontrivial to show that a smallest partition exists. We use this result to show that our new description of the crystal $B(\Lambda_0)^L$ has nodes which are smallest in dominance order in their regularization class. All of the work of this section is inspired by Brant Jones of UC Davis, who gave the first definition of a locked box.

\subsection{Locked boxes}

We recall a partial ordering on the set of partitions of $n$. For two partitions $\lambda$ and $\mu$ of $n$, we say that $\lambda \leq \mu$ if $\sum_{j=1}^i \lambda_j \leq \sum_{j=1}^i \mu_j$ for all $i$. This order is usually called the \textit{dominance order}.

\begin{definition}
The $k^{th}$ ladder of $\lambda$ is the ladder which contains the box $(k,1)$ (this ladder necessarily has residue $-k+1 \mod \ell$). 
\end{definition}

Finding all of the partitions which belong to a regularization class is not easy. The definition of locked boxes below formalizes the concept that some boxes in a partition cannot be moved down their ladders if one requires that the new diagram remain a partition.

\begin{definition}
For a partition $\lambda$, we label boxes of $\lambda$ as \textit{locked} by the following procedure:
\begin{enumerate}
\item If a box $x$ has a locked box directly above it (or is on the first row) and every unoccupied space in the same ladder as $x$, lying below $x$, has an unoccupied space directly above it then $x$ is locked. Boxes locked for this reason are called type I locked boxes.
\item If a box $y$ is locked, then every box to the left of $y$ in the same row is also locked. Boxes locked for this reason are called type II locked boxes.

\end{enumerate}
Boxes which are not locked are called \textit{unlocked}.  \end{definition}

\begin{remark}
Locked boxes can be both type I and type II.
\end{remark}

\begin{example}
Let $\ell = 3$ and let $\lambda = (6,5,4,3,1,1)$. 
Then labeling the locked boxes for $\lambda$ with an $L$ and the unlocked boxes with a $U$ yields the picture below.

\begin{center}
$\tableau{L&L&L&U&U&U\\
L&L&L&U&U\\
L&L&U&U\\
L&L&U\\
L\\
L}
$
\end{center}
\end{example}

The following lemmas follow from the definition of locked boxes.
\begin{lemma}\label{locked_up}
If $(a,b)$ is locked and $a>1$ then $(a-1,b)$ is locked. Equivalently, all boxes which sit below an unlocked box in the same column are unlocked.

\end{lemma}

\begin{proof}
If $(a,b)$ is a type I locked box then by definition $(a-1,b)$ is locked. If $(a,b)$ is a type II locked box and not type I then there exists a $c$ so that $(a,c)$ is a type I locked box with $c > b$. But then by definition of type I locked box, $(a-1,c)$ is locked. Then $(a-1,b)$ is  a type II locked box.
\end{proof}

\begin{lemma}\label{locked_up_ladder}
If there is a locked box in space $(\alpha, \beta)$ and there is a box in space $(\alpha - (\ell-1), \beta+1)$ then the box $(\alpha - (\ell-1), \beta+1)$ is locked. 
\end{lemma}

\begin{proof}
To show this, suppose that $(\alpha- (\ell-1), \beta+1)$ is unlocked. Then let $(\gamma, \beta+1)$ be the highest unlocked box in column $\beta+1$. This is unlocked because it violates the type I locked condition. Then the box $(\gamma+\ell-1, \beta)$ violates the type I locked condition, and it will not have a locked box directly to the right of it ($(\gamma, \beta+1)$ is unlocked, so  $(\gamma+\ell-1, \beta+1)$ will be unlocked if it is occupied, by Lemma \ref{locked_up}). Hence $(\gamma+\ell-1, \beta)$ is unlocked, so $(\alpha, \beta)$ must be unlocked since it sits below $(\gamma+\ell-1, \beta)$ (by Lemma \ref{locked_up}), a contradiction.
\end{proof}

For two partitions $\lambda$ and $\mu$ in the same regularization class, there are many ways to move the boxes in $\lambda$ on their ladders to obtain $\mu$. We define an \textit{arrangement of} $\mu$ \textit{from} $\lambda$ to be an assignment of each box in the Young diagram of $\lambda$ to a box in the Young diagram of $\mu$. An arrangement will be denoted by a set of ordered pairs $(x,y)$ with $x \in \lambda$ and $y \in \mu$, where both $x$ and $y$ are in the same ladder, and each $x \in \lambda$ and $y \in \mu$ is used exactly once. Such a pair $(x,y)$ denotes that the box $x$ from $\lambda$ is moved into position $y$ in $\mu$. $x$ is said to \textit{move up} if the position $y$ which it is paired with is higher in the ladder than $x$ is. Similarly, $x$ is said to \textit{move down} if the position $y$ is lower in the ladder than $x$ is.  $x$ is said to \textit{stay put} if the position $y$ in $\mu$ has the same coordinates as position $x$ in $\lambda$. 

\begin{remark} We introduce the notation $\mathcal{B}(p)$ to denote the position paired with $p$ in an arrangement $\mathcal{B}$, i.e. $(p, \mathcal{B}(p)) \in \mathcal{B}$. 
\end{remark}

\begin{example}
$\lambda = (3,3,1,1,1)$ and $\mu = (2,2,2,2,1)$ are in the same regularization class when $\ell = 3$ (see Example \ref{reg_example}). One possible arrangement of $\mu$ from $\lambda$ would be $\mathcal{B} = \{   ((1,1),(1,1)), ((1,2),(1,2)),((1,3),(5,1)),((2,1),(2,1)),$ $ ((2,2),(2,2)), ((2,3),(4,2)), ((3,1),(3,1)),((4,1),(4,1)), ((5,1),(3,2))\}$. This corresponds to moving the labeled boxes from $\lambda$ to $\mu$  in the corresponding picture below. Note that $(5,1)$ moves up, $(1,3)$ and $(2,3)$ move down, and all other boxes stay put. 
\begin{center}

$\begin{array}{cc}
\tableau{1&2&3\\
4&5&6\\
7\\
8\\
9}
&
\tableau{1&2\\
4&5\\
7&9\\
8&6\\
3}
\end{array}$
\end{center}
\end{example}

\subsection{Algorithm for finding the smallest partition in a regularization class}
 For any partition $\lambda$, to find the smallest partition (with respect to dominance order) in a regularization class we first label each box of $\lambda$ as either locked or unlocked as above. Then we create an arrangement $\mathcal{S}_\lambda$ which moves all unlocked boxes down their ladders, while keeping these unlocked boxes in order (from bottom to top), while locked boxes do not move. The partition obtained from following the arrangement $\mathcal{S}_\lambda$ will be denoted $\mathcal{S} \lambda$. It is unclear that this algorithm will yield the smallest partition in $\mathcal{RC}(\lambda)$, or even that $\mathcal{S} \lambda$ is a partition. In this subsection, we resolve these issues.

\begin{proposition}\label{moveboxesdown} Let $\lambda$ and $\mu$ be partitions in the same regularization class. Then there exists an arrangement of $\mu$ from $\lambda$ such that all of the boxes in $\lambda$ which are locked do not move down.
\end{proposition}

\begin{proof}

To find a contradiction, we suppose that any arrangement of $\mu$ from $\lambda$ must have at least one locked box in $\lambda$ which moves down. 

For any arrangement $\mathcal{C}$, we let $x_\mathcal{C}$ denote the highest rightmost locked box in $\lambda$ which moves down in $\mathcal{C}$. Since all arrangements of $\mu$ from $\lambda$ have at least one locked box in $\lambda$ which moves down, $x_\mathcal{C}$ is defined for all $\mathcal{C}$.

Among all arrangements of $\mu$ from $\lambda$, let $\mathcal{D}$ be one which has $x_{\mathcal{D}}$ in the lowest, leftmost position. Let $x= x_{\mathcal{D}} = (a,b)$. 

We will exhibit a box $w$ on the same ladder as $x$, such that $\mathcal{D}(w)$ is equal to or above $(a,b)$ in $\mu$ and either $w$ is below $x$ or $w$ is above $x$ and unlocked. If such a box exists, then letting $\mathcal{A} = (\mathcal{D} \setminus \{(x, \mathcal{D}(x)), (w, \mathcal{D}(w))\}) \cup \{(x, \mathcal{D}(w)), (w, \mathcal{D}(x)) \}$ will yield a contradiction, as $x_{\mathcal{A}}$ will be in a position to the left of and/or below $x$, which contradicts our choice of $\mathcal{D}$. 

If there exists a $w$ in the same ladder as $x$, below $x$, with $\mathcal{D}(w)$ at or above $(a,b)$ in $\mu$ then we are done. So now we assume that every box below $x$ in the ladder for $x$ does not move on or above $x$.

There are two cases to consider:

\vspace{10pt}
 \noindent\textbf{Case I:} $x$ is a type II locked box and not a type I locked box. 

In this case, there is a locked box directly to the right of $x$ (definition of a type II lock), which we label $y$. 

If $y$ stays put according to $\mathcal{D}$, (i.e. $(y, y) \in \mathcal{D}$) then some box $z$ in the same ladder as $x$ in $\lambda$ must move into the position for $x$ according to $\mathcal{D}$, i.e. $(z,x) \in \mathcal{D}$. Let $w=z$. Since $w$ will move into position $x$, $w$ is above $x$, so it must be unlocked (because $x = x_{\mathcal{D}}$ was the highest locked box which moved down according to $\mathcal{D}$) and we are done.

If $y$ moves according to $\mathcal{D}$ then it must move up (since it is to the right of $x$ and locked).

 If there are any boxes $\tilde{x}$ in the ladder of $x$, above $x$, which are at the end of their row, then $\tilde{x}$ is unlocked (it is not a type II lock because there are no boxes to the right of it, and it can't be a type I lock because $x$ is not a type I lock). If $\tilde{x}$ were to move according to $\mathcal{D}$ to some box on or above $x$, then we could use $w = \tilde{x}$ and be done. So otherwise we assume all such boxes move down below $x$ according to $\mathcal{D}$.

Similarly, if there are any boxes $\hat{x}$ above $x$, in the ladder for $x$, which are directly to the left of a box $\hat{y}$ in the ladder for $y$, and the box $\hat{y}$ moves down according to $\mathcal{D}$, then the box $\hat{x}$ must be unlocked (its not a type II lock because $\hat{y}$ is unlocked, being above $x$, and its not a type I lock because $x$ is not). If $\hat{x}$ moves to a position on or above $x$, then we could use $w=\hat{x}$ and be done. Otherwise, we must assume that all such $\hat{x}$ move down below $x$. 

Assuming we cannot find any $w$ by these methods, we let
\begin{itemize}
\item $k$ denote the number of boxes on the ladder of $x$, above $(a,b)$, in $\lambda$,
\item $j$ denote the number of boxes on the ladder of $y$, above $(a,b)$, in $\lambda$,
\item $k'$ denote the number of boxes on the ladder of $x$, above $(a,b)$, in $\mu$,
\item $j'$ denote the number of boxes on the ladder of $y$, above $(a,b)$, in $\mu$,
\item $m$ denote the number of boxes on the ladder of $y$, above $(a,b)$, in $\lambda$ which move down (i.e. the number of boxes which are of the form $\hat{y}$ above).
\end{itemize}

The number of boxes of the form $\tilde{x}$ is then $k-j$. Also, $j' \geq j-m +1$, since the $m$ boxes need not move below position $(a,b)$, but the box in position $y$ moves above position $(a,b)$. $k' \geq k - m - (k-j) = j-m $ since the number of boxes in ladder $x$ will go down by at least $m$ for the boxes of the form $\hat{x}$ and $k-j$ for the boxes of the form $\tilde{x}$ (by assumption boxes in the ladder for $x$ cannot move up from below). Hence $k' \leq j-m < j-m+1 \leq j'$.   This is a contradiction, as there must be at least as many boxes above $x$ on the ladder of $x$  as there are on the ladder of $y$ for $\mu$ to be a partition.

\vspace{10pt}

\noindent
\textbf{Case II:} $x$ is a type I locked box.

We will show that this case cannot occur. Specifically we will show that the assumption leads to a contradiction of $\lambda$ being a partition.

Since $x$ is a type I lock, the number of boxes on the ladder for $x$ on or below row $a$ in $\lambda$ must be strictly greater than the number of boxes on the ladder of $(a-1,b)$ on or below row $a$. Since $x$ is moving down, some box in the ladder of $(a-1,b)$ above row $a$ must move down (since no box on the ladder for $x$, below $x$, moves above $x$, by assumption). Since $x$ is the highest locked box which can move down, this must be an unlocked box. Hence there is an unlocked box in the ladder of $(a-1,b)$ which is in a row above $x$. 

The existence of an unlocked box in the ladder for $(a-1,b)$ above row $a$ implies that there is an empty position on the ladder for $(a-1,b)$ somewhere above row $a$. This is implied by Lemma \ref{locked_up_ladder} above, since if every space on the ladder $(a-1,b) \in \lambda$ above row $a$  was occupied, then all of those boxes would be locked. 

Let $m$ be the column which contains the lowest such empty position.
  The coordinates for such a position would then be $(a-1-(m-b)(\ell-1),m)$. Let $(c_1, d_1)$ be the highest empty position on the ladder for $x$ below $x$ (we know a space must exist since $x$ is moved down). Since $x$ is a type I lock, $(c_1-1, d_1)$ is also an empty position. We know that $(a-1,b)$ is locked since $x$ is a type I lock. If $(a-1, b)$ is also a type I lock then the space $(c_1-2, d_1)$ must also be empty. Continuing this, if $(a-(\ell-2), b)$ were a type I lock then $(c_1- (\ell-1), d_1)$ would have to be empty, but this would contradict the fact that there should be a box in $(c_1-(\ell-1),d_1 +1)$ since it is in the ladder for $x$ and $(c_1, d_1)$ was chosen to be the the highest empty position on the ladder for $x$ below $x$. 

So there exists a $k_1$ so that $(a-k_1,b)$ is a type II lock, with $k_1 \leq \ell-2$. Since $(a-k_1, b)$ is a type II lock, there exists a type I locked box in a position $(a-k_1, n_1)$ with $b < n_1 < m$. Let $(c_2, d_2)$ denote the highest empty position in the ladder for $(a-k_1,n_1)$ below $(a-k_1, n_1)$ ($(c_2, d_2)$ exists because the position in column $d_1$ on the ladder for $(a-k_1, n_1)$ is empty). Since $(a-k_1,n_1)$ is a type I lock, the box $(c_2-1, d_2)$ is also empty. Continuing as above, if $(a-k_1-(\ell-2), n_1)$ were a type I lock, then $(c_2- (\ell-1), d_2)$ would be empty, which contradicts our choice of $(c_2,d_2)$. Hence there exists a $k_2 \leq \ell-2$ so that $(a-k_1-k_2 , n_1)$ is a type II lock. This implies that there is a type I lock in some position $(a-k_1-k_2, n_2)$ with $n_1 < n_2 < m$. Continuing this, we get a sequences for $k$ and $n$ with each $k_i \leq \ell-2$ and $b < n_1 < n_2 < \dots < n_i < m$. 

I claim that each of the boxes $(a-\sum_i{k_i}, n_i)$ are in a ladder below the ladder of $(a-1,b)$. If this is the case then the sequences for $k$ and $n$ would eventually have to produce a type I locked box in column $m$ or greater, below row $a-1-(m-b)(\ell-1)$.  This contradicts that $\lambda$ is a partition, since there is no box is position $(a-1-(m-b)(\ell-1),m)$.

To show the claim, we just note that each successive type I locked box comes from moving up at most $\ell-2$ and to the right at least one space. These are clearly in a ladder below the ladder of $(a-1, b)$, since ladders move up $\ell-1$ boxes each time they move one box to the right.

\end{proof}

\begin{corollary}\label{smallestpartition}
Fix $n\in \mathbb{N}$. Suppose for any partition $\lambda \vdash n$, $\mathcal{S} \lambda$ is also a partition. Then $\mathcal{S} \lambda$ is the smallest partition (in dominance order) in the regularization class of $\lambda$. Futhermore, all boxes of $\mathcal{S} \lambda$ are locked.
\end{corollary}

\begin{proof}
Let $\mu$ be any partition in the regularization class of $\lambda$. Then by Proposition \ref{moveboxesdown} above, we can choose an arrangement $\mathcal{B}$ so that all of the boxes which we must move down from $\lambda$ to form $\mu$ are unlocked boxes. But the arrangement $\mathcal{S}_\lambda$ moves all unlocked boxes as far down as they can go, so $\mathcal{S}\lambda \leq \mu$. For the second statement, assume some box of $\mathcal{S} \lambda$ is unlocked. Then $\mathcal{S}_{\mathcal{S} \lambda}$ must move some boxes down. Following the arrangement $\mathcal{S}_{\mathcal{S} \lambda}$ yields another partition $\mathcal{S}^2\lambda$ which is smaller than $\mathcal{S} \lambda$, which contradicts the first statement23 of this corollary. 
\end{proof}

\begin{proposition}\label{algo_yields_partition}
Let $\lambda$ be a partition. Then $\mathcal{S}\lambda$ is a partition.
\end{proposition}

\begin{proof}

 There are two possibilities we must rule out: That an unlocked box in $\lambda$ which was moved down via the arrangement  $\mathcal{S}_\lambda$ has an empty position directly to the left, or that it has an empty position directly above. If we show that the algorithm never leaves an empty position above a moved unlocked box, then 
as a consequence we can easily prove that we get no empty positions to the left of a moved box.
  
Suppose the algorithm never leaves an empty position above a moved box. If there was an empty position $x$ directly to the left of a moved box $\mathcal{S}_\lambda(y)$ then below $x$ there would be a box $z$ on the same ladder as $\mathcal{S}_\lambda(y)$. If $z$ was not moved, then it must have been locked, which implies all of the boxes above it were locked, including the box which was in position $x$, contradicting that position $x$ is empty. Otherwise, $z$ moved, so our assumption implies that $z$ has no empty position above it. Applying this procedure again, we can determine that there must be a box directly above the box directly above $z$. Applying this procedure will eventually imply that there must be a box in position $x$. Therefore our goal is to show that moving down all unlocked boxes produces no box below an empty position.
 
  We prove this by induction on $n$, the number of boxes in $\lambda$. The $n = 1$ case is clear. We assume that if $\eta$ is a partition of $k<n$ then $\mathcal{S}\eta$ has no box below an empty position (and hence is a partition by the previous paragraph). We let $\lambda$ be any partition of $n$ and we will show that $\mathcal{S} \lambda$ has no box below an empty position. 
 
 The inductive proof is broken into three cases. 

The first case is that there exists an $i$ so that $\lambda_1 = \lambda_i$ and the box $x = (i, \lambda_i)$ is locked. Then all of the boxes in the first $i$ rows are also locked. Let $\mu = \lambda \setminus \{\textrm{first $i$ rows} \}$.  Boxes are locked in $\mu$ if and only if they are locked in $\mu \subset \lambda$. Since $|\mu| < |\lambda|$, $\mathcal{S} \mu$ has no box below an empty position. We append the first $i$ rows back on top of $\mathcal{S} \mu$ to form $\mathcal{S} \lambda$. Hence $\mathcal{S} \lambda$ has no box below an empty position. 

The second case is when there does not exist such an $i$. Let $j$ be so that $\lambda_1 = \lambda_j \neq \lambda_{j+1}$. If $j >1$ then we let $x$ be the box $(j,\lambda_j)$ and let $y = (j-1, \lambda_{j-1})$. Let $\mu = \lambda \setminus \{x \}$. Let $\nu = \lambda \setminus \{x,y \}$. Note that boxes in $\mu$ (and $\nu$) are locked if and only if they are locked in $\lambda$. Since $|\mu| < |\lambda|$, we can form a partition $\mathcal{S}\mu$ by bringing down all unlocked boxes. Now we place $x$ into the lowest empty position on the ladder of $x$ in $\mathcal{S}\mu$. If it is not a partition, it is because there is an empty position above where $x$ was placed. If this is the case, then the box $y$ was moved down below where $x$ was pushed down, and $y$ has a box $z$ below it (if it didn't then $x$ would move below $y$). But then in $\nu$ when we move down all of the boxes we would have $z$ below an empty position (where $y$ is in $\mu$). But $|\nu| < |\lambda|$, so $\mathcal{S}\nu$ should not have any empty positions.

Lastly, if $j = 1$ then we let $x = (1, \lambda_1)$. Since $x$ is unlocked, there is at least one empty position in the ladder for $x$ which has a box above it. Let $\mu = \lambda \setminus \{x \}$. Boxes are locked in $\mu$ if and only if they were locked in $\lambda$. We move all the unlocked boxes of $\mu$ down; by induction and Corollary \ref{smallestpartition} this is the smallest partition in the regularization class of $\mu$, denoted $\mathcal{S} \mu$. Since the number of boxes on the ladder for $x$ was at most the number of boxes on the ladder directly above $x$ in $\lambda$ (this is because $x$ cannot be a type I lock), the number of boxes on the ladder for $x$ is strictly less than the number on the ladder above $x$ in $\mu$. Therefore, there exists an empty position on the ladder for $x$ directly below a box in $\mathcal{S} \mu$. We let $(a,b)$ be the lowest such empty position. If $(a,b)$ is the lowest empty position in $\mathcal{S} \mu$ on the ladder for $x$, then moving $x$ into that space will yield a partition which is obtained from moving $x$ into the lowest empty position in its ladder. If not, then there must be an empty position on the ladder for $x$ directly below an empty position, all below row $a$. Let $(c,d)$ be the highest such empty position on the ladder for $x$ below $(a,b)$.  Let $(m_0,d)$ be the lowest box in column $d$ ($m_0$ is at least $c-\ell$, since there is a box in the space $(c-(\ell-1), d+1)$). Let $(m_1, b)$ be the box in column $b$ in the same ladder as $(m_0+1,d)$. By Corollary \ref{smallestpartition}, all of the boxes of $\mathcal{S}\mu$ are locked. But $(m_1,b)$ has the space $(m_0+1,d)$ in the same ladder, so if $(m_1,b)$ is locked it must be because a box to the right of it is locked (i.e. it is a type II locked box). Call this box $(m_1, b_1)$. We continue by letting $(m_2, b_1)$ be the box in the same ladder as $(m_0+1,d)$ in column $b_1$. Similarly, since this box is locked, there must be a box to the right of it which is locked. That box will be $(m_2, b_2)$. This process must eventually conclude at step $(m_k, b_{k-1})$ where $b_{k-1}$ is less than the column of $x$ (since every box in $\mu$ has a column value less than that of $x$) and the space $(m_k, b_{k-1} +1)$ is empty. We just need to show that the box $(m_k, b_{k-1})$ is actually a box in the diagram of $\lambda$ (i.e. that $m_k$ is at least 1).  In fact, $(m_0+1, d)$ is either in the ladder directly to the left of $x$ or below this ladder. Hence $(m_k, b_{k-1})$ is either in the ladder left of $x$ or below. It must also be in a column strictly to the left of $x$. Therefore $(m_k, b_{k-1})$ is a box in $\lambda$. But then the box $(m_k, b_{k-1})$ must be unlocked, since it has no locked boxes to the right and the empty position $(m_0+1, d)$ in the same ladder below it. This contradicts all of the boxes of $\mathcal{S} \mu$ being locked.

\end{proof}

\begin{example}
Continuing from the example above $(\lambda = (6,5,4,3,1,1)$ and $\ell = 3)$, we move all of the unlocked boxes down to obtain the smallest partition in $\mathcal{RC} (\lambda)$, which is $\mathcal{S} \lambda = (3,3,2,2,2,2,2,1,1,1,1)$. The boxes labeled $L$ are the ones which were locked in $(6,5,4,3,1,1)$ (and did not move). 

\begin{center}
$\tableau{L&L&L\\
L&L&L\\
L&L\\
L&L\\
L & \mbox{} \\
L & \mbox{}\\
 \mbox{}& \mbox{}\\
 \mbox{}\\
 \mbox{}\\
 \mbox{}\\
 \mbox{}}
$ \end{center}

\end{example}

\begin{theorem}
$\mathcal{S} \lambda$ is the unique smallest partition in its regularization class with respect to dominance order. It can be classified as being the unique partition (in its regularization class) which has all locked boxes.  
\end{theorem}

\begin{proof}
This follows from Corollary \ref{smallestpartition} and Proposition \ref{algo_yields_partition}. 
\end{proof}

\subsection{The nodes of the ladder crystal are smallest in dominance order}

The nodes of $B(\Lambda_0)^L$ have been defined recursively by applying the operators $\widehat{f}_{i}$. 
We now give a nonrecursive description which determines when a partition is a node of $B(\Lambda_0)^L$. 

\begin{proposition} Let $\lambda$ be a partition of $n$. Let $\mathcal{RC}(\lambda)$ be its regularization class. If $\lambda$ is a node of $B(\Lambda_0)^L$ then $\lambda$ is the smallest partition in $\mathcal{RC}(\lambda)$ with respect to dominance order.
\end{proposition}

\begin{proof} 
The proof is by induction on the number of boxes in $\lambda$. Suppose we have a partition in $B(\Lambda_0)^L$ which is smallest in its regularization class. Equivalently, all of the boxes of $\lambda$ are locked. We want to show that $\widehat{f_i} \lambda$ is still smallest in its regularization class. Let $x = \widehat{f_i} \lambda \setminus \lambda$, so that $\widehat{f_i}\lambda$ is just $\lambda$ with the addable $i$-box $x$ inserted. There are two cases to consider.

The first is that the insertion of $x$ into $\lambda$ makes $x$ into an unlocked box. Since the position $y$ above $x$ must be locked (or $x$ is in the first row), there must exist an empty position $x'$ in the same ladder as $x$ which has a box $y'$ directly above $x'$ in the ladder of $y$. Also, the box $z$ to the left of $x$ must be a type I lock, since $x$ was not in $\lambda$ (implying $z$ could not be type II). But then the position $z'$ to the left of $x'$ must have a box, since there is a box in the position above $z'$ (the box to the left of $y'$). Since $\lambda$ was smallest in dominance order, it could not have had a $+$ below a $-$ on the same ladder (otherwise the $-$ could be moved down to the $+$, contradicting that $\lambda$ was smallest in dominance order). Hence there is no $-$ between the two $+'s$ in positions $x$ and $x'$ of $\lambda$. This means that the crystal rule would have chosen to add to $x'$ instead of $x$, a contradiction.  

The second case is that the insertion of $x$ into $\lambda$ unlocks a box. If this is the case then let $y$ be the position above $x$. The only case to consider is adding the box $x$ unlocks some box above the row of $x$, on the ladder directly below $x$. Let $z$ denote the lowest such box. If there is a box $x'$ directly above $z$ and no box directly to the right of $x'$ then $x'$ would have been unlocked in $\lambda$, since the box $y$ is in $\lambda$ but $x$ is not. This contradicts the assumption that $\lambda$ was the smallest. 

The only other possibility is that the box $x'$ above $z$  has a box directly to the right of it (this still works even if $z$ is in the first row of the partition).  If there was a box to the right of $z$, then that box would be locked since it was locked in $\lambda$. But then $z$ would be locked independently of the addition of $x$. So $z$ is at the end of its row. The space one box to the right of $z$ (let us name it $w$) must therefore be empty and have the same residue as $x$. It is an addable box, so it will contribute a $+$ to the $i$-ladder-signature. There are no $-$ boxes on the ladder for $x$ below $x$, because if there was a removable $i$-box on the same ladder as $x$ below $x$ then the box $z$ would have been unlocked in $\lambda$. Therefore,  no $-$ will cancel the $+$ from space $w$. Since $w$ is on the ladder past the ladder for $x$, this yields a contradiction as $\widehat{f_i}$ should have added a box to position $w$.

\end{proof}

One can view $B(\Lambda_0)$ as having nodes $\{ \mathcal{RC}(\lambda) : \lambda \vdash n, n\geq 0 \}$. The usual model of $B(\Lambda_0)$ takes the representative $\mathcal{R} \lambda \in \mathcal{RC}(\lambda)$, which happens to be the largest in dominance order. Here, we will take a different representative of $\mathcal{RC}(\lambda)$, the partitions $\mathcal{S} \lambda$, which are smallest in dominance order. One must then give a description of the edges of the crystal graph. We will show in Section \ref{reg_and_crystal} that the crystal operators $\widehat{f}_i$ and $\widehat{e}_i$ are the correct operators to generate the edges of the graph. In other words, we will show that the crystal $B(\Lambda_0)^L$ constructed above is indeed isomorphic to $B(\Lambda_0)$.

\section{Extending theorems to the ladder crystal}\label{extend}

\subsection{Motivation}

In this section, we prove analogues of Theorems \ref{top_and_bottom} and \ref{other_cases} for our newly defined crystal $B(\Lambda_0)^L$, with $(\ell,0)$-JM partitions taking the place of $\ell$-partitions.
We start with a few lemmas about $B(\Lambda_0)^L$. 

\subsection{Crystal theoretic results for the ladder crystal and JM partitions}
Our goal will be to show that all $(\ell,0)$-JM partitions are in $B(\Lambda_0)^L$. First we prove that all $\ell$-cores are in $B(\Lambda_0)^L$.

\begin{lemma}\label{cores}
If $\lambda$ is an $\ell$-core, then $\widehat{\varphi} = \varphi$ and $\widehat{f}_i^{\hat{\varphi}} \lambda = \tilde{f}_i^{\varphi} \lambda$. Hence all $\ell$-cores are nodes of $B(\Lambda_0)^L$.
\end{lemma}

\begin{proof}
$\ell$-cores have either all addable or all removable $i$-boxes for any fixed $i$. Therefore, for a core $\lambda$ we have that $\varphi = \widehat{\varphi}$. Thus $\widehat{f}_i^{\varphi}$ and $\tilde{f}_i^{\varphi}$ will both add all addable $i$-boxes to $\lambda$, just in different relative orders.

For the last statement, recall that applying the crystal operators $\widetilde{f}_i^{\varphi}$ will generate all $\ell$-cores (starting with the empty partition $\emptyset$).
\end{proof}

The following lemma is a well known recharacterization of $\ell$-cores.
\begin{lemma}\label{hook_length_divisible}
A box $x$ has a hook length divisible by $\ell$ if and only if there exists a residue $i$ 
so that the last box in the row of $x$ has residue $i$ and the last box in the column of $x$ 
has residue $i+1$. In particular, any partition which has such a box $x$ is not an $\ell$-core.
\end{lemma}

The upcoming lemma is used only in the proof of Lemma \ref{cancelation}. 

\begin{lemma}\label{nodes_JM}
If $\lambda$ is an $(\ell,0)$-JM partition then for every $0\leq i < \ell$, there is 
no $i$-ladder in the Young diagram of $\lambda$ which has a $-$ above a $+$.
\end{lemma}

\begin{proof}
Suppose $\lambda$ is an $(\ell,0)$-JM partition. Then $\lambda \approx (\mu,r,s,\rho,\sigma)$.
 First suppose that $\mu = \emptyset$. Let $j$ be the residue of the addable box in the first
 row of $\lambda$. Then all removable boxes have residue either $j-1$ (in the first $r$ rows)
 or $j+1$ (in the first $s$ columns), so no ladder contains  both a $-$ and a $+$.

Suppose $\mu \neq \emptyset$. The only possible way of having a $-$ above a $+$ in a ladder is
 to have a removable box of residue $j$ on or above row $r+1$ and an addable box of residue $j$ 
on or to the left of column $s+1$. But then the core $\mu$ has a box of residue $j$ at the end of its
 first row (row $r+1$ of $\lambda$) and a box of residue $j+1$ at the end of its first column 
(on column $s+1$ of $\lambda$). This implies that $\mu$ isn't a core, by Lemma \ref{hook_length_divisible},
since $\ell \mid h_{(1,1)}^\mu$.
\end{proof}

The following Lemma \ref{cancelation} will be used in this section for proving our crystal theorem generalizations for $(\ell,0)$-JM partitions.

\begin{lemma}\label{cancelation} Let $\lambda$ be an $(\ell,0)$-JM partition. Then its ladder 
$i$-signature is the same as the reduced ladder $i$-signature. In other words, there is no 
$-+$ cancelation in the ladder $i$-signature of $\lambda$. 
\end{lemma}

\begin{proof}
Suppose there is a $-+$ cancelation in the ladder $i$-signature of an $(\ell,0)$-JM partition $\lambda$.
 By Lemma \ref{nodes_JM}, it must be that a removable $i$-box occurs on a ladder to the left of a ladder
 which contains an addable $i$-box. Suppose the removable $i$-box is in box $(a,b)$ and the addable 
$i$-box is in box $(c,d)$. We will suppose that $a > c$ (the case $a<c$ 
is similar). Then $\ell \mid h_{(c,b)}^{\lambda}$. Also $h_{(a,b)}^{\lambda} = 1$ since $(a,b)$ is a 
removable box. If all of $h_{(c,k)}^{\lambda}$ are divisible by $\ell$ for 
$k \in \{b+1, b+2, \dots, d-1 \}$, then the box $(a,b)$ could not be in a ladder to the left of the 
ladder of $(c,d)$. Therefore there is a $b<j<d$ so that $\ell \nmid h_{(c,j)}^{\lambda}$. Hence 
$\lambda$ is not an $(\ell,0)$-JM partition.
\end{proof}

\subsection{Generalizations of the crystal theorems to the ladder crystal}
We will now prove analogues of Theorems \ref{top_and_bottom} and \ref{other_cases} for $(\ell,0)$-JM partitions in the ladder crystal $B(\Lambda_0)^L$.
\begin{theorem}\label{top_and_bottom_JM}
Suppose that $\lambda$ is an $(\ell,0)$-JM partition and $0\leq i < \ell$. Then
\begin{enumerate}
\item\label{f_JM} $\widehat{f}_{i}^{\widehat{\varphi}} \lambda$ is an $(\ell,0)$-JM partition,
\item\label{e_JM} $\widehat{e}_{i}^{\widehat{\varepsilon}} \lambda$ is an $(\ell,0)$-JM partition.
\end{enumerate}
\end{theorem}
\begin{proof}
We will prove \eqref{f_JM}; \eqref{e_JM} follows similarly. Suppose that $\lambda \approx 
(\mu,r,s,\rho,\sigma)$ has an addable $m$-box in the first row, and an addable $n$-box in the 
first column for two residues $m,n$. If $m\neq i \neq n$ then $\widehat{f}_i^{\widehat{\varphi}}$
 will only add boxes to the core $\mu$ in the Young diagram of $\lambda$. But 
$\widehat{f}_{i-r+s}^{\widehat{\varphi}} \mu$ will again be a core, by Lemma \ref{cores}. 
Hence $\widehat{f}_i^{\widehat{\varphi}} \lambda \approx (\widehat{f}_{i-r+s}^{\varphi} \mu , r,s,\rho, \sigma)$.

If $m =  i \neq n$ then the partition $\nu \approx (\mu, r, 0, \rho, \emptyset)$ is an $\ell$-partition.
 From Lemma \ref{cancelation}, we have no cancelation of $-+$ in $\lambda$, so that 
$\widehat{f}_{i-s} ^{\widehat{\varphi}_{i-s} (\nu)} \nu = \tilde{f}_{i-s} ^{\varphi_{i-s} (\nu)} \nu$. 
By Theorem \ref{top_and_bottom}, $\widehat{f}_{i-s} ^{\widehat{\varphi}_{i-s} (\nu)} \nu $ is an 
$\ell$-partition. Say 
$\widehat{f}_{i-s} ^{\widehat{\varphi}_{i-s} (\nu)}\nu \approx (\mu', r',0, \rho', \emptyset)$. 
Then $\widehat{f}_i^{\widehat{\varphi}} \lambda \approx (\mu',r',s,\rho',\sigma)$. 
A similar argument works when $m \neq i = n$ by using that the transpose of 
$(\mu, 0, s , \emptyset, \sigma)$ is an $\ell$-partition. 

We now suppose that $m = i = n$. 
In this case, $\lambda$ has an addable $i$-box in the first $r+1$ rows and $s+1$ columns. It may also have
addable $i$-boxes within the core $\mu$. $\lambda$ has no removable $i$-boxes. Thus we get 
$\widehat{f}_i^{\widehat{\varphi}} \lambda \approx (\widehat{f}_{i-r+s}^{\widehat{\varphi}_{i-r+s}(\mu)} \mu, r, s, \rho, \sigma)$ 
is an $(\ell,0)$-JM partition.
\end{proof}

\begin{theorem}\label{other_cases_JM}
Suppose that $\lambda$ is an $(\ell,0)$-JM partition. Then
\begin{enumerate}
\item\label{f_theorem_JM} $\widehat{f}_{i}^{k} \lambda$ is not an $(\ell,0)$-JM partition for $0<k < \widehat{\varphi} -1,$
\item\label{e_theorem_JM} $\widehat{e}_{i}^k \lambda$ is not an $(\ell,0)$-JM partition for $1<k< \widehat{\varepsilon}$.
\end{enumerate}
\end{theorem}
\begin{proof}
We will prove \eqref{f_theorem_JM}, and we claim that \eqref{e_theorem_JM} is similar. 
We will prove this by showing that there must exist boxes $(a,b)$ , $(a,y)$ and $(x,b)$ 
in the Young diagram of $\lambda$ such that $\ell \mid h_{(a,b)}^\lambda$ and 
$\ell \nmid h_{(a,y)}^\lambda, \ell \nmid h_{(x,b)}^\lambda$. Since $0< k < \widehat{\varphi}-1$ 
there are at least two ladder conormal $i$-boxes in $\widehat{f}_{i}^{k} \lambda$ and at least one 
ladder normal $i$-box. Label the normal box $n_1$.  Label the conormal boxes $n_2$ and $n_3$ 
(without loss of generality, $n_2$ is in a row above $n_3$). We will show that any partition
 with at least one ladder normal $i$-box and at least two ladder conormal $i$-boxes cannot
 be an $(\ell,0)$-JM partition. There are three cases to consider.

The first case is that $n_1$ is above $n_2$ and $n_3$. Then the hook length in the row of $n_1$
 and column of $n_3$ is divisible by $\ell$, but the hook length in the row of $n_2$ and column
 of $n_3$ is not. Also, the hook length for box $n_1$ is 1, which is not divisible by $\ell$. 

The second case is that $n_1$ is in a row between the row of $n_2$ and $n_3$. In this case, $\ell$
 divides the hook length in the row of $n_1$ and column of $n_3$. Also $\ell$ does not divide the
 hook length in the row of $n_2$ and column of $n_3$, and the hook length for the box $n_1$ is 1. 

The last case is that $n_1$ is below $n_2$ and $n_3$. In this case, $\ell$ divides the hook length
 in the column of $n_1$ and row of $n_2$, but $\ell$ does not divide the hook length in the column
 of $n_3$ and row of $n_2$. Also the hook length for the box $n_1$ is 1. 

\end{proof}

\subsection{All JM partitions are nodes of the ladder crystal}

This short subsection contains the proof that all $(\ell,0)$-JM partitions occur as nodes of $B(\Lambda_0)^L$. This surprising fact is extremely important to the proofs of Theorems \ref{top_and_bottom_weak} and \ref{other_cases_weak}. 

\begin{theorem}\label{irreducible_nodes} If $\lambda$ is an $(\ell,0)$-JM partition then $\lambda$ is a node of $B(\Lambda_0)^L$.

\begin{proof}
The proof is by induction on the size of a partition. If the partition has size zero then it is the empty partition which is an $(\ell,0)$-JM partition and is a node of the crystal $B(\Lambda_0)^L$.

Suppose $\lambda \vdash n$ is an $(\ell,0)$-JM partition. Let $i$ be a residue so that $\lambda$ has at least one ladder conormal box of residue $i$. We can find such a box since no $-+$ cancellation exists by Lemma \ref{cancelation}. 
Define $\mu$ to be $\widehat{e}_i^{\widehat{\varepsilon}} \lambda$. Then $\mu \vdash (n-\hat{\varepsilon})$ is an $(\ell,0)$-JM partition by Theorem \ref{top_and_bottom_JM}, of smaller size than $\lambda$. By induction $\mu$ is a node of $B(\Lambda_0)^L$. But $\widehat{f}_i^{\widehat{\varepsilon}} \mu= \widehat{f}_i^{\widehat{\varepsilon}} \widehat{e}_i^{\widehat{\varepsilon}} \lambda = \lambda$ by Lemma \ref{e_inverse_f}, so $\lambda$ is a node of $B(\Lambda_0)^L$.
\end{proof}

\end{theorem}

\section{Reinterpreting the crystal rule}\label{crystal_lemmas}

In this short section we prove some lemmas necessary for our main theorem (that the crystals $B(\Lambda_0)$ and $B(\Lambda_0)^L$ are isomorphic). We also reinterpret the crystal rule on the classical $B(\Lambda_0)$ in terms of the new crystal rule.

\subsection{Two lemmas needed for crystal isomorphism}
\begin{lemma}\label{nodes}
If $\lambda \in B(\Lambda_0)^L$ then for every $0\leq i < \ell$, there is no $i$-ladder in the Young 
diagram of $\lambda$ which has a $-$ above a $+$ (i.e. there do not exist boxes $(a,b)$ and $(c,d)$ on
 the same $i$-ladder of $\lambda$ so that $(a,b)$ is removable, $(c,d)$ is addable and $a < c$). 
\end{lemma}

\begin{proof}
If there was a $-$ box above an empty $+$ box on the same ladder, then one could move that $-$ box down 
to the $+$ space to form a new partition in the same regularization class. But this cannot happen since
 the nodes of $B(\Lambda_0)^L$ are smallest in dominance order.
\end{proof}

\begin{lemma}\label{ell_regular_nodes}
If $\lambda \in B(\Lambda_0)$ then for every $0\leq i < \ell$, there is no $i$-ladder in the Young diagram of $\lambda$ which has a $+$ above a $-$ (i.e. there do not exist boxes $(a,b)$ and $(c,d)$ on the same $i$-ladder of $\lambda$ so that $(a,b)$ is addable, $(c,d)$ is removable and $a < c$). 
\end{lemma}
\begin{proof}
Since $\lambda$ is in $B(\Lambda_0)$, it is $\ell$-regular. Suppose such positions $(a,b)$ and $(c,d)$ exist. Then the rim hook starting at box $(a, b-1)$ and following the border of $\lambda$ down to the $(c,d)$ box will cover exactly $(b-d) \ell$ boxes. However, this border runs over only $b-d$ columns, so there must exist a column which contains at least $\ell$ boxes at the ends of their rows. This contradicts $\lambda$ being $\ell$-regular. 
\end{proof}

\subsection{Reinterpreting the classical crystal rule}

This short theorem proves that the classical crystal rule can be interpreted in terms of ladders. In fact 
the only difference between the classical rule and the ladder crystal rule is that ladders 
are read bottom to top instead of top to bottom.

\begin{theorem}\label{regular_rule}
The $i$-signature (and hence reduced $i$-signature) of an $\ell$-regular partition $\lambda$ in 
$B(\Lambda_0)$ can be determined by reading from its leftmost ladder to rightmost ladder, bottom to top.
\end{theorem}

\begin{proof}
If positions $x$ and $y$ contain $+'s$ (or $-'s$) of the same residue with the ladder 
of $x$ to the left of the ladder of $y$, then by the regularity of $\lambda$, $x$ will 
be in a row below $y$. Hence reading the $i$-signature up ladders from left to right is 
equivalent to reading up the rows of $\lambda$ from bottom to top.

\end{proof}

\begin{example}
$\lambda = (6,5,3,3,2,2,1)$ and $\ell = 3$. Suppose we wanted to find the 2-signature for $\lambda$. 
Then we could read from leftmost ladder (in this case, the leftmost ladder relevant to the 2-signature
 is the one which contains position (8,1)) to rightmost ladder. Inside each ladder we read from bottom 
to top . The picture below shows the positions which 
correspond to addable and removable 2-boxes, with their ladders. The 2-signature is $+---$.

\begin{center}
$\begin{array}{cc}
\tableau{\mbox{}&\mbox{}&\mbox{}&\mbox{}&\mbox{}&-\\
\mbox{}&\mbox{}&\mbox{}&\mbox{}&\mbox{}\\
\mbox{}&\mbox{}&\mbox{}\\
\mbox{}&\mbox{}&-\\
\mbox{}&\mbox{}\\
\mbox{}&-\\
\mbox{}}
&

\put(-112,-122){$+$}

\put (-112,-122){\line(1,2){70}}
\put (-85,-122){\line(1,2){70}}

\end{array}
$
\end{center}
\end{example}

\section{Regularization and crystal isomorphism}\label{reg_and_crystal}
The results of this section come from ideas originally sketched out with Steve Pon in the summer of 2007.
\subsection{Crystal isomorphism}\label{crystal_isomorphism}
We will now prove that our crystal $B(\Lambda_0)^L$ from Section \ref{new_crystal} is indeed isomorphic 
to the crystal $B(\Lambda_0)$.  We first classify positions in a Young diagram. Let $\lambda$ be
 an arbitrary partition and suppose $x = (i,j)$ with $i,j \geq 0$. We use the convention that
 $x$ is in $\lambda$ if either of $i,j$ are not positive. We say the type of $x$ with respect
 to $\lambda$ is:
\begin{enumerate}
    \renewcommand{\labelenumi}{(\alph{enumi})}
\item if the box $x$ is in $\lambda$ and the box $(i+1, j+1)$ is in $\lambda$.
\item if the box $x$ is in $\lambda$, the boxes $(i+1. j)$ and $(i,j+1)$ are in $\lambda$ and $(i+1, j+1)$ is not in $\lambda$.
\item if the box $x$ is in $\lambda$, $j>0$, the box $(i+1, j)$ is in $\lambda$ and $(i,j+1)$ is not in $\lambda$.
\item if the box $x$ is in $\lambda$, the box $(i, j+1)$ is in $\lambda$ and $(i+1,j)$ is not in $\lambda$.
\item if the box $x$ is in $\lambda$ and $(i+1,j), (i,j+1)$ are not in $\lambda$.
\item if $j=0$ and the box $(i,1)$ is not in $\lambda$.
\item if the box $x$ is not in $\lambda$, and $(i-1,j-1)$ is of type $(e)$.
\item if the box $x$ is not in $\lambda$, and $(i-1,j-1)$ is of type $(c)$.
\item if the box $x$ is not in $\lambda$, and $(i-1,j-1)$ is of type $(d)$.
\item if the box $x$ is not in $\lambda$, and $(i-1,j-1)$ is of type $(b)$.
\item if the box $x$ is not in $\lambda$, and $(i-1,j-1)$ is of type $(f)$ or not in $\lambda$.
\end{enumerate}
\begin{remark}
A box $x$ is removable if and only if it is type $(e)$ and addable if and only if it is type $(j)$
\end{remark}
\begin{example}
Let $\lambda = (4,2,1,1)$. Then the boxes are labeled by type in the following picture. 

\begin{center}
$\begin{array}{cc}
\tableau{a&b&d&fe\\
b&e\\
c\\
e} &

\put(-93, -83){f}
\put(-75, -83){k}
\put(-57, -83){k}
\put(-39, -83){k}
\put(-21, -83){k}
\put(-3, -83){k}
\put(15, -83){k}
\put(-93, -65){f}
\put(-93, -47){b}
\put(-93, -29){a}
\put(-93, -11){a}
\put(-93, 7){a}
\put(-93, 25){a}
\put(-75, 25){a}
\put(-57, 25){a}
\put(-39, 25){a}
\put(-21, 25){b}
\put(-3, 25){d}
\put(15, 25){d}
\put (-3,7){j}
\put (15,7){i}
\put (-39,-11){j}
\put (-21,-11){i}
\put (-3,-11){g}
\put (15,-11){k}
\put (-57,-29){j}
\put (-39,-29){g}
\put (-21,-29){k}
\put (-3,-29){k}
\put (15,-29){k}
\put (-57,-47){h}
\put (-39,-47){k}
\put (-21,-47){k}
\put (-3,-47){k}
\put (15,-47){k}
\put (-75,-65){j}
\put (-57,-65){g}
\put (-39,-65){k}
\put (-21,-65){k}
\put (-3,-65){k}
\put (15,-65){k}

\end{array}$
\end{center}
\end{example}

\begin{definition}
The $k^{th}$ ladder of $\lambda$ will refer to all of the boxes $(i,j)$ of $\lambda$ which are on the same ladder as $(k,1)$.
\end{definition}

\begin{definition}
The number of boxes of type $\alpha$ in ladder $k$ of the partition $\lambda$ will be denoted $n_\alpha^k (\lambda)$. 
\end{definition}

The following lemma is the basis of the proof that the  two crystals $B(\Lambda_0)$ and $B(\Lambda_0)^L$ are isomorphic. 

\begin{lemma}\label{ladderlemma}  Let $\lambda$ be a node of $B(\Lambda_0)^L$. Let $\mathcal{L}$ be
 the $k^{th}$ ladder of $\lambda$
and $\mathcal{K}$ be the $k^{th}$ ladder of $\mathcal{R}\lambda$. Let $i$ be the residue of the 
ladders $\mathcal{L}$ and $\mathcal{K}$. Then the number of ladder-normal $i$-boxes on $\mathcal{L}$ 
is the same as the number of normal $i$-boxes on $\mathcal{K}$. Similarly, the number of 
ladder-conormal $i$-boxes on $\mathcal{L}$ is the same as the number of conormal $i$-boxes 
on $\mathcal{K}$. In particular, the ladder-good (ladder-cogood) $i$-box of $\lambda$ lies on
 the same ladder as the good (cogood) $i$-box of $\mathcal{R} \lambda$.
\end{lemma}

\begin{proof}
Suppose we are looking at the boxes on $\mathcal{L}$. Then note that:
\begin{itemize}
\item $n_b^k (\lambda) + \min (n_c^k (\lambda), n_d^k (\lambda) ) = n_b^k (\mathcal{R} \lambda)$.

\item $n_e^k (\lambda) + \min (n_c^k (\lambda), n_d^k (\lambda) ) = n_e^k (\mathcal{R} \lambda)$.

\end{itemize}
This follows from the fact that a box of type $(c)$ and a box of type $(d)$ on $\mathcal{L}$ 
 will combine under regularization to form a box of type $(b)$ and a box of type $(e)$ on $\mathcal{K}$. 

Similarly, 
\begin{itemize}
\item $n_g^k (\lambda) + \min (n_h^k (\lambda), n_i^k (\lambda) ) = n_g^k (\mathcal{R} \lambda)$.

\item $n_j^k (\lambda) + \min (n_h^k (\lambda), n_i^k (\lambda) ) = n_j^k (\mathcal{R} \lambda)$.
\end{itemize}

Lemma \ref{nodes} implies that there is no cancelation in the ladder $i$-signature on $\mathcal{L}$
 (i.e. no type $(e)$ node above a type $(j)$ node).
Similarly, Lemma \ref{ell_regular_nodes} implies that there is no cancelation in the $i$-signature on $\mathcal{K}$
 (i.e. no type $(e)$ node below a type $(j)$ node).

We also note that a box of type $(b)$ on the $k^{th}$ ladder of a partition 
implies that there is an addable $i$-box on the $(k+\ell)^{th}$ 
ladder, and that a box of type $(g)$ on the $k^{th}$ ladder of a partition 
implies that there is a removable $i$-box on the $(k-\ell)^{th}$ ladder. 

$\mathcal{K}$ will contribute $\underbrace{+ \dots +}_{n_j^k(\lambda)} \underbrace{- \dots -}_{n_e^k(\lambda)}$ to the ladder $i$-signature of $\lambda$. 
$\mathcal{L}$ will contribute 
$\underbrace{+ \dots +}_{n_j^k(\lambda)}  \underbrace{+ \dots +}_{\min(n_h^k(\lambda), n_i^k(\lambda))} \underbrace{- \dots -}_{n_e^k(\lambda)} \underbrace{- \dots -}_{\min(n_c^k(\lambda), n_d^k(\lambda))}$ 
to the $i$-signature of $\mathcal{R} \lambda$. 
Since $\min (n_c^k(\lambda), n_d^k(\lambda)) = min(n_h^{k+\ell}(\lambda), n_i^{k+\ell}(\lambda))$, the 
$\min (n_c^k(\lambda), n_d^k(\lambda))$ $-$'s 
will cancel with the $\min (n_h^{k+\ell}(\lambda), n_i^{k+\ell}(\lambda))$ $+$'s 
in the $k+\ell^{th}$ ladder, leaving just the $n_e^k(\lambda)$ $-$'s in the $k^{th}$ ladder. 
Similarly, only the $n_j^k(\lambda)$ $+$'s will count in the $i$-signature of $\mathcal{R} \lambda$.
Hence the number of normal and conormal boxes on any ladder is unchanged by regularization.

Since there are the same number of
 (co)normal boxes on each ladder of $\mathcal{R} \lambda$ as there are ladder-(co)normal boxes in each
 ladder of $\lambda$, the (co)good box of $\mathcal{R}\lambda$ must be on the same ladder as the
 ladder (co)good box of $\lambda$.
\end{proof}

\begin{corollary}
Let $\lambda$ be a node in $B(\Lambda_0)^L$. Then $\widehat{\varphi}_i(\lambda) = \varphi_i(\mathcal{R}\lambda)$ and $\widehat{\varepsilon}_i(\lambda) = \varepsilon_i(\mathcal{R}\lambda)$.
\end{corollary}

\begin{theorem}\label{crystal_commutes}
Regularization commutes with the crystal operators. In other words:

\begin{enumerate}
\item $\mathcal{R} \circ \widehat{f}_i = \tilde{f}_i \circ \mathcal{R}$,
\item $\mathcal{R} \circ \widehat{e}_i = \tilde{e}_i \circ \mathcal{R}$.
\end{enumerate}

\end{theorem}

\begin{proof}
(1) will follow from Lemma \ref{ladderlemma}, since applying an $\widehat{f}_i$ to a partition $\lambda$ will place an $i$-box in the same ladder of $\lambda$ as $\tilde{f}_i$ places in $\mathcal{R} \lambda$. (2) follows similarly.
\end{proof}

\begin{corollary}\label{crystalisisom}
The crystal $B(\Lambda_0)$ is isomorphic to $B(\Lambda_0)^L$.
\end{corollary}

\begin{proof}
The map $\mathcal{R} : B(\Lambda_0)^L \to B(\Lambda_0)$ gives the isomorphism.
The map $\mathcal{S}: B(\Lambda_0) \to B(\Lambda_0)^L$ described in Section \ref{locked} is the inverse of $\mathcal{R}$.
The other crystal isomorphism axioms are routine to check.
\end{proof}

\begin{example}
Let $\lambda = (2,1,1,1)$ and $\ell = 3$. Then $\mathcal{R} \lambda = (2,2,1)$. Also $\widehat{f}_2 \lambda = (2,1,1,1,1)$ and $\tilde{f}_2 (2,2,1) = (3,2,1)$. But $\mathcal{R} (2,1,1,1,1) = (3,2,1)$. 

\begin{center}
$
\begin{array}{cc} (2,1,1,1) \xrightarrow{\widehat{f}_2} (2,1,1,1,1)\\
\\
(2,2,1)\;\;\;\; \xrightarrow{\widetilde{f}_2} \;\;\;\;(3,2,1) 
&
\put(-30,48){\vector(0,-1){18}}
\put(-100,48){\vector(0,-1){18}}
\put(-120,35){$\mathcal{R}$}
\put(-20,35){$\mathcal{R}$}
\end{array}$
\end{center}

\end{example}

\subsection{Crystal rules for weak l-partitions}

We first recall the following result due to James, which says that $D^{\mathcal{R}\lambda}$ occurs with  multiplicity one in $S^\lambda$.

\begin{theorem}[James \cite{J2}]\label{decomp} Let $\lambda$ be any partition. Then
$d_{\lambda, \mathcal{R}\lambda} = 1$.
\end{theorem}

We can now prove our generalizations of Theorem \ref{top_and_bottom} and Theorem \ref{other_cases}.

\begin{theorem}\label{top_and_bottom_weak}
Suppose that $\lambda$ is a weak $\ell$-partition and $0\leq i < \ell$. Then
\begin{enumerate}
\item $\widetilde{f}_{i}^{\varphi} \lambda$ is a weak $\ell$-partition,
\item $\widetilde{e}_{i}^{\varepsilon} \lambda$ is a weak $\ell$-partition.
\end{enumerate}
\end{theorem}

\begin{proof} Let $\lambda$ be a weak $\ell$-partition. Then $D^{\lambda} = S^{\nu}$ for some $(\ell,0)$-JM
 partition $\nu$ with $\mathcal{R} \nu = \lambda$ (see Proposition \ref{reg_prop}). From Theorem 
\ref{irreducible_nodes} we know that $\nu \in B(\Lambda_0)^L$.
Corollary \ref{crystalisisom} implies that $\hat{\varphi}_i(\nu) = \varphi$. By Theorem 
\ref{top_and_bottom_JM}, $\widehat{f}_i^{\varphi} \nu$ is another $(\ell,0)$-JM partition. 
By Theorem \ref{crystal_commutes} we know that $\mathcal{R} \widehat{f}_i^{\varphi} \nu = \tilde{f}_i^{\varphi} \lambda$.
 Theorem \ref{decomp} then implies that $D^{\tilde{f}_i^{\varphi} \lambda} = S^{\widehat{f}_i^{\varphi} \nu}$
, since $S^{\widehat{f}_i^{\varphi} \nu}$ is irreducible by Theorem \ref{JM_irred}. 
Hence $\tilde{f}_i^{\varphi} \lambda$ is a weak $\ell$-partition. (2) follows similarly.
\end{proof}

\begin{theorem}\label{other_cases_weak}
Suppose that $\lambda$ is a weak $\ell$-partition. Then
\begin{enumerate}
\item $\widetilde{f}_{i}^{k} \lambda$ is not a weak $\ell$-partition for $0<k < \varphi -1,$
\item $\widetilde{e}_{i}^k \lambda$ is not a weak $\ell$-partition for $1<k< \varepsilon$.
\end{enumerate}
\end{theorem}

\begin{proof}
Let $\lambda$ be a weak $\ell$-partition. We must show that there does not exist an $(\ell,0)$-JM partition
$\mu$ in the regularization class of $\tilde{f}_i^k \lambda$. There exists an 
$(\ell,0)$-JM partition $\nu$ in $B(\Lambda_0)^L$ so that $D^{\lambda} = S^{\nu}$. By Theorem 
\ref{other_cases_JM}, $\widehat{f}_i^{k} \nu$ is not an $(\ell,0)$-JM partition. But by Theorem
 \ref{irreducible_nodes} we know all $(\ell,0)$-JM partitions occur in $B(\Lambda_0)^L$. Also, 
only one element of $\mathcal{RC}(\widetilde{f}_i^k \lambda)$ occurs in $B(\Lambda_0)^L$ and we
 know this is $\widehat{f}_i^k \nu$. Therefore no such $\mu$ can exist, so
 $\widetilde{f}_i^k \lambda$ is not a weak $\ell$-partition. (2) follows similarly.
\end{proof}

\section{Representation theoretic proof of Theorem \ref{top_and_bottom_weak}}\label{new_proof_JM}

This section is devoted to a representation theoretic proof of Theorem \ref{top_and_bottom_weak}. All definitions needed for this section can be found in \ref{new_proof}.

\indent T\footnotesize HEOREM \normalsize\ref{top_and_bottom_weak}.
\textit{Suppose that $\lambda$ is a weak $\ell$-partition and $0\leq i < \ell$. Then
\begin{enumerate}
\item $\widetilde{f}_{i}^{\varphi} \lambda$ is a weak $\ell$-partition,
\item $\widetilde{e}_{i}^{\varepsilon} \lambda$ is a weak $\ell$-partition.
\end{enumerate}
}

\begin{proof}[Alternate Proof of Theorem \ref{top_and_bottom_weak}] Suppose $\lambda$ is a weak $\ell$-partition and $| \lambda| = n$. Recall that $\lambda$ is a weak $\ell$-partition if and only if there is a $\mu$ so that $S^{\mu} = D^{\lambda}$. Let $F$ denote the number of addable $i$-boxes of $\mu$ and let $\nu$ denote the partition corresponding to $\mu$ plus all addable $i$-boxes.

First, we induce $S^{\mu}$ from $H_n(q)$ to $H_{n+F}(q)$. Applying  Proposition \ref{branching_rule} $F$ times yields $$\displaystyle [Ind_n^{n+F} S^{\mu}] = \sum_{\eta_{F} \succ \eta_{F-1} \succ \dots \succ \eta_1 \succ \mu} [S^{\eta_{F}}].$$ 
Projecting to the direct summand which has the same central character as $D^{\nu}$ (i.e. the central character of $\mu$ with $F$ more $i$'s), we get that $[S^{\nu}]$ occurs in this sum with coefficient $F !$ (add the $i$-boxes in any order). Everything else in this sum has a different central character. 

We next  
apply \eqref{ind} from Theorem \ref{groj} $F$ times to obtain 
$$[Ind_n^{n+F} D^{\lambda}] = \bigoplus_{i_1, \dots, i_{F}} [f_{i_1} \dots f_{i_{F}} {D^{\lambda}}].$$

Projecting again to the direct summand with the same central character as $\nu$, in $K(Rep)$ we get exactly $[(f_i)^F D^{\lambda}]$. Since $\lambda$ is a weak $\ell$-partition, $S^{\mu} = D^{\lambda}$, so $Ind_n^{n+F} S^{\mu} = Ind_n^{n+f} D^{\lambda}$ and we have shown that $F![S^{\nu}] = [(f_i)^F D^{\lambda}]$.

Since $S^{\mu} = D^{\lambda}$ and $(f_i)^{\varphi+1} D^{\lambda} = 0$, we know that $F \leq \varphi$ $(Ind_n^{n+\varphi+1} D^{\lambda}$ has no direct summand with the correct central character by Theorem \ref{character}). Similarly, since $Ind_n^{n+F+1} S^{\mu}$ has no composition factors with central character $\chi(D^{\lambda})$ union with $F+1$ many $i$'s, we know that $F \geq \varphi$. Hence $F = \varphi$. 

By part \ref{f^phi} of Theorem \ref{groj},  $[(f_i)^{\varphi} D^{\lambda}] = \varphi ! [\widetilde{f}_i^{\varphi} D^{\lambda}]$. Then by Theorem \ref{tilda}, $[S^{\nu}] = [D^{\widetilde{f}_i^{\varphi}\lambda}]$. 

Therefore $S^{\nu} = D^{\widetilde{f}_i^{\varphi} \lambda}$. Hence $\widetilde{f}_{i}^{\varphi} \lambda$ is a weak $\ell$-partition.

  The proof that $\widetilde{e}_{i}^{\varepsilon_i(\lambda)} \lambda$ is a weak $\ell$-partition follows similarly, with the roles of induction and restriction changed in Proposition \ref{branching_rule}, and the roles of $e_i$ and $f_i$ changed in Theorem \ref{groj}.

\end{proof}

\section{The Mullineux map}\label{MullMap}
 The operation on the category of $H_n(q)$ modules of tensoring with the sign module is a functor which takes irreducible modules to irreducible modules. For instance, when $q$ is not a root of unity, then $S^\lambda \otimes sign \cong S^{\lambda'}$.
  When $\lambda$ is an $\ell$-regular partition, and $D^\lambda$ denotes the irreducible module corresponding to $\lambda$ then $D^\lambda \otimes sign$ is some irreducible module $D^{m(\lambda)}$. This describes a map $m$ on $\ell$-regular partitions called the \textit{Mullineux map}. Recent results of Fayers \cite{F2} settle a conjecture of Lyle \cite{L2} which effectively computes the Mullineux map in certain cases by means of regularization and transposition. This section will highlight the interpretation of Fayers result in terms of the ladder crystal. It should be noted that Ford and Kleshchev gave a recursive construction for computing the Mullineux map \cite{FK}.
 
 \subsection{Statement of Fayers results}
 
 Since the Mullineux map is the modular analog of transposition, a natural  attempt to compute $m(\lambda)$ for an $\ell$-regular partition $\lambda$ would be to transpose $\lambda$ and then regularize. If a partition is not $\ell$-regular, we could similarly guess that $m(\mathcal{R}\lambda)$ was the composition of transposition and regularization. This is not always the case. However, a conjecture of Lyle \cite{L2}, which was proven recently by Fayers \cite{F2} gives a precise classification for when this holds. The definition below was taken from Fayers \cite{F2}.
 
 \begin{definition} An L-partition is a partition which has no box $(i,j)$ in the diagram of 
$\lambda$ such that $\ell \mid h_{i,j}^\lambda$ and either $\mathrm{arm}(i,j) < (\ell-1) * \mathrm{leg}(i,j)$ 
or $\mathrm{leg}(i,j) < (\ell-1) * \mathrm{arm}(i,j)$.
 \end{definition}
 
 \begin{theorem}\label{Lpart}[Fayers \cite{F2}] A partition is an L-partition if and only if 
$m(\mathcal{R} \lambda) = \mathcal{R} (\lambda')$. 
 \end{theorem}

 \subsection{All L-partitions occur as nodes in the ladder crystal}
 Now we prove that all L-partitions are nodes in the crystal $B(\Lambda_0)^L$. To do this, we will use an equivalent condition of L-partitions given by Lyle in \cite{L2}. 
 
 \begin{lemma} A partition $\lambda$ is an L-partition if and only if there does not exist a box $(i,j)$ in the diagram of $\lambda$ such that $\ell \mid h_{(i,j)}^\lambda$, and $\frac{h_{(i,j)}^\lambda}{\ell} \leq \min \{\mathrm{arm}(i,j), \mathrm{leg}(i,j) \}$.
 \end{lemma}
 
 \begin{proof}
 Suppose that $\lambda$ is not an L-partition. Let $(i,j)$ be such that $\ell \mid h_{i,j}^\lambda$ and
 suppose that $\mathrm{arm}(i,j) < (\ell-1)* \mathrm{leg}(i,j)$ 
(the case where $\mathrm{leg}(i,j) < (\ell-1)*\mathrm{arm}(i,j)$ is similar). Then
 $h_{i,j}^\lambda = \mathrm{arm}(i,j) + \mathrm{leg}(i,j) +1 < \ell * \mathrm{leg}(i,j) +1$. 
So $\frac{h_{(i,j)}^\lambda}{\ell} <  \mathrm{leg}(i,j) + \frac{1}{\ell}$, or 
$\frac{h_{(i,j)}^\lambda}{\ell} \leq  \mathrm{leg}(i,j)$. Therefore $\frac{h_{(i,j)}^\lambda}{\ell} 
\leq  \min \{\mathrm{arm}(i,j), \mathrm{leg}(i,j)\}$. 
 
 Now suppose there exists a box $(i,j)$ so that $\ell \mid h_{(i,j)}^\lambda$, and 
$$\frac{h_{(i,j)}^\lambda}{\ell} \leq \min \{\mathrm{arm}(i,j), \mathrm{leg}(i,j) \}.$$ Let
 $k = \min \{\mathrm{arm}(i,j), \mathrm{leg}(i,j) \}$. If $k = \mathrm{arm}(i,j)$, then
 $h_{i,j}^\lambda \leq \ell*\mathrm{arm}(i,j)$, so 
$\mathrm{leg}(i,j) = h_{i,j}^\lambda - \mathrm{arm}(i,j) - 1 \leq (\ell-1)*\mathrm{arm}(i,j) -1$.
 Therefore $\mathrm{leg}(i,j) < (\ell-1)*\mathrm{arm}(i,j)$. The $k = \mathrm{leg}(i,j)$ case follows 
similarly. 
 \end{proof}
 
 \begin{theorem}\label{lylepart}
 All L-partitions occur as nodes of the crystal $B(\Lambda_0)^L$.
 \end{theorem}
 
 \begin{proof}
 Suppose that $\lambda$ is not a node of the crystal $B(\Lambda_0)^L$. Then there exists an unlocked box somewhere in the diagram of $\lambda$, say at position $(a,b)$. Let $c$ be the smallest integer so that $(c,b)$ is unlocked (so that either $(c-1,b)$ is locked or $c =1$). Then the box $(c, \lambda_c)$ is also unlocked (if it were locked then $(c,b)$ would be a type II locked box). If there is an unlocked box above $(c, \lambda_c)$ then we will apply this procedure again (move to the highest unlocked box in column $\lambda_c$ and then the rightmost box in that row) until we can assume that $(c,\lambda_c)$ is unlocked and sits below a locked box (or possibly $c=1$). Then $(c, \lambda_c)$ must be unlocked from violating the type I lock condition, so that there exists a position $(i,j)$ in the same ladder as $(c, \lambda_c)$ such that $(i,j)$ is not in $\lambda$ but $(i-1,j)$ is in $\lambda$. Then the hook length $h_{(c,j)}^\lambda$ is divisible by $\ell$. In fact, $h_{(c,j)}^\lambda = \ell * \mathrm{arm}(c,j)$, 
since $(i,j)$ is in the same ladder as $(c, \lambda_c)$ and $(i,j)$ is the first non-occupied position in column $j$ in $\lambda$. But then 
$\frac{h_{(c,j)}^\lambda}{\ell} = \mathrm{arm}(c,j)$ and $\mathrm{leg}(c,j) = (\ell-1)*\mathrm{arm}(c,j) - 1 \geq arm(c,j)$ since $\ell > 2$. 
Hence $\frac{h_{(c,j)}^\lambda}{\ell} \leq min(\mathrm{arm}(c,j), \mathrm{leg}(c,j))$, 
which means $\lambda$ is not an L-partition.
 
  \end{proof}
 
 It was pointed out to the author by M. Fayers that a classification of the nodes of $B(\Lambda_0)^L$ can be described in terms of hook lengths and arm lengths. We now include this classification.
 
 \begin{theorem} A partition $\lambda$ belongs to the crystal $B(\Lambda_0)^L$ if and only if there does not exist a box $(i,j) \in \lambda$ such that $h_{(i,j)}^\lambda = \ell * \mathrm{arm}(i,j)$.
 \end{theorem}
 
 \begin{proof}
 If $\lambda$ has a box $(i,j)$ with hook length $\ell*\mathrm{arm}(i,j)$ then the box $(i, \lambda_i)$ will be unlocked, as it is on the same ladder as the empty position directly below the last box in column $j$. Hence $\lambda$ is not in $B(\Lambda_0)^L$. 
 
 If $\lambda$ does not belong to the crystal $B(\Lambda_0)^L$ then there exists an unlocked box $(a,b)$ in the diagram of $\lambda$. Similar to the proof of Theorem \ref{lylepart}, this implies that there exists an unlocked box $(i,\lambda_i)$ which is either directly below a locked box or $i=1$. 
This box is then unlocked because it violates a type I lock rule, which means that in the same ladder as
 $(i,\lambda_i)$ there is an empty position $(n,j)$ directly below a box $(n-1,j)$. But then the box $(i,j)$ will satisfy $h_{(i,j)}^\lambda = \ell * \mathrm{arm}(i,j)$.
 \end{proof} 
  
 Finally we note that as a corollary we obtain a second proof that all $(\ell,0)$-JM partitions are nodes of the crystal $B(\Lambda_0)^L$.
 
 \begin{corollary} All $(\ell,0)$-JM partitions (in fact, all $(\ell,p)$-JM partitions) are nodes of the crystal $B(\Lambda_0)^L$.
 \end{corollary}
 
 \begin{proof}
The two composition factors $D^{\mathcal{R} (\lambda')}$ and $D^{m(\mathcal{R}\lambda)}$ must occur in
 the Specht module $S^{\lambda'}$. If $\lambda$ is a JM partition then so is $\lambda'$, so 
$S^{\lambda'}$ is irreducible, i.e. $S^{\lambda'} \cong D^{\mathcal{R}(\lambda')} = D^{m(\mathcal{R} \lambda)}$. By Theorem \ref{Lpart}, $\lambda$ is an L-partition. Hence all JM-partitions are L-partitions.
 \end{proof}
        

    %
    %

    \newchapter{Core Bijection}{A Bijection on Core Partitions}{A Bijection on Core Partitions and a Parabolic Quotient of the Affine Symmetric Group}
    \label{sec:LabelForChapter4}

        
    
   The results of this chapter are joint work with Brant Jones and Monica Vazirani.
\section{Introduction}
 In Section \ref{counting_ell_cores}, we gave a simple bijection between $\ell$-cores with first part $k$ and $(\ell-1)$-cores with first part $\leq k$. 
In this chapter we review some combinatorial models for cores and interpret the bijection in various guises.  In the Coxeter system setting, the bijection has a geometric interpretation as a projection from the root lattice of type $A_{\ell-1}$ to an embedded copy of the root lattice of type $A_{\ell-2}$; see Figure~\ref{f:sl3}.  We observe that the bijection reduces the Coxeter length of the corresponding minimal length coset representative by exactly $k$.  We also show that the bijection has a natural description in terms of another correspondence between $\ell$-cores and $(\ell-1)$-bounded partitions due to Lapointe and Morse \cite{LM}.

\subsection{Organization}

In Section~\ref{s:intro} we review the bijection $\Phi_{\ell}^k$ in terms of partition diagrams.  In Section~\ref{s:geometry}, we review the correspondence between $\ell$-cores and minimal length coset representatives for $\widetilde{S_{\ell}} / S_{\ell}$ where $\widetilde{S_{\ell}}$ denotes the affine symmetric group and $S_{\ell}$ denotes the finite symmetric group.  In Section~\ref{s:geometric_bijection}, we give a geometric version of the bijection $\Phi_{\ell}^k$ on the root lattice of type $A_{\ell-1}$.  In Section~\ref{s:l-m}, we show that $\Phi_{\ell}^k$ also has a natural description in terms of bounded partitions using the correspondence
\[ \rho_{\ell-1} : \{ \ell\text{-cores} \} \rightarrow \{ \text{partitions with first part } \leq \ell-1 \} \]
due to Lapointe and Morse \cite{LM}.

\section{Definitions, notation and a review of the bijection}\label{s:intro}

\subsection{Preliminaries}
We let $\delta_{i,j}$ denote the Kronecker delta function.  

The subset of $\mathcal{C}_{\ell}$ having first part $k$ will be denoted $\mathcal{C}_{\ell}^k$ and the subset of $\mathcal{C}_{\ell}$ having first part $\leq k$ will be denoted $\mathcal{C}_{\ell}^{\leq k}$. 

\subsection{Abaci}  Here we generalize the notion of $\beta$-numbers from Chapter \ref{sec:LabelForChapter2}.

Each partition $\lambda = (\lambda_1, \dots, \lambda_r)$ is determined by its hook lengths in the first column, i.e. the $h_{(i,1)}^{\lambda}$. From a sequence $(\alpha_1, \dots, \alpha_r)$ of positive decreasing integers one obtains a partition $\mu$ by requiring that the hook length $h_{(i,1)}^{\mu} = \alpha_i$ for $1 \leq i \leq r$. This gives a bijection between the set of partitions and the set of strictly decreasing sequences of positive integers. 

One can generalize this process by looking at the set $B$ of infinite sequences $b = (b_1, b_2, \dots )$ of integers. We give $B$ the group structure of component-wise addition. We define the element \textbf{1} $= (1,1,1,1, \dots) \in B$. Let $S$ denote the subgroup generated by \textbf{1} under addition, so $S = \{ (n,n,n, \dots ) : n \in \mathbf{Z} \}$. A sequence $b = (b_1, b_2, \dots ) \in B$ is said to \textit{stabilize} if there exists an $n$ so that $b_i - b_{i+1} = 1$ for all $i > n$. The set $\mathcal{B}$ is defined to be the subset of $B$ of strictly decreasing sequences that stabilize, modulo the added relation $\equiv$ that two sequences are equivalent if their difference is in $S$.
\begin{example}
$(11, 7, 4, 1, -1, -2, -3, \dots)$ is in $\mathcal{B}$. In $\mathcal{B}$, we have 
\[ (9, 5, 2, -1, -3, -4, -5, \dots) \equiv (11, 7, 4, 1, -1, -2, -3, \dots). \]
\end{example}

We define a bijection $\beta$ between the set $\mathcal{P}$ of partitions and $\mathcal{B}$. To a partition $\lambda = (\lambda_1, \lambda_2, \dots \lambda_r)$ of length $r$, we define $\beta(\lambda)$ to be the equivalence class of $(h_{(1,1)}^{\lambda}, h_{(2,1)}^{\lambda}, h_{(3,1)}^{\lambda}, \dots , h_{(r,1)}^{\lambda}, -1, -2, -3, -4, \dots)$ in $\mathcal{B}$.  

\begin{example}
$\beta(8,5,3,1)$ is the equivalence class of $(11, 7, 4, 1, -1, -2, -3, \dots).$
\end{example}

We generalize some definitions from Chapter \ref{sec:LabelForChapter2} about abaci.

An \textit{abacus diagram} is a diagram containing $\ell$ columns labeled $0, 1, \dots, \ell-1$, called \textit{runners}.  The horizontal cross-sections or rows will be called \em levels \em and runner $i$ contains entries labeled by $r \ell + i$ on each level $r$ where $-\infty < r < \infty$.  We draw the abacus so that each runner is vertical, oriented with $-\infty$ at the top and $\infty$ at the bottom, with runner 0 in the leftmost position, increasing to runner $\ell-1$ in the rightmost position.  Entries in the abacus diagram may be circled; such circled elements are called \textit{beads}. Entries which are not circled will be called \textit{gaps}.  The linear ordering of the entries given by the labels $r \ell + i$ is called the \em reading order \em of the abacus and corresponds to scanning left to right, top to bottom.

\begin{example}
The following abacus diagram has beads in positions ($\ldots$, -3, -2, -1, 1, 2, 4, 5, 8) and gaps in positions (0, 3, 6, 7, 9, 10, 11, $\dots$). Level 0 is the row which contains $0,1,2$. 
\noindent
\\
\\

\begin{center}
\begin{picture}(80,120)(0,10)
\put (42.5,96){.}
\put (42.5,92){.}
\put (42.5,100){.}
\put (72.5,96){.}
\put (72.5,92){.}
\put (72.5,100){.}
\put (11.5,96){.}
\put (11.5,92){.}
\put (11.5,100){.}
\put (9,80) {-3}
\put (39,80){-2}
\put (69,80){-1}
\put (10,65) {0}
\put (40,65){1}
\put (70,65){2}
\put (10,50){3}
\put (40, 50){4}
\put (70,50){5}
\put (10,35){6}
\put (40,35){7}
\put (70,35){8}
\put (10,20){9}
\put (38,20){10}
\put (68,20){11}
\put(0,105){\line(1,0){85}}
\tiny
\put (0,120){Runner}
\put (30,120){Runner}
\put (60,120){Runner}
\normalsize
\put (10,110){0}
\put (40,110){1}
\put (70,110){2}
\put (-48,65){Level 0 $\to$}
\put (-50,80){Level -1 $\to$}
\put (-48,50){Level 1 $\to$}
\put (-48,35){Level 2 $\to$}

\put (72.5,53){\circle{13}}
\put (72.5,68){\circle{13}}
\put (72.5,38){\circle{13}}
\put (42.5,68){\circle{13}}
\put (42.5,53){\circle{13}}

\put (72.5,83){\circle{13}}
\put (42.5,83){\circle{13}}
\put (12.5,83){\circle{13}}
\put (42.5,12){.}
\put (42.5,8){.}
\put (42.5,4){.}
\put (72.5,12){.}
\put (72.5,8){.}
\put (72.5,4){.}
\put (11.5,12){.}
\put (11.5,8){.}
\put (11.5,4){.}
\end{picture}
\end{center}
\end{example}

A representative $\omega$ of $\beta(\lambda)$ will be called a \textit{set of $\beta$-numbers} for $\lambda$.  Suppose $\omega = (\omega_1, \omega_2, \dots )$ is a set of $\beta$ numbers for $\lambda$. An \em abacus \em for $\lambda$ is obtained by circling the entries of $\omega$ in an abacus diagram. 

\begin{example}
The following two diagrams are abaci for $\lambda = (8,5,3,1)$, the first comes from the $\beta$-numbers $(11,7,4,1,-1,-2,-3 , \dots)$ and the second comes from the equivalent $\beta$-numbers $(9,5,2,-1,-3,-4,-5, \dots )$.  We list the beads in reverse reading order to be compatible with stability in $\mathcal{B}$.
\\\\
$$
\begin{array}{lccccccr}
\begin{picture}(80,80)
\put (42.5,96){.}
\put (42.5,92){.}
\put (42.5,100){.}
\put (72.5,96){.}
\put (72.5,92){.}
\put (72.5,100){.}
\put (11.5,96){.}
\put (11.5,92){.}
\put (11.5,100){.}
\put (8,80) {-3}
\put (38,80){-2}
\put (68,80){-1}
\put (10,65) {0}
\put (40,65){1}
\put (70,65){2}
\put (10,50){3}
\put (40, 50){4}
\put (70,50){5}
\put (10,35){6}
\put (40,35){7}
\put (70,35){8}
\put (10,20){9}
\put (38,20){10}
\put (66,20){11}
\put (72.5,23){\circle{13}}
\put (42.5,68){\circle{13}}
\put (42.5,53){\circle{13}}
\put (42.5,38){\circle{13}}

\put (72.5,83){\circle{13}}
\put (42.5,83){\circle{13}}
\put (12.5,83){\circle{13}}
\put (42.5,12){.}
\put (42.5,8){.}
\put (42.5,4){.}
\put (72.5,12){.}
\put (72.5,8){.}
\put (72.5,4){.}
\put (11.5,12){.}
\put (11.5,8){.}
\put (11.5,4){.}
\end{picture}
& & & & & & &
\begin{picture}(80,80)
\put (42.5,96){.}
\put (42.5,92){.}
\put (42.5,100){.}
\put (72.5,96){.}
\put (72.5,92){.}
\put (72.5,100){.}
\put (11.5,96){.}
\put (11.5,92){.}
\put (11.5,100){.}

\put (8,80) {-6}
\put (38,80){-5}
\put (68,80){-4}
\put (8,65) {-3}
\put (38,65){-2}
\put (68,65){-1}
\put (10,50){0}
\put (40, 50){1}
\put (70,50){2}
\put (10,35){3}
\put (40,35){4}
\put (70,35){5}
\put (10,20){6}
\put (40,20){7}
\put (70,20){8}
\put (10,5){9}
\put (38,5){10}
\put (68,5){11}
\put (72.5,53){\circle{13}}
\put (72.5,68){\circle{13}}
\put (72.5,38){\circle{13}}
\put (12.5,68){\circle{13}}
\put (12.5,8){\circle{13}}
\put (72.5,83){\circle{13}}
\put (42.5,83){\circle{13}}
\put (12.5,83){\circle{13}}
\put (42.5,-3){.}
\put (42.5,-7){.}
\put (42.5,-11){.}
\put (72.5,-3){.}
\put (72.5,-7){.}
\put (72.5,-11){.}
\put (11.5,-3){.}
\put (11.5,-7){.}
\put (11.5,-11){.}
\end{picture}
\end{array}
$$
\end{example}

\begin{remark}
Note that an abacus for $\lambda$ is not unique because it depends on the set of $\beta$-numbers chosen for $\lambda$.  However, from any abacus of $\lambda$ one can obtain the partition $\lambda$ by counting the number of gaps before every bead in the abacus in reading order. In the first example above for instance, we see that $\lambda_1 = 8$ since the eight numbers 10,9,8,6,5,3,2,0 are exactly the eight gaps before the bead corresponding to the final bead at position $11$. We will say that a bead is \textit{active} if it 
occurs in the positions between the first gap and the last bead, in reading order.  The active beads are those that correspond to a nonzero part in a partition.  In the left example above, the bead in spot 11 is active since it corresponds to the part $\lambda_1 = 8$, whereas the bead in spot -1 is not active since it corresponds to $\lambda_5 = 0$.
\end{remark}

\begin{definition}
We define the \textit{balance number} of an abacus to be the sum over all runners of the largest level in that runner which contains a bead. We say that an abacus is \textit{balanced} if its balance number is zero.
\end{definition}

\begin{example}
In the example above, the balance number of the first diagram is $-1+2+3 = 4$. The balance number for the second diagram is $3 + -2 + 1 = 2$, so neither are balanced.
\end{example}
\begin{remark}\label{r:unique_abacus}
Note that there is a unique abacus which represents a given partition for each balance number. In particular, there is a unique abacus of $\lambda$ with balance number 0.  The balance number for a set of $\beta$-numbers of $\lambda$ will increase by exactly $1$ when the vector \textbf{1} is added to the set of $\beta$-numbers. On the abacus picture, this corresponds to shifting all of the beads forward one entry in the reading order.
\end{remark}


\begin{definition}
A runner is called \textit{flush} if no bead on the runner is preceded in reading order by a gap on that same runner.  We say that an abacus is \textit{flush} if every runner is flush.
\end{definition}

\begin{theorem}\cite[Theorem 2.7.16, Lemma 2.7.38]{JK}\label{t:flush_abacus}
$\lambda$ is an $\ell$-core if and only if any (equivalently, every) abacus of $\lambda$ is flush.  Moreover, in the balanced flush abacus of an $\ell$-core $\lambda$, each active bead on runner $i$ corresponds to a row of $\lambda$ whose rightmost box has residue $i$.
\end{theorem}

In the case that the corresponding abacus is not balanced, the boxes corresponding to the active beads on runner $i$ will share the same residue, but the residue may not be $i$.

\begin{example}\label{52111}
One can check that the partition $\lambda = (5,2,1,1,1)$ is a 4-core. One set of $\beta$-numbers for $\lambda$ is $(8,4,2,1,0,-2,-3,-4, \dots )$. This abacus is balanced as $2+0+0+(-2)=0$. All of the runners are flush. The active beads on runner 0 lie in positions 8, 4, 0 and these correspond to rows 1, 2 and 5 of the partition diagram whose final box of residue 0 is highlighted.

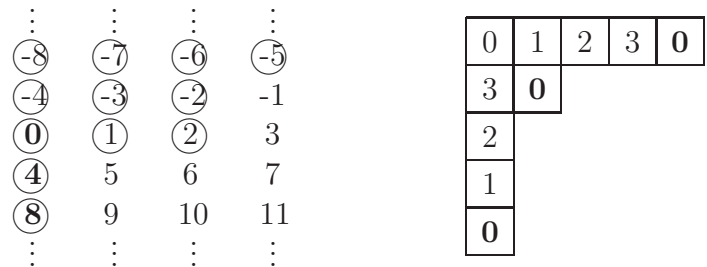
\begin{figure}[ht]
\begin{center}
\begin{tabular}{llll}
\begin{picture}(105,40)(0,80)
\put (42.5,96){.}
\put (42.5,92){.}
\put (42.5,100){.}
\put (72.5,96){.}
\put (72.5,92){.}
\put (72.5,100){.}
\put (102.5,96){.}
\put (102.5,92){.}
\put (102.5,100){.}
\put (11.5,96){.}
\put (11.5,92){.}
\put (11.5,100){.}
\put (9,80) {-8}
\put (39,80){-7}
\put (69,80){-6}
\put (99,80){-5}
\put (9,65) {-4}
\put (39,65){-3}
\put (69,65){-2}
\put (99,65){-1}

\put (10,50){\bf 0}
\put (40, 50){1}
\put (70,50){2}
\put (101,50){3}

\put (10,35){\bf 4}
\put (40,35){5}
\put (70,35){6}
\put (101,35){7}

\put (10,20){\bf 8}
\put (40,20){9}
\put (68,20){10}
\put (99,20){11}

\put (72.5,53){\circle{13}}
\put (72.5,68){\circle{13}}
\put (72.5,83){\circle{13}}

\put (102.5,83){\circle{13}}

\put (42.5,53){\circle{13}}
\put (42.5,68){\circle{13}}
\put (42.5,83){\circle{13}}

\put (12.5,23){\circle{13}}
\put (12.5,38){\circle{13}}
\put (12.5,53){\circle{13}}
\put (12.5,68){\circle{13}}
\put (12.5,83){\circle{13}}

\put (42.5,12){.}
\put (42.5,8){.}
\put (42.5,4){.}
\put (72.5,12){.}
\put (72.5,8){.}
\put (72.5,4){.}
\put (102.5,12){.}
\put (102.5,8){.}
\put (102.5,4){.}
\put (11.5,12){.}
\put (11.5,8){.}
\put (11.5,4){.}
\end{picture}
& & \hspace{0.5in}  &
\parbox[t]{2in}{
\tableau{\mbox{0} & \mbox{1} & \mbox{2} & \mbox{3} & \mbox{\bf 0} \\
         \mbox{3} & \mbox{\bf 0} \\
         \mbox{2} \\
         \mbox{1} \\
         \mbox{\bf 0} \\
} } \\
 & & & \\
\end{tabular}
\end{center}
\caption{This abacus represents the 4-core $(5,2,1,1,1)$.  The boxes of the corresponding partition diagram have been filled with their residue. }\label{f:ex_pi}
\end{figure}
\end{example}

\subsection{The bijection on abacus configurations}

Here we recall the bijection $\Phi_{\ell}^k : \mathcal{C}_{\ell}^k \rightarrow \mathcal{C}_{\ell-1}^{\leq k}$ from Section \ref{ell_core_bijection}.  Given $\lambda \in \mathcal{C}_{\ell}^k$ and an abacus for $\lambda$, remove the whole runner which contains the largest bead. Place the remaining runners into an $(\ell-1)$ abacus in order. In other words, renumber the runners $0, \ldots, \ell-2$, keeping the levels of the entries as before. This will correspond to an $(\ell - 1)$-core $\mu$ with largest part at most $k$. Then we define $\Phi_{\ell}^k : \mathcal{C}_{\ell}^k \rightarrow \mathcal{C}_{\ell-1}^{\leq k}$ to be the map which takes $\lambda$ to $\mu$. Observe that $ \Phi_{\ell}^k$ is well defined, independent of the choice of abacus for $\lambda$. 

To see that $\Phi_{\ell}^k$ is a bijection, observe that the map can be reversed.  Starting from an abacus of the $(\ell-1)$-core $\mu$, insert a new flush runner whose largest bead occurs just after the $k^{\textrm{th}}$ gap in the reading order.  This yields a flush abacus for the $\ell$-core $\lambda$ with $\lambda_1 = k$.

\begin{example}
Let $\ell = 4$ and $\lambda = (8,5,2,2,1,1,1)$. An abacus for $\lambda$ is:

\begin{center}
\begin{picture}(80,100)
\put (42.5,96){.}
\put (42.5,92){.}
\put (42.5,100){.}
\put (72.5,96){.}
\put (72.5,92){.}
\put (72.5,100){.}
\put (102.5,96){.}
\put (102.5,92){.}
\put (102.5,100){.}
\put (11.5,96){.}
\put (11.5,92){.}
\put (11.5,100){.}
\put (8,80) {-8}
\put (38,80){-7}
\put (68,80){-6}
\put (98,80){-5}
\put (8,65) {-4}
\put (38,65){-3}
\put (68,65){-2}
\put (98,65){-1}

\put (10,50){0}
\put (40, 50){1}
\put (70,50){2}
\put (101,50){3}

\put (10,35){4}
\put (40,35){5}
\put (70,35){6}
\put (101,35){7}

\put (10,20){8}
\put (40,20){9}
\put (67,20){10}
\put (99,20){11}

\put (72.5,53){\circle{13}}
\put (72.5,68){\circle{13}}
\put (72.5,83){\circle{13}}

\put (102.5,83){\circle{13}}

\put (42.5,53){\circle{13}}
\put (42.5,68){\circle{13}}
\put (42.5,83){\circle{13}}

\put (72.5,23){\circle{13}}
\put (72.5,38){\circle{13}}
\put (102.5,68){\circle{13}}
\put (12.5,83){\circle{13}}

\put (42.5,12){.}
\put (42.5,8){.}
\put (42.5,4){.}
\put (72.5,12){.}
\put (72.5,8){.}
\put (72.5,4){.}
\put (102.5,12){.}
\put (102.5,8){.}
\put (102.5,4){.}
\put (11.5,12){.}
\put (11.5,8){.}
\put (11.5,4){.}
\end{picture}
\end{center}

The largest $\beta$-number is 10. Removing the whole runner containing the 10, we get the remaining diagram with runners relabeled for $\ell = 3$
\begin{center}
\begin{picture}(80,100)
\put (42.5,96){.}
\put (42.5,92){.}
\put (42.5,100){.}
\put (72.5,96){.}
\put (72.5,92){.}
\put (72.5,100){.}
\put (102.5,96){.}
\put (102.5,92){.}
\put (102.5,100){.}
\put (11.5,96){.}
\put (11.5,92){.}
\put (11.5,100){.}
\put (8,80) {-6}
\put (38,80){-5}
\put (68,80){$\times$}
\put (98,80){-4}
\put (8,65) {-3}
\put (38,65){-2}
\put (68,65){$\times$}
\put (98,65){-1}
\put (10,50){0}
\put (40, 50){1}
\put (68,50){$\times$}
\put (101,50){2}
\put (10,35){3}
\put (40,35){4}
\put (68,35){$\times$}
\put (101,35){5}
\put (10,20){6}
\put (40,20){7}
\put (68,18){$\times$}
\put (101,20){8}
\put (72.5,53){\circle{13}}
\put (72.5,68){\circle{13}}
\put (72.5,83){\circle{13}}
\put (102.5,83){\circle{13}}
\put (42.5,83){\circle{13}}
\put (42.5,68){\circle{13}}
\put (42.5,53){\circle{13}}
\put (72.5,22){\circle{13}}
\put (72.5,38){\circle{13}}
\put (102.5,68){\circle{13}}
\put (12.5,83){\circle{13}}
\put (42.5,12){.}
\put (42.5,8){.}
\put (42.5,4){.}
\put (72.5,12){.}
\put (72.5,8){.}
\put (72.5,4){.}
\put (102.5,12){.}
\put (102.5,8){.}
\put (102.5,4){.}
\put (11.5,12){.}
\put (11.5,8){.}
\put (11.5,4){.}
\end{picture}
\end{center}
These are a set of $\beta$-numbers for the partition $(2,1,1)$, which is a 3-core with largest part $\leq 8$.  For the reverse bijection when $k=8$, notice that the eighth gap is at entry 7 which dictates where we insert the new runner and beads.  Also note in this example that the first abacus has balance number -1 while its image has balance number -3, so balance number is not necessarily preserved.
\end{example}

\subsection{The bijection on core partitions}\label{s:phi_on_diagram}

Another way to describe $\Phi_{\ell}^k$ is on the Young diagram of $\lambda$. Applying $\Phi_{\ell}^k$ to $\lambda$ is the same as removing all of the rows $i$ of $\lambda$ for which $ h_{(i,1)} \equiv h_{(1,1)} \mod \ell$. 
To illustrate, we show the bijection on the same example $\lambda = (8,5,2,2,1,1,1)$, but performed on a Young diagram instead of an abacus. We start by drawing the Young diagram and writing the hooks lengths of the boxes in the first column. The bijection simply deletes the rows which have a hook length in the first column equivalent to the hook length $h_{(1,1)}^{\lambda}\mod \ell$.
\noindent

$$
\tableau{14 & \mbox{} & \mbox{} & \mbox{} & \mbox{} & \mbox{} & \mbox{} & \mbox{} \\
10 & & \mbox{} & \mbox{} & \mbox{} \\
6 & \mbox{}\\
{\bf 5} & \mbox{}\\
{\bf 3}\\
2\\
{\bf 1}} 
\put (-275,8){\small Hook length $\equiv 14 \, mod \, 4 \to$}
\put (-275,-10){\small Hook length $\equiv 14 \, mod \, 4 \to$}
\put (-275,-28){\small Hook length $\equiv 14 \, mod \, 4 \to$}
\put (-275,-82){\small Hook length $\equiv 14 \, mod \, 4 \to$}
\put (10,-65){ \vspace{0.2in} $\stackrel{\Phi_{4}^8}{\mapsto}$}
\put (40,-71){\tableau{\mbox{} & \mbox{} \\ \mbox{} \\ \mbox{}}}
$$

Deleting the corresponding rows, we get that $\Phi_4^8 (8,5,2,2,1,1,1) = (2,1,1)$.
\normalsize

\section{Cores and the action of the affine symmetric group on the finite root lattice}\label{s:geometry}

In this section we recall that the $\ell$-cores index a system of minimal length coset representatives for $\widetilde{S_{\ell}} / S_{\ell}$ and describe some associated geometry.

\subsection{The affine root system}\label{s:aff_geometry}
Following \cite{humphreys}, let $\{ \e_1, \e_2, \dots, \e_{\ell} \}$ be an orthonormal basis of the Euclidean space $\mathbf{R}^{\ell}$ and denote the corresponding inner product by $(\cdot, \cdot)$.
For $1 \leq i \leq \ell - 1$, let $s_i$ be the reflection defined by interchanging $\e_i$ and $\e_{i+1}$; the reflecting hyperplanes are discussed below.  Then $\{s_1, \ldots, s_{\ell-1}\}$ are a set of Coxeter generators for the symmetric group $S_{\ell}$, which acts on $\mathbf{R}^{\ell}$ by permuting coordinates in the $\e_i$ basis.

Let $s_0$ be the affine reflection of $\mathbf{R}^{\ell}$ defined on $v = \sum_{j=1}^{\ell} a_j \e_j$ by
\[ s_0(v) = (a_{\ell} + 1) \e_1 + a_2 \e_2 + \cdots + a_{\ell-1} \e_{\ell-1} + (a_{1} - 1) \e_{\ell}. \]
Define the \em simple roots \em $\Delta$ \em of type $A_{\ell-1}$ \em to be the collection of $\ell-1$ vectors
\[ \a_1 = \e_1 - \e_2, \ \ \a_2 = \e_2 - \e_3, \ \ \ldots, \ \ \a_{\ell-1} = \e_{\ell-1} - \e_{\ell}. \]

The $\mathbf{Z}$-span $\Lambda_R$ of $\Delta$ is called the \em root lattice of type $A_{\ell-1}$\em.
Let $V = \mathbf{R} \otimes_{\mathbf{Z}} \Lambda_R \subsetneq \mathbf{R}^{\ell}$.  Observe that each reflection $s_i$ preserves $V$ and so
$\{s_0, s_1, \ldots, s_{\ell-1}\}$ are a set of Coxeter generators for the affine symmetric group $\widetilde{S_{\ell}}$ acting on $V$.  For $w \in \widetilde{S_{\ell}}$ we let $l(w)$ denote Coxeter length.  From now on, we restrict our attention from $\mathbf{R}^{\ell}$ to $V$.  

In this presentation we see that $S_{\ell}$ is a parabolic subgroup of $\widetilde{S_{\ell}}$.  We form the parabolic quotient 
\[ \widetilde{S_{\ell}} / S_{\ell} = \{ w \in \widetilde{S_{\ell}} : l(w s_i) > l(w) \text{ for all $s_i$ where $1 \leq i \leq \ell-1$} \}. \]
By a standard result in the theory of Coxeter groups, this set gives a unique representative of minimal length from each coset $w S_{\ell}$ of $\widetilde{S_{\ell}} / S_{\ell}$.  For more on this construction, see \cite[Section 2.4]{bjorner-brenti}.
Another standard result is that $\widetilde{S_{\ell}}$ acts on $V$ as the semidirect product of $S_{\ell}$ and the translation group corresponding to the root lattice $\Lambda_R$.  
Hence, $\widetilde{S_{\ell}} / S_{\ell}$ is also in bijection with $\Lambda_R$ and we identify $\Lambda_R$ with the translation subgroup $\{t_{\mathbf{a}} : \mathbf{a} \in \Lambda_R\}$ of $\widetilde{S_{\ell}}$.


Let us consider this situation geometrically.  Denote the set of \em finite roots \em by $\Pi = \{ w \a_{i} : w \in S_{\ell}, \a_i \in \Delta \} \subset V$.  It is a standard fact that each root $\a \in \Pi$ can be written as an integral linear combination of the simple roots $\Delta$ such that all of the coefficients are positive or all coefficients are negative.  Therefore, $\Pi$ can be decomposed as $\Pi = \Pi^{+} \uplus \Pi^{-}$.


For each finite root $\a$ and integer $k$ we can define an affine hyperplane
\[ H_{\a, k} = \{ v \in V : (v, \a) = k \}. \]
Observe that $s_i$ is the reflection over the hyperplane $H_{\a_i,0}$ for $1 \leq i < \ell$ while $s_0$ is the reflection over $H_{\theta, 1}$ where $\theta = \e_1 - \e_{\ell} = \sum_{i=1}^{\ell-1} \a_i$.  Let $\mathcal{H}$ denote the collection of all affine hyperplanes $H_{\a, k}$ for $\a \in \Pi, k \in \mathbf{Z}$.  Let $\mathcal{A}$ be the set of all connected components of $V \setminus \bigcup_{H \in \mathcal{H}} H$.  Each element of $\mathcal{A}$ is called an \em alcove\em.  In particular,
\[ A_{\circ} = \{ v \in V : 0 < (v, \a) < 1 \text{ for all } \a \in \Pi^{+} \} \]
is called the \em fundamental alcove \em whose closure is a fundamental domain for the action of $\widetilde{S_{\ell}}$ on $V$.  


\begin{proposition}\cite[Section 4.5]{humphreys}
The affine Weyl group $\widetilde{S_{\ell}}$ permutes the collection of alcoves transitively and freely.
The closure of $A_{\circ}$ is a fundamental domain for the action of $\widetilde{S_{\ell}}$ on $V$.
\end{proposition}

Define 
\[ B_{\circ} = \bigcup_{w \in S_{\ell}} w A_{\circ}. \]
The set $B_{\circ}$ contains one alcove for each permutation in $S_{\ell}$ and the closure of $B_{\circ}$ is a fundamental domain for the action of translation by $\Lambda_R$ on $V$.  Moreover, for any affine permutation $t_{\mathbf{a}} w \in \Lambda_R \rtimes S_{\ell}$ we have that $B_{\circ} + \mathbf{a}$ corresponds to the left coset of $S_{\ell}$ containing $t_{\mathbf{a}} w$ in $\widetilde{S_{\ell}} / S_{\ell}$.
In fact, the set of all $B_{\circ} + \mathbf{a}$ is precisely the $\widetilde{S_{\ell}}$-orbit of $B_{\circ}$.
Since length can be computed by counting the minimal number of hyperplanes which must be crossed in a path back to $A_{\circ}$, we have that $\mathbf{a}$ determines the alcove of $B_{\circ} + \mathbf{a}$ that represents $t_{\mathbf{a}} w$ and has minimal length:  it is the alcove which requires the fewest such hyperplane crossings.  This describes a bijection that we denote $\varpi : \Lambda_R \rightarrow \widetilde{S_{\ell}} / S_{\ell}$ in which $\mathbf{a} \mapsto t_{\mathbf{a}} w_{\mathbf{a}}$.
Figure~\ref{f:sl3} shows the Euclidean space $V$ associated to type $A_{2}$ in the context of the bijection $\Phi_{\ell}^k$.

\subsection{Cores are minimal length coset representatives}

We now show how an $\ell$-core can be associated to each $\mathbf{a} \in \Lambda_R$.
Let $\mathbf{a} = (a_1, \ldots, a_{\ell})$ be a vector in $\Lambda_R$ written with respect to the $\e_i$ basis, so each $a_i \in \mathbf{Z}$ and $\sum_{i = 1}^{\ell} a_i = 0$.  We form a balanced flush abacus from $\mathbf{a}$ by filling the $(i-1)^{\mathrm{st}}$ runner with beads from $-\infty$ down to level $a_i$.  By Remark~\ref{r:unique_abacus} and Theorem~\ref{t:flush_abacus}, every $\ell$-core has exactly one balanced flush abacus.  Hence, we obtain bijections whose composition we denote by $\pi$.
\[ \pi: \{ (a_1, \ldots, a_{\ell}) : a_i \in \mathbf{Z}, \sum_{i=1}^{\ell} a_i = 0 \} \rightarrow \{ \text{balanced flush abaci} \} \rightarrow \mathcal{C}_{\ell}. \]

\begin{remark}
Because the runners of an abacus are usually labeled by $0 \leq i \leq \ell-1$ but coordinates of $\mathbf{R}^{\ell}$ are labeled by $1 \leq j \leq \ell$ we will sometimes coordinatize $\Lambda_R$ as $\mathbf{a} = (a_1, \ldots, a_{\ell})$ or as $\mathbf{b} = (b_0, \ldots, b_{\ell-1})$.
\end{remark}

\begin{example}
Let $\ell = 4$ and let $\mathbf{a} = 2\e_1 + 0\e_2 +0\e_3 -2 \e_4$. Then we draw an abacus as shown in Figure~\ref{f:ex_pi} above with beads down to level 2 in runner 0, level 0 in runner 1, level 0 in runner 2, and level -2 in runner 3.
\end{example}

Next, we observe that the $\ell$-cores inherit an action of $\widetilde{S_{\ell}}$ from the bijection $\pi$.  To describe the action, we draw the diagram of $\lambda = (\lambda_1, \ldots, \lambda_r)$ and fill each box $(i,j)$ with the residue $j-i \mod \ell$.  Fix $\lambda \in \mathcal{C}_{\ell}$ and suppose $\mathbf{b} = (b_0, \ldots, b_{\ell-1}) = \pi^{-1}(\lambda)$.  For $0 \leq i \leq \ell-1$, we say that $s_i$ is an \em ascent \em for $\lambda$ if $b_{i-1} > b_{i} - \delta_{i,0}$, and we say that $s_i$ is a \em descent \em for $\lambda$ if $b_{i-1} < b_{i} - \delta_{i,0}$.  Here, we interpret $b_{-1}$ as $b_{\ell-1}$.

\begin{remark}
Note that the definition for $s_i$ to be a descent (respectively, ascent, neither) given above corresponds to $l(\varpi(s_i \mathbf{b})) < l(\varpi(\mathbf{b}))$ (respectively, $>$, $=$).  
\end{remark}

\begin{example}
The 4-core $(5,2,1,1,1)$ corresponds to $w = s_0 s_1 s_2 s_3 s_2 s_1 s_0 = t_{(2,0,0,-2)}$.  Hence, $s_1$ and $s_3$ are ascents for $\lambda$ and $s_0$ is a descent for $\lambda$.  Observe that $s_2$ is neither an ascent nor a descent because $s_2 w = s_0 s_1 s_2 s_3 s_2 s_1 s_0 s_2$ ceases to be a minimal length left coset representative.  Correspondingly $s_2 (2, 0, 0, -2) = (2, 0, 0, -2)$.
\end{example}

\begin{proposition}\label{p:action}
Let $\lambda$ be an $\ell$-core.  If $s_i$ is an ascent for $\lambda$ then $s_i$ acts on $\lambda$ by adding all boxes with residue $i$ to $\lambda$ such that the result is a partition.  If $s_i$ is a descent for $\lambda$ then $s_i$ acts on $\lambda$ by removing all of the boxes with residue $i$ that lie at the end of both their row and column so that their removal results in a partition.  If $s_i$ is neither an ascent nor a descent for $\lambda$ then $s_i$ does not change $\lambda$.
\end{proposition}

\begin{proof}
Begin by considering the action of $s_i$ on abaci that comes from the action on $\Lambda_R \subset V$.  Applying $s_i$ for $1 \leq i \leq \ell-1$ corresponds to exchanging adjacent runners $i-1$ and $i$ in the balanced flush abacus whose runners are labeled $0, \ldots, \ell-1$.  Applying the $s_0$ generator first adds a bead to runner $\ell-1$ and removes a bead from runner $0$ so that they stay flush, and then exchanges the two runners.  

Since the coordinates of $\mathbf{b} \in \Lambda_R$ sum to 0, Theorem~\ref{t:flush_abacus} implies that each active bead in runner $i$ of the balanced flush abacus corresponds to a row of $\lambda$ whose rightmost box has residue $i$.  Because the abacus is flush, exchanging runners $i$ and $i-1$ either adds some set of boxes with residue $i$ to the diagram of $\lambda$ in the case that $s_i$ is an ascent, or else removes a set of boxes with residue $i$ in the case that $s_i$ is a descent.  The result is again a balanced flush abacus so corresponds to an $\ell$-core.  If $s_i$ is neither an ascent nor a descent then $b_i = b_{i-1} + \delta_{i,0}$ so the abacus remains unchanged.

Observe that a box with residue $i$ is removable if and only if it lies at the end of its row and column.  This occurs if and only if it corresponds to an active bead on runner $i$ with a gap immediately preceding it in the reading order of the abacus.  Similarly, a box with residue $i$ is addable if and only if it corresponds to a gap on runner $i$ with an active bead immediately succeeding it in the reading order of the abacus.  The action of $s_i$ swaps runners $i$ and $i-1$ which therefore interchanges all of the $i$-addable and $i$-removable boxes.
\end{proof}

\begin{remark}
This action can be described as a special case of the action on the crystal graph associated to the irreducible highest-weight representation $V(\Lambda_0)$ of $\widehat{\mathfrak{sl}_{\ell}}$ that is given by the operators $\widetilde{f_i}^{\varphi_i - \varepsilon_i}$ or $\widetilde{e_i}^{\varepsilon_i - \varphi_i}$.
\end{remark}


Let $\lambda$ be an $\ell$-core.  Then we can recursively define a canonical reduced expression for $\varpi(\pi^{-1}(\lambda))$ that we denote $w(\lambda)$ by choosing $w(\lambda) = s_i w(\widehat{\lambda})$ where $i$ is the residue of the rightmost box in the bottom row of $\lambda$ and $\widehat{\lambda}$ is the result of applying $s_i$ to $\lambda$ as in Proposition~\ref{p:action}.  Note that $s_i$ is always a descent for this choice of $i$.  The empty partition corresponds to the identity Coxeter element.  This reduced expression was previously defined in \cite[Definition 45]{LM}.  

\begin{example}
The canonical reduced expression for the 4-core $(5, 2, 1, 1, 1)$ shown in Figure~\ref{f:ex_pi} is $s_0 s_1 s_2 s_3 s_2 s_1 s_0$.  The first step in reducing this expression to the identity removes the three boxes labeled 0 that lie at the end of their rows which is recorded as the leftmost $s_0$ in the expression.
\[
\tableau{\mbox{0} & \mbox{1} & \mbox{2} & \mbox{3} & \mbox{\bf 0} \\
         \mbox{3} & \mbox{\bf 0} \\
         \mbox{2} \\
         \mbox{1} \\
         \mbox{\bf 0} \\
} \stackrel{s_0}{\hspace{0.2in} \longleftarrow \hspace{0.2in}}
\tableau{\mbox{0} & \mbox{1} & \mbox{2} & \mbox{3} \\
         \mbox{3} \\
         \mbox{2} \\
         \mbox{\bf 1} \\
} \stackrel{s_1}{\hspace{0.2in} \longleftarrow \hspace{0.2in}}
\tableau{\mbox{0} & \mbox{1} & \mbox{2} & \mbox{3} \\
         \mbox{3} \\
         \mbox{\bf 2} \\
} 
\]
\[
\stackrel{s_2}{\hspace{0.2in} \longleftarrow \hspace{0.2in}}
\tableau{\mbox{0} & \mbox{1} & \mbox{2} & \mbox{\bf 3} \\
         \mbox{\bf 3} \\
} \stackrel{s_3}{\hspace{0.2in} \longleftarrow \hspace{0.2in}}
\tableau{\mbox{0} & \mbox{1} & \mbox{\bf 2} \\
} \stackrel{s_2}{\hspace{0.2in} \longleftarrow \hspace{0.2in}}
\tableau{\mbox{0} & \mbox{\bf 1} \\
} 
\]
\[
\stackrel{s_1}{\hspace{0.2in} \longleftarrow \hspace{0.2in}}
\tableau{\mbox{\bf 0} \\
} \stackrel{s_0}{\hspace{0.2in} \longleftarrow \hspace{0.2in}}
\emptyset.
\]
\end{example}

\begin{proposition}\label{p:redexp}
For $\lambda \in \mathcal{C}_{\ell}$ we have that $w(\lambda)$ is a reduced expression for the minimal length coset representative indexed by $\lambda$.

The Coxeter length of $w(\lambda)$ is 
\[ l(w(\lambda)) = \sum_{i = 0}^{\ell-1} \lambda_{R(i)} \]
where $R(i)$ is the longest row of $\lambda$ whose rightmost box has residue $i$.
\end{proposition}

\begin{proof}
Let $i$ be the residue of the rightmost box in the bottom row $\lambda_r$ of $\lambda$.  Then the balanced flush abacus configuration corresponding to $\lambda$ has an active bead $B$ representing $\lambda_r$ on runner $i$ by Theorem~\ref{t:flush_abacus}.  Observe that $B$ has a gap immediately preceding it in the reading order because $B$ is the first active bead in the reading order and $\lambda_r \neq 0$.  Since the abacus is flush, every box of $\lambda$ with residue $i$ that lies at the end of its row corresponds to some active bead $B'$ on runner $i$ of the abacus with a gap immediately preceding $B'$ in the reading order.  Hence, every box of $\lambda$ with residue $i$ that lies at the end of its row also lies at the end of its column and applying $s_i$ removes every such box by Proposition~\ref{p:action}.

Iterating this process eventually produces the empty partition, corresponding to the identity Coxeter element.  At each step, we have shown that applying $s_i$ removes exactly one box from $R(i)$ yielding the length formula.
\end{proof}


\begin{example} Let $\lambda = (10,7,4,3,2,2,2,1,1,1)$ and $\ell = 4$. Then the biggest part ending in residue 1 is $\lambda_1 = 10$, the biggest part ending in residue 3 is $\lambda_4 = 3$, the biggest part ending in residue 0 is $\lambda_6 = 2$, and no part ends in residue 2. Hence the Coxeter length of the minimal length coset representative for this 4-core is $10+3+2 = 15$. The canonical minimal length coset representative for this core is $w(\lambda) = s_3s_0s_1s_2s_3s_0s_1s_3s_2s_1s_0s_3s_2s_1s_0$.
$$\tableau{0&1&2&3&0&1&2&3&0&1\\
3&0&1&2&3&0&1\\
2&3&0&1\\
1&2&3\\
0&1\\
3&0\\
2&3\\
1\\
0\\
3}$$
\end{example}

There is another way to obtain the root vector $\pi^{-1}(\lambda) = (b_0, \ldots, b_{\ell-1}) \in \Lambda_R$ from the partition $\lambda \in \mathcal{C}_{\ell}$ due to Garvan, Kim and Stanton \cite{garvan-kim-stanton}.  Say that \em region $r$ \em of the diagram of $\lambda$ is the set of boxes $(i,j)$ satisfying $(r-1) \ell \leq j-i < r \ell$.  We call a box \em row-exposed \em if it lies at the end of its row.  Then, set $b_i$ to be the maximum region of $\lambda$ which contains a row-exposed box with residue $i$ for $0 \leq i \leq \ell-1$.  In particular, we pad $\lambda$ with parts of size zero if necessary and label all of the boxes before the $0^{\textrm{th}}$ column by their residue.  In this way $b_i$ is well-defined because column 0 contains infinitely many row-exposed boxes.  We call the vector $(b_0, \ldots, b_{\ell-1})$ obtained in this fashion the \em $n$-vector of $\lambda$ \em and we show that it is the same vector as $\pi^{-1}(\lambda)$.

\begin{example}
Let $\ell = 4$ and $\lambda = (6,3,1,1)$. From the picture below, we see that the $n$-vector for $\lambda$ is $(-1,2,0,-1)$. 
$$
\tableau{0&1&2&3&0&1\\
3&0&1\\
2\\
1}
\put (-118,6){3}
\put (-118,-12){2}
\put (-118,-30){1}
\put (-118,-48){0}
\put (-118,-66){3}
\put (-118,-84){2}
\put (-118,-102){1}
\put (-118,-120){0}
\put (-30,-50){Region 1}
\put (-75,-85){Region 0}
\put (-100,-138){Region -1}
\put (10,-20){Region 2}
\put (-108,-124){\line(0,1){76}}
\put (10,6){$\leftarrow$ First 1}
\put (-84,-30){$\leftarrow$ First 2}
\put (-170,-66){First 3 $\to$}
\put (-170,-120){First 0 $\to$}
\linethickness{1.5pt}
\put (-36,18){\line(-1,0){18}}
\put (-36,18){\line(0,-1){18}}
\put (-36,0){\line(1,0){18}}
\put (-18,0){\line(0,-1){18}}
\put (-18,-18){\line(1,0){18}}
\put (0,-18){\line(0,-1){18}}
\put (0,-36){\line(1,0){18}}
\put (18,-36){\line(0,-1){18}}
\put (18,-54){\line(1,0){18}}
\put (36,-54){\line(0,-1){18}}
\put (36,-72){\line(1,0){18}}
\put (54,-72){\line(0,-1){18}}
\put (54,-90){\line(1,0){18}}
\put (72,-90){\line(0,-1){18}}
\put (72,-108){\line(1,0){18}}
\put (90,-108){\line(0,-1){18}}
\put (90,-126){\line(1,0){18}}
\put (108,-126){\line(0,-1){18}}
\put (108,-144){\line(1,0){18}}
\put (-108,18){\line(-1,0){18}}
\put (-108,18){\line(0,-1){18}}
\put (-108,0){\line(1,0){18}}
\put (-90,0){\line(0,-1){18}}
\put (-90,-18){\line(1,0){18}}
\put (-72,-18){\line(0,-1){18}}
\put (-72,-36){\line(1,0){18}}
\put (-54,-36){\line(0,-1){18}}
\put (-54,-54){\line(1,0){18}}
\put (-36,-54){\line(0,-1){18}}
\put (-36,-72){\line(1,0){18}}
\put (-18,-72){\line(0,-1){18}}
\put (-18,-90){\line(1,0){18}}
\put (0,-90){\line(0,-1){18}}
\put (0,-108){\line(1,0){18}}
\put (18,-108){\line(0,-1){18}}
\put (18,-126){\line(1,0){18}}
\put (36,-126){\line(0,-1){18}}
\put (36,-144){\line(1,0){18}}
\put (-108,-54){\line(-1,0){18}}
\put (-108,-54){\line(0,-1){18}}
\put (-108,-72){\line(1,0){18}}
\put (-90,-72){\line(0,-1){18}}
\put (-90,-90){\line(1,0){18}}
\put (-72,-90){\line(0,-1){18}}
\put (-72,-108){\line(1,0){18}}
\put (-54,-108){\line(0,-1){18}}
\put (-54,-126){\line(1,0){18}}
\put (-36,-126){\line(0,-1){18}}
\put (-36,-144){\line(1,0){18}}
\put (-108,-126){\line(-1,0){18}}
\put (-108,-126){\line(0,-1){18}}
\put (-108,-144){\line(1,0){18}}
$$
\end{example}

\begin{lemma}\label{l:n_vector}
Let $(b_0, \ldots, b_{\ell-1})$ be the $n$-vector of an $\ell$-core $\lambda$ and let $s_i(\lambda)$ be the result of applying $s_i$ to $\lambda$ as described in Proposition~\ref{p:action}.  Then, the $n$-vector of $s_i(\lambda)$ is $(b_0, \dots, b_{i}, b_{i-1}, \dots b_{\ell-1})$ for $1 \leq i \leq \ell-1$, or $(b_{\ell-1} + 1, b_2, \dots, b_{\ell-2}, b_{0}-1)$ if $i = 0$.
\end{lemma}
\begin{proof}
Denote the $n$-vector of $s_i(\lambda)$ by $(n_0, \dots, n_{\ell-1})$.  It follows from the proof in \cite[Bijection 2]{garvan-kim-stanton} that if there exists a row-exposed $i$-box in region $r$ then there exist row-exposed $i$-boxes in each region $< r$.  

Observe that a row $\lambda_p$ of $\lambda$ has an addable $i$-box in region $r$ exactly if:
\begin{enumerate}
\item the row $\lambda_p$ has a row-exposed $(i-1)$-box in region $r - \delta_{i,0} \leq b_{i-1}$, and
\item the row $\lambda_{p-1}$ does not have a row-exposed $i$-box, so $b_i < r$.
\end{enumerate}

Similarly, a row $\lambda_q$ has a removable $i$-box in region $s$ exactly if:
\begin{enumerate}
\item the row $\lambda_q$ has a row-exposed $i$-box in region $s \leq b_i$, and
\item the row $\lambda_{q+1}$ does not have a row-exposed $(i-1)$-box, so $b_{i-1} < s - \delta_{i,0}$.
\end{enumerate}

Moreover, an $i$-box is addable (removable) in $\lambda$ only if it is removable (addable, respectively) in $s_i(\lambda)$.  Therefore, when we apply $s_i$ to $\lambda$ we interchange each of the regions $r$ and $s$ satisfying
$ b_i < r \leq b_{i-1} + \delta_{i,0} $
and
$ b_{i-1} + \delta_{i,0} < s \leq b_{i}. $

Hence, we obtain
$ n_i = b_{i-1} + \delta_{i,0} $
and
$ n_{i-1} = b_i - \delta_{i,0}. $
\end{proof}

\begin{corollary}
The $n$-vector of $\lambda$ is $\pi^{-1}(\lambda)$.
\end{corollary}
\begin{proof}
The $n$-vector of the empty partition is equal to $\pi^{-1}(\emptyset) = (0, \ldots, 0)$.  The result then follows by induction on the Coxeter length of $w(\lambda)$ by Lemma~\ref{l:n_vector}.
\end{proof}

\begin{proposition}\label{p:k_from_vector}
Suppose that $\pi(\mathbf{a}) = \pi(a_1, \ldots, a_{\ell}) = \lambda$.  Then we have 
\[ \lambda_1 = (a_i - 1) \ell + i \]
where $a_i$ is the rightmost occurrence of the largest coordinate in $\mathbf{a}$.  Also, 
\[ \lambda_1 = \sum_{j = 1}^{i-1} (a_i - a_j) + \sum_{j=i+1}^{\ell} (a_i - a_j - 1). \]
\end{proposition}
\begin{proof}
Consider the balanced flush abacus corresponding to $\lambda$.  Then $\lambda_1$ corresponds to the last active bead $B$ in the reading order, and $B$ lies on runner $i-1$.  In particular, if there are multiple occurrences of the largest coordinate in $\mathbf{a}$ then $\lambda_1$ corresponds to the bead on the rightmost runner.  The number of boxes in $\lambda_1$ is the number of gaps prior to $B$ in the reading order of the abacus and the second formula follows from counting these gaps, using that the beads are flush on each runner.

Since the abacus is balanced, we have that the number of beads strictly below the zero level must be equal to the number of gaps weakly above the zero level.  If the last active bead occurs in level $j$ of runner $i-1$ then we could move all of the beads below the zero level to fill in the gaps above the zero level, and so count the gaps starting from entry 0 to the entry that contained $B$ as $(j-1) \ell + i$.  This yields the first formula.
\end{proof}

\begin{example}
The balanced flush abacus corresponding to the 4-core \[ \pi(1, -2, 2, -1) = \lambda = (7, 4, 3, 2, 1, 1, 1) \] is shown below together with the diagram in which the beads have been moved to calculate $\lambda_1 = 7$ as $(2-1) 4 + 3$.  Here, $a_3 = 2$ is the largest entry, corresponding to runner 2.

\begin{tabular}{ccc}
\begin{picture}(105,110)
\put (42.5,96){.}
\put (42.5,92){.}
\put (42.5,100){.}
\put (72.5,96){.}
\put (72.5,92){.}
\put (72.5,100){.}
\put (102.5,96){.}
\put (102.5,92){.}
\put (102.5,100){.}
\put (11.5,96){.}
\put (11.5,92){.}
\put (11.5,100){.}
\put (8,80) {-8}
\put (38,80){-7}
\put (68,80){-6}
\put (98,80){-5}
\put (8,65) {-4}
\put (38,65){-3}
\put (68,65){-2}
\put (98,65){-1}

\put (10,50){0}
\put (40, 50){1}
\put (70,50){2}
\put (101,50){3}

\put (10,35){4}
\put (40,35){5}
\put (70,35){6}
\put (101,35){7}

\put (10,20){8}
\put (40,20){9}
\put (68,20){10}
\put (99,20){11}

\put (5,45.5){\line(90,0){105}}

\put (72.5,53){\circle{13}}
\put (72.5,68){\circle{13}}
\put (72.5,83){\circle{13}}

\put (102.5,83){\circle{13}}

\put (42.5,83){\circle{13}}

\put (12.5,38){\circle{13}}
\put (12.5,53){\circle{13}}
\put (12.5,68){\circle{13}}
\put (12.5,83){\circle{13}}

\put (72.5,38){\circle{13}}
\put (72.5,23){\circle{13}}
\put (101.5,68){\circle{13}}

\put (42.5,12){.}
\put (42.5,8){.}
\put (42.5,4){.}
\put (72.5,12){.}
\put (72.5,8){.}
\put (72.5,4){.}
\put (102.5,12){.}
\put (102.5,8){.}
\put (102.5,4){.}
\put (11.5,12){.}
\put (11.5,8){.}
\put (11.5,4){.}
\end{picture} & \hspace{0.4in} &
\begin{picture}(105,100)
\put (42.5,96){.}
\put (42.5,92){.}
\put (42.5,100){.}
\put (72.5,96){.}
\put (72.5,92){.}
\put (72.5,100){.}
\put (102.5,96){.}
\put (102.5,92){.}
\put (102.5,100){.}
\put (11.5,96){.}
\put (11.5,92){.}
\put (11.5,100){.}
\put (8,80) {-8}
\put (38,80){-7}
\put (68,80){-6}
\put (98,80){-5}
\put (8,65) {-4}
\put (38,65){-3}
\put (68,65){-2}
\put (98,65){-1}

\put (10,50){0}
\put (40, 50){1}
\put (70,50){2}
\put (101,50){3}

\put (10,35){4}
\put (40,35){5}
\put (70,35){6}
\put (101,35){7}

\put (10,20){8}
\put (40,20){9}
\put (68,20){\underline{\bf{10}}}
\put (99,20){11}

\put (72.5,53){\circle{13}}
\put (72.5,68){\circle{13}}
\put (72.5,83){\circle{13}}

\put (102.5,83){\circle{13}}

\put (42.5,53){\circle{13}}
\put (42.5,68){\circle{13}}
\put (42.5,83){\circle{13}}

\put (12.5,53){\circle{13}}
\put (12.5,68){\circle{13}}
\put (12.5,83){\circle{13}}

\put (102.5,68){\circle{13}}
\put (103.5,53){\circle{13}}

\put (42.5,12){.}
\put (42.5,8){.}
\put (42.5,4){.}
\put (72.5,12){.}
\put (72.5,8){.}
\put (72.5,4){.}
\put (102.5,12){.}
\put (102.5,8){.}
\put (102.5,4){.}
\put (11.5,12){.}
\put (11.5,8){.}
\put (11.5,4){.}
\end{picture} \\
\end{tabular}
\end{example}

\begin{corollary}\label{c:hyperplane}
For $k \geq 0$, let $H_{\ell}^{k}$ denote the affine hyperplane
\[ H_{\ell}^{k} = \{ \mathbf{a} = (a_1, \ldots, a_{\ell}) \in \mathbf{R}^{\ell} : (\mathbf{a}, \e_{(k\ \mathrm{mod}\ \ell)}) = \lceil {\frac{k}{\ell} } \rceil \} \cap V \]
inside $V$, where $1 \leq (k \mod \ell) \leq \ell$.
Then under the correspondence $\pi$, the $\ell$-cores $\lambda$ with $\lambda_1 = k$ all lie inside $H_{\ell}^{k} \bigcap \Lambda_R$.
\end{corollary}
\begin{proof}
We can write $k > 0$ uniquely as $(j - 1) \ell + i$ for $1 \leq i \leq \ell$ and $j \geq 1$.  In this case, $j = \lceil {\frac{k}{\ell}} \rceil$ and $i \equiv k \mod \ell$.  The result then follows from the first formula of Proposition~\ref{p:k_from_vector}.  If $\lambda_1 = 0$ then $k = 0$ and $\lambda = \emptyset$ so the statement holds.
\end{proof}

\section{The bijection as an affine linear isometry}\label{s:geometric_bijection}

\subsection{A geometric interpretation}

From Proposition~\ref{p:k_from_vector} we see that if $a_i$ is the rightmost occurrence of the largest coordinate in $\mathbf{a} = (a_1, \ldots, a_{\ell}) = \pi^{-1}(\lambda)$, then $\lambda_1 \equiv i \mod \ell$.  The next result describes $\Phi_{\ell}^{k}$ in terms of the root lattice coordinates.

\begin{theorem}\label{t:main}
Let $\psi_{\ell}$ be the affine map defined by 
$$\psi_{\ell} (a_1, \ldots, a_{\ell}) = (a_{\ell}+1, a_1, a_2, \ldots, a_{\ell-1}).$$  Then, 
\[ \pi^{-1} \circ \Phi_{\ell}^{k} \circ \pi(a_1, \ldots, a_{\ell})  = \psi_{\ell-1}^{a_i} (a_1, \ldots, \widehat{a_i}, \ldots, a_{\ell}) \]
where $a_i$ is the rightmost occurrence of the largest entry among $\{a_1, \ldots, a_{\ell}\}$ and the circumflex indicates omission.
\end{theorem}
\begin{proof}
Suppose $\pi(a_1, \ldots, a_{\ell}) = \lambda \in \mathcal{C}_{\ell}^k$.  Then $\lambda$ corresponds to a balanced flush abacus $A$ in which the first row of $\lambda$ corresponds to the last active bead $B$ in the reading order for the abacus, and $B$ occurs on runner $i-1$.  The bijection $\Phi_{\ell}^{k}(\lambda)$ is defined on $A$ by deleting the runner $i-1$ in order to obtain the abacus of an $(\ell-1)$-core.  Since the original abacus is balanced, when we remove runner $i-1$ from $A$ we are left with an abacus $A'$ in which the balance number is $-a_i$.  Applying $\psi_{\ell-1}$ corresponds to shifting all of the entries of $A'$ forward one entry in the reading order of $A'$, or equivalently adding $\mathbf{1}$ to the $\beta$-numbers for $\lambda$.  Hence, applying $\psi_{\ell-1}^{a_i}$ to $A'$ produces a balanced flush abacus for the same partition as $\Phi_{\ell}^{k}(\lambda)$.
\end{proof}

This result is illustrated in Figure~\ref{f:sl3}.  If we fix $i = k \mod \ell$ and $a_i = \lceil {\frac{k}{\ell}} \rceil$ then it follows directly from the definition of $\psi$ that the map $\psi_{\ell-1}^{a_i} (a_1, \ldots, \widehat{a_i}, \ldots, a_{\ell})$ can be factored as a translation in the root lattice of $A_{\ell-1}$ composed with a certain projection to the root system of $A_{\ell-2}$.  The translation sends the affine hyperplane $H_{\ell}^k$ to a linear hyperplane $V'$ in $V$ defined by $a_i = 0$.  The projection sends $V' \subset V \subset \mathbf{R}^{\ell}$ to $\mathbf{R}^{\ell-1}$ with the standard root system of $A_{\ell-2}$.  This projection depends on $k$ and we caution the reader that Figure~\ref{f:sl3} only illustrates the translation.  Example~\ref{e:coords} works out the factorization of $\psi_{\ell-1}^{a_i}(a_1, \ldots, \widehat{a_i}, \ldots, a_{\ell})$ in a particular case.  

\begin{example}\label{e:coords}
Let $\ell = 3$.  The affine hyperplane $H_{3}^7$ contains $(3, 1, -4) = \pi^{-1}(7, 5, 4^2, 3^2, 2^2, 1^2)$.  Translation by the vector $\mathbf{t} = (-3, 1, 2)$ sends $H_{3}^7$ to the linear hyperplane
\[ V' = \{ (a_1, a_2, a_3) \in V : a_1 = 0 \}. \]
We view this as a vector space $W$ with orthonormal basis $\{e_1', e_2'\}$ and an associated root system of type $A_{\ell-2}$.  The projection from $V'$ to $W$ identifies $e_1'$ with $e_3$ and $e_2'$ with $e_2$.  Hence, we have $\psi^{3}(1,-4) = (-2, 2)$ corresponding to $\Phi_{3}^7(7,5,4^2,3^2,2^2,1^2) = (4,3,2,1)$.
\end{example}

\begin{figure}[p!]  
\centering
\includegraphics{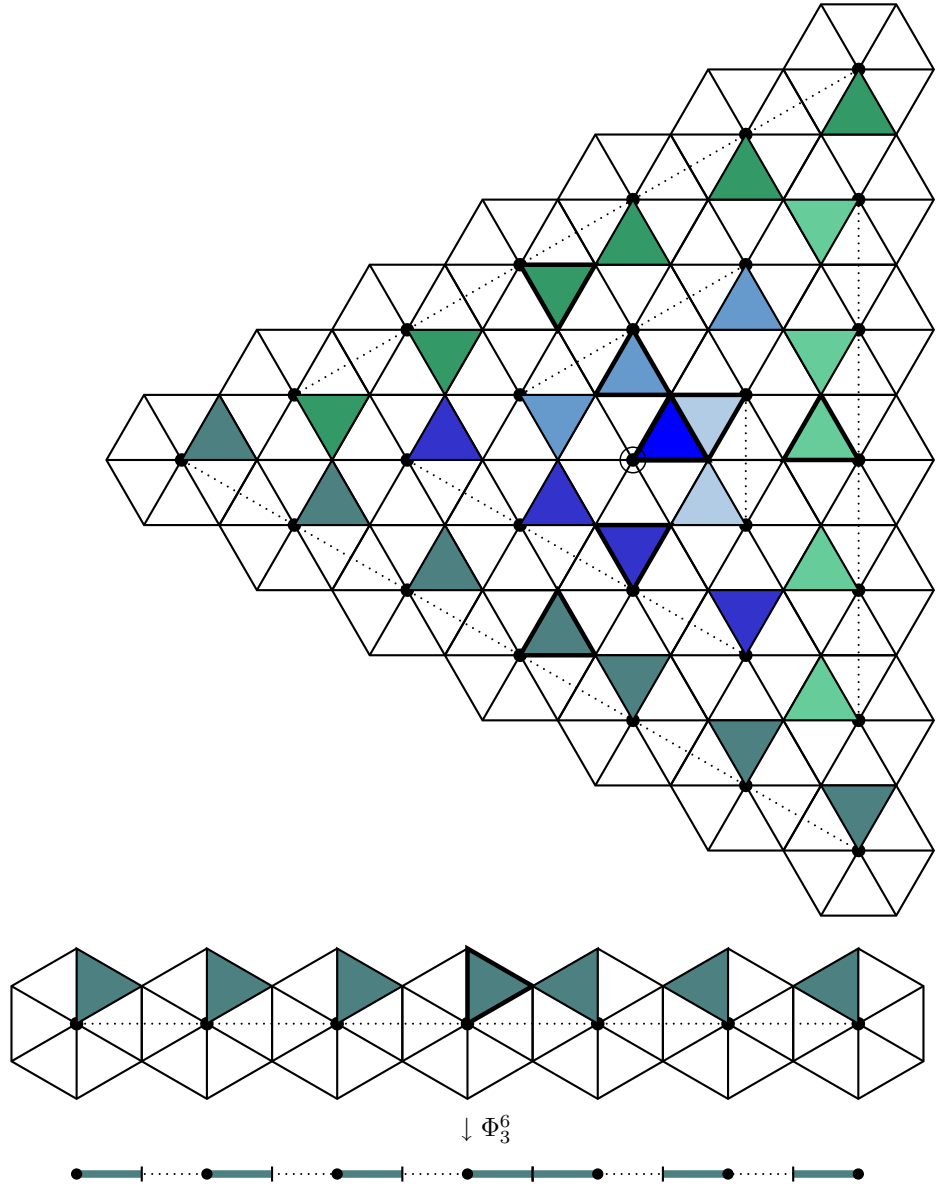}
\caption{\small The bijections $\Phi_{3}^{k}$ as a projection of $A_{2}^{(1)} \rightarrow A_{1}^{(1)}$ for $1 \leq k \leq 6$.  The $B_{\circ} + \mathbf{a}$ are hexagons which tile the plane centered at $\mathbf{a}$ which is darkened, and each $B_{\circ} + \mathbf{a}$ contains 6 triangular alcoves.  Each alcove corresponds to an element of $\widetilde{S_{\ell}}$ and each hexagon corresponds to a coset in $\widetilde{S_{\ell}} / S_{\ell}$, so there is a unique 3-core alcove in each hexagon.  These are shaded above.  Two alcoves $B_{\circ} + \mathbf{a}$ and $B_{\circ} + \mathbf{b}$ share the same color only if the first parts of the partitions $\pi(\mathbf{a})$ and $\pi(\mathbf{b})$ agree. }\label{f:sl3}
\end{figure}

\begin{remark}
If we focus on $len(\lambda)$ instead of $\lambda_1$, all of the $\ell$-cores with fixed length $m = len(\lambda)$ have $\pi^{-1}(\lambda)$ lying in the affine hyperplane 
\[ \{ \mathbf{a}  \in \mathbf{R}^{\ell} : (\mathbf{a}, \e_{((1-m)\ \mathrm{mod}\ \ell)}) = -\lceil { \frac{k}{\ell} } \rceil \} \cap V. \]
If we drew dotted lines in Figure~\ref{f:sl3} connecting the $\pi^{-1}(\lambda)$ with fixed $m = len(\lambda)$ (instead of those with fixed $k=\lambda_1$), then the lines would appear to spiral backwards from the direction of those in Figure~\ref{f:sl3}.  This can be explained by the fact that sending a partition to its transpose corresponds to the transformation of $\mathbf{R}^{\ell}$ given by $(a_1, \ldots, a_{\ell}) \mapsto (-a_{\ell}, \ldots,- a_1)$ as shown in \cite{garvan-kim-stanton}.
\end{remark}

\subsection{The bijection as a subexpression in Coxeter generators}

Recall from Section~\ref{s:phi_on_diagram} that $\Phi_{\ell}^k$ removes all rows from $\lambda$ in the same equivalence class as the first row.  In Proposition~\ref{p:redexp}, we described a canonical reduced expression $w(\lambda)$ for $\lambda = (\lambda_1, \ldots, \lambda_m) \in \mathcal{C}_{\ell}$.  In this construction, the rows of $\lambda$ are partitioned into $\ell$ equivalence classes which we denote by $[j]$ according to the residue $j$ of their rightmost box.  Let $i$ be the residue of the rightmost box $\mathsf{B}$ in the last row $\lambda_m$ of $\lambda$.  Observe that $\mathsf{B}$ is a removable $i$-box unless $\lambda = \emptyset$, and applying $s_i$ removes one box from each of the rows $\equiv [i]$.

We claim that two rows of $s_i (\lambda)$ are equivalent if and only if the rows were equivalent in $\lambda$.  To see this, consider that there can be no rows $\equiv [i-1]$ in $\lambda$.  Otherwise there exists a box in the same column as $\mathsf{B}$ whose row is $\equiv [i-1]$, and so the hooklength of this box is divisible by $\ell$ which contradicts $\lambda$ being an $\ell$-core.  Therefore, the rows $\equiv [i-1]$ in $s_i(\lambda)$ are precisely the rows $\equiv [i]$ in $\lambda$ with the possible exception of $\lambda_m$ if $\lambda_m = 1$.  In any case, no other rows change equivalence classes.

Suppose $w(\lambda) = s_{i_1} s_{i_2} \cdots s_{i_p}$ is the canonical reduced expression for $\lambda$ obtained from Proposition~\ref{p:redexp}, so $i_1$ is the residue of the box $\mathsf{B}$ and $i_p = 0$.  Working left to right to reduce $w(\lambda)$ to the identity, each application of $s_{i_j}$ removes a box from every row that is in the same equivalence class as the last row of the intermediate partition.  Let $J$ be the subset of $\{1, \ldots, p\}$ such that the first row and the last row are not in the same equivalence class in the $\ell$-core $s_{i_j} s_{i_{j+1}} \cdots s_{i_p}(\emptyset)$.  Then the subexpression of $w(\lambda)$ corresponding to the indices in $J$ gives the canonical minimal length coset representative for the $(\ell-1)$-core $\Phi_{\ell}^k(\lambda)$ after relabeling the residues with respect to $\ell-1$.

Since we remove a box from every row in the equivalence class of the last row at each step, we remove in particular a box from the longest row in that equivalence class.  Hence, we see that the positions $\{1, \ldots, p\} \setminus J$ that are deleted from $w(\lambda)$ in the application of $\Phi_{\ell}^k$ correspond with boxes in the first row of $\lambda$, so applying $\Phi_{\ell}^k$ reduces the Coxeter length by exactly $k$.

\begin{example}
Suppose $\ell = 5$ and let $\lambda=(9,5,3,2,2,1,1,1,1)$.  
We label the diagram of $\lambda$ as shown in Figure~\ref{f:subexp} by residues with respect to $\ell$ on the bottom of each box, and with respect to $\ell-1$ on the top of those boxes that are not in rows equivalent to the first row.  We put dots as placeholders in the tops of these boxes.

\begin{figure}[h]
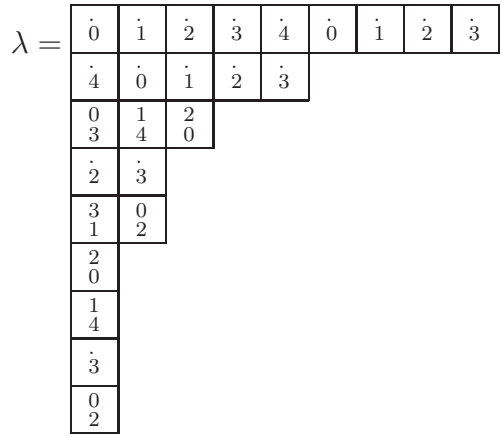

\[
\lambda = \tableau{ { }^{.}_{0} &  { }^{.}_{1} &  { }^{.}_{2} &  { }^{.}_{3} &  { }^{.}_{4} &  { }^{.}_{0} &  { }^{.}_{1} &  { }^{.}_{2} &  { }^{.}_{3} \\
 { }^{.}_{4} &  { }^{.}_{0} &  { }^{.}_{1} &  { }^{.}_{2} &  { }^{.}_{3} \\
 { }^{0}_{3} &  { }^{1}_{4} &  { }^{2}_{0} \\
 { }^{.}_{2} &  { }^{.}_{3} \\
 { }^{3}_{1} &  { }^{0}_{2} \\
 { }^{2}_{0} \\
 { }^{1}_{4} \\
 { }^{.}_{3} \\
 { }^{0}_{2} }
\]
\caption{$\Phi_{\ell}^k$ as a Coxeter subexpression of $w = {\bf s_2} s_3 {\bf s_4 s_0 s_1} s_2 {\bf s_4 s_3 } s_1 s_0 s_4 s_3 s_2 s_1 s_0$.} \label{f:subexp}
\end{figure}

We see directly from this diagram that 
\[ w(\lambda) = w = {\bf s_2} s_3 {\bf s_4 s_0 s_1} s_2 {\bf s_4 s_3 } s_1 s_0 s_4 s_3 s_2 s_1 s_0 \]
where the entries in $J$ have been highlighted.  For example, applying $s_2$ to $\lambda$ removes the last row as well as the last box from row 5.  Since we did not remove a box from the first row, we have position $1 \in J$.  The next step applies $s_3$ to remove the last row of $s_2(\lambda)$ as well as the last box from rows 1, 2, and 4, so position $2 \notin J$.

After shifting the residues in the bold subexpression $s_2 s_4 s_0 s_1 s_4 s_3$ to be calculated relative to $\ell-1$ as shown in the diagram, we obtain $w(\Phi_{\ell}^{k}(\lambda)) = s_0 s_1 s_2 s_3 s_1 s_0$.  The Coxeter length has been reduced by $15 - 6 = 9 = \lambda_1$.
\end{example}

\section{Relation with a correspondence of Lapointe--Morse}\label{s:l-m}

\subsection{Interpretation of the bijection in terms of the Lapointe-Morse correspondence}
In this section, we show that the bijection $\Phi_{\ell}^k$ can be succinctly expressed as a map between $k$-bounded partitions (one whose first part is $\leq k$) and $(k-1)$-bounded partitions, using the correspondence
\[ \rho_{k+1} : \{ (k+1)\text{-cores} \} \rightarrow \{ \text{partitions with first part } \leq k \} \] 
of Lapointe and Morse \cite[Section 3]{LM} to which we refer the reader for more details.  Let $\widetilde{\Phi_{\ell}^k}$ be defined by the property that the left square in the following diagram commutes, and define $\Upsilon_{\ell}^k$ to simply delete the first column from the diagram of the partition.  Then, we claim that the right square in the diagram also commutes.  Here, $tr$ denotes the transpose of a partition.

\medskip

\[
\begin{CD}
\left\{ \parbox{0.55in}{$\lambda \in \mathcal{C}_{\ell}^k$} \right\} @> tr >> \left\{ \parbox{1.3in}{$\lambda \in \mathcal{C}_{\ell}$, $len(\lambda) = k$} \right\} @> \rho_{\ell} >> \left\{ \parbox{1.2in}{partitions $\nu$ with $\nu_1 \leq \ell-1$ and $len(\nu) = k$} \right\} \\
@ V{\Phi_{\ell}^k} VV @ VV {\widetilde{\Phi_{\ell}^k}} V @ VV {\Upsilon_{\ell}^k} V \\
\left\{ \parbox{0.65in}{$\mu \in \mathcal{C}_{\ell-1}^{\leq k}$} \right\} @> tr >> \left\{ \parbox{1.45in}{$\mu \in \mathcal{C}_{\ell-1}$, $len(\mu) \leq k$} \right\} @>> \rho_{\ell-1} > \left\{ \parbox{1.2in}{partitions $\sigma$ with $\sigma_1 \leq \ell-2$ and $len(\sigma) \leq k$} \right\} \\
\end{CD}
\]

\bigskip

Recall that $h_{\mathsf{B}}^{\lambda}$ denotes the hooklength of a box $\mathsf{B}$ in a partition diagram $\lambda$.  Let $\lambda = (\lambda_1, \ldots, \lambda_k)$ be an $\ell$-core with first column length $k$.  Observe that $\widetilde{\Phi_{\ell}^k}$ acts by removing entire \em columns \em from $\lambda$ by transposing the explanation in Section~\ref{s:phi_on_diagram}.  Row-wise, we can describe $\widetilde{\Phi_{\ell}^k}$ as removing the leftmost box $\mathsf{B}$ in the row together with all of the boxes $\mathsf{B}'$ in the row having $h_{\mathsf{B}}^{\lambda} \equiv h_{\mathsf{B}'}^{\lambda} \mod \ell$.

As in \cite{LM}, we view $\rho_{\ell}(\lambda)$ as the result of left-justifying all of the rows in a skew-diagram $\lambda / \gamma$, where $\gamma$ consists of the boxes of $\lambda$ having hooklength $> \ell$.  We say that the boxes of $\lambda$ lying in $\lambda / \gamma$ are the \em skew boxes \em while the boxes of $\gamma$ are the \em non-skew boxes\em.  

\begin{example}
Suppose $\ell = 5$ and consider the $\ell$-core $\lambda = (6,4,3,3,2,1,1)$.  In the diagrams below, we have labeled the boxes by their hooklengths.  The skew-boxes of $\gamma \subset \lambda$ are indicated in boldface, so $\rho_{\ell}(\lambda) = (3,2,2,2,2,1,1)$.  The entries that are deleted in the application of $\widetilde{\Phi_{\ell}^k}$ are indicated with underline.  
\[
\lambda = \tableau{ \underline{12} & 9 & \underline{7} & {\bf 4} & \underline{\bf 2} & {\bf 1} \\
\underline{9} & 6 & \underline{\bf 4} & {\bf 1} \\
\underline{7} & {\bf 4} & \underline{\bf 2} \\
\underline{6} & {\bf 3} & \underline{\bf 1} \\
\underline{\bf 4} & {\bf 1} \\
\underline{\bf 2} \\
\underline{\bf 1} \\
}
\ \ \hspace{0.4in} \ \ \
\widetilde{\Phi_{\ell}^k}(\lambda) = \tableau{ 7 & {\bf 3} & {\bf 1} \\
5 & {\bf 1} \\
{\bf 3} \\
{\bf 2} \\
{\bf 1} \\
}
\]
From these diagrams, we see that $\widetilde{\Phi_{\ell}^k}(\lambda)$ is a 4-core and $\Upsilon_{\ell}^k(\rho_{\ell}(\lambda)) = (2,1,1,1,1) = \rho_{\ell-1}(\widetilde{\Phi_{\ell}^k}(\lambda))$.
\end{example}

To simplify notation, let $\widetilde{\lambda} = \widetilde{\Phi_{\ell}^k}(\lambda)$ and define $\widetilde{\gamma}$ to consist of the boxes of $\widetilde{\lambda}$ having hooklength $> \ell-1$ so that $\rho_{\ell-1}(\widetilde{\lambda})$ is the result of left justifying the boxes of $\widetilde{\lambda} / \widetilde{\gamma}$.  

\begin{lemma}\label{l:lm-one-rem}
There is exactly one skew box deleted from each row of $\lambda$ in the application of $\widetilde{\Phi_{\ell}^k}$.  Also, a box $\mathsf{B}$ in $\lambda$ that is not deleted in the application of $\widetilde{\Phi_{\ell}^k}$ is skew with respect to $\ell$ if and only if the corresponding box $\widetilde{\mathsf{B}}$ of $\widetilde{\Phi_{\ell}^k}(\lambda)$ is a skew box with respect to $\ell-1$.
\end{lemma}

\begin{proof}
It suffices to prove these statements for a fixed row.  Since the skew boxes are those with hooklength $< \ell$, we delete at most one skew box from each row of $\lambda$ when we apply $\widetilde{\Phi_{\ell}^k}$.  Next, we show that at least one skew box is deleted.  Let $\mathsf{L}$ be the leftmost box of $\lambda$ in row $i$.  Since $\lambda$ is an $\ell$-core, the partition $\widehat{\lambda} = (\lambda_i, \lambda_{i+1}, \ldots, \lambda_k)$ is also an $\ell$-core.  Form an unbalanced abacus $A$ of $\widehat{\lambda}^{tr}$ such that the beads of $A$ correspond to hooklengths of the first \em row \em of $\widehat{\lambda}$ which is the $i^{\textrm{th}}$ row of $\lambda$.  In particular, those beads corresponding to the boxes with equivalent hooklength to $h_{\mathsf{L}}^{\lambda}$ form the rightmost longest runner of the abacus $A$ and they are flush.  Hence, there exists a box $\mathsf{L}'$ in the same row as $\mathsf{L}$ having hooklength $\equiv h_{\mathsf{L}}^{\widehat{\lambda}} = h_{\mathsf{L}}^{\lambda} \mod \ell$ such that $1 \leq h_{\mathsf{L}'}^{\lambda} \leq \ell-1$, so the corresponding box $\mathsf{L}'$ is a skew box.  The box $\mathsf{L}'$ is deleted from $\lambda$ when we apply $\widetilde{\Phi_{\ell}^k}$.

To prove the second statement, suppose $\mathsf{B}$ is a box in the $i$th row of $\lambda$ that does not get deleted in $\widetilde{\lambda}$.  Let $\widetilde{\mathsf{B}}$ be the corresponding box in $\widetilde{\lambda}$.  Then $\mathsf{B}$ corresponds to an active bead in the abacus $A$ that is not on same runner as $\mathsf{L}$.  We have that $\mathsf{B}$ is skew with respect to $\ell$ if and only if the corresponding bead lies on level 0 of the abacus $A$.  Since removing the runner containing $\mathsf{L}$ does not change any levels of the remaining beads, we have that $\widetilde{\mathsf{B}}$ is skew with respect to $\ell-1$ if and only if $\mathsf{B}$ is skew with respect to $\ell$.
\end{proof}

\begin{theorem}
The map $\Upsilon_{\ell}^k$ is a bijection which makes the diagram at the beginning of this section commute.
\end{theorem}
\begin{proof}
We see from Lemma~\ref{l:lm-one-rem} that the notion of skew box is preserved under the application of $\widetilde{\Phi_{\ell}^k}$.  Thus, we have that $\rho_{\ell-1}(\widetilde{\Phi_{\ell}^k}(\lambda))$ is formed from $\rho_{\ell}(\lambda)$ by simply deleting the first column.
\end{proof}

     \newchapter{Conclusion}{Conclusion}{Conclusion}
    \label{conclusion}


We will end by mentioning some open problems related to this thesis. Cossey, Ondrus and 
Vinroot (see \cite{COV}) have a construction for the case of the symmetric group in characteristic 
$p$ which is an analog of our decomposition of $\ell$-partitions from Section \ref{construct}. 
Their theorem should be the starting point for any attempt at generalizing our results to this case.


One possible direction to continue in would be to give a necessary and sufficient condition on a weak $\ell$-partition which determines when $\widetilde{f}_i^{\varphi-1} \lambda$ and $\widetilde{e}_i^{\varepsilon-1} \lambda$ are weak $\ell$-partitions.

We conjecture that similar arguments could be found to prove Theorems 
\ref{top_and_bottom_weak} and \ref{other_cases_weak} in the 
more general setting of $(\ell,p)$-JM partitions. We gave a representation theoretic proof 
of Theorem \ref{top_and_bottom_weak}. This proof is
 extendable to prove a version of Theorem \ref{top_and_bottom_weak} for $(\ell,p)$-JM partitions.
 It is fairly trivial to see that an $(\ell,p)$-JM partition is an $(\ell,0)$-JM partition. 
Hence our crystal contains all irreducibles for $H_n(q)$ where $q$ is an $\ell^{th}$ root of unity
 over a field of characteristic $p$.

\section{Decomposition Numbers}\label{decomposition_numbers}
\subsection{Definitions}

Let $\lambda$ be a partition of $n$ and let $\mu$ be an $\ell$-regular partition of $n$. Then we 
recall the definition of the decomposition number $d_{\lambda,\mu} := [S^{\lambda}:D^{\mu}]$.

\subsection{Open Problem}
Just as Dipper and James constructed all of the irreducibles for $H_n(q)$ as quotients of the Specht
 modules $S^{\lambda}$ with $\lambda$ $\ell$-regular, one of our main goals is now to give a 
construction of the irreducibles of $H_n(q)$ indexed by the nodes of 
$B(\Lambda_0)^L$. In other words, we are currently looking for a construction of the irreducibles
 indexed by the smallest partition in a regularization class, whereas they are usually indexed 
by the largest partition in a regularization class. One obvious benefit of such a construction
 would be that all of the irreducibles of $H_n(q)$ which are isomorphic to a Specht module are 
instantly recognized. For example, if $\ell = 3$ and $\lambda = (2,1)$ then James' construction says
 to look for $D^{(2,1)}$ as a quotient of $S^{(2,1)}$, whereas our crystal has the node $(1,1,1)$
 which is an $(\ell,0)$-JM partition and hence $S^{(1,1,1)} = D^{(2,1)}$. Giving such a 
construction should lead to a nicer description of the irreducibles since all irreducible Specht
 modules will be used. Also, based on computational evidence, we make the following conjecture.
\begin{conjecture}
Let $\lambda$ be a partition. Then $\dim S^{\mathcal{S}\lambda} \leq \dim S^{\mathcal{R} \lambda}$.
\end{conjecture}

\begin{example}
Let $\ell = 3$ and let $\lambda = (3,2,1)$. Then $\lambda \in B(\Lambda_0)$, and $\mu = (2,1,1,1,1) \in B(\Lambda_0)^L$ with $\mathcal{R}\mu = \lambda$. The dimension of $S^{\lambda}$ is 16 and the dimension of $S^{\mu}$ is 5. Both have a copy of $D^{\lambda}$ (by Proposition \ref{decomp}), which is of dimension 4. From the viewpoint of trying to decompose the smallest module possible, $D^{\lambda}$ is easier to find in $S^{\mu}$ than it is in $S^{\lambda}$.
\end{example}

    %
    %

    %
    %

\end{document}